\documentclass[numbers,webpdf,imaiai]{ima-authoring-template}%

\graphicspath{{figures/}}

\theoremstyle{thmstyletwo}%
\newtheorem{theorem}{Theorem}[section]
\newtheorem{proposition}[theorem]{Proposition}%
\newtheorem{assumption}[theorem]{Assumption}%
\newtheorem{corollary}[theorem]{Corollary}
\newtheorem{definition}[theorem]{Definition}
\newtheorem{notation}[theorem]{Notation}

\newtheorem{lemma}[theorem]{Lemma}

\newtheorem{remark}[theorem]{Remark}

\numberwithin{equation}{section}

\newcommand{\scT}{\mathcal{T}}

\newcommand{\scS}{\mathcal{S}}

\newcommand{\txtb}{\textup{\texttt{b}}}
\newcommand{\txtc}{\textup{\texttt{c}}}

\newcommand{\wV}{\widetilde{V}}
\newcommand{\whV}{\widehat{V}}
\newcommand{\whv}{\widehat{v}}
\newcommand{\whw}{\widehat{w}}
\newcommand{\whp}{\widehat{p}}
\newcommand{\bfmin}{\mathbf{min}\,}

\newcommand{\shapereg}{C_{\textup{\textsf{SR}}}}
\newcommand{\CPOU}{C_{\textup{\textsf{POU}}}}

\usepackage{bm}

\usepackage[show]{ed}   
\usepackage{soul}
\definecolor{darkgreen}{rgb}{0.,0.7,0.}

\newcommand{\Ceig}{C_{\textup{\textsf{eig}}}} 
\newcommand{\Cinv}{C_{\textup{\textsf{inv}}}}
\newcommand{\Cinvt}{C^2_{\textup{\textsf{inv}}}}
\newcommand{\Cpos}{C_{\textup{\textsf{pos}}}}
\newcommand{\Cstar}{C_{\ast}}

\newcommand{\assign}{\mathrel{\mathop:}=}

\newcommand{\tV}{\widetilde{V}}
\numberwithin{equation}{section}

\DeclareMathOperator{\diverg}{div}
\DeclareMathOperator{\supp}{supp}
\DeclareMathOperator{\diam}{diam}

\DeclareMathOperator{\linspan}{span}
\DeclareMathOperator{\dof}{dof}
\newcommand{\ovdof}{\overline{\dof}}
\newcommand{\rd}{\mathrm{d}}

\newcommand{\Cstab}{C_{\textup{\textsf{stab}}}}
\newcommand{\amin}{a_{\textup{\textsf{min}}}}
\newcommand{\amax}{a_{\textup{\textsf{max}}}}
\newcommand{\nmin}{n_{\textup{\textsf{min}}}}
\newcommand{\nmax}{n_{\textup{\textsf{max}}}}
\renewcommand{\kappa}{k}

\newcounter{counter_a}
\newenvironment{myenum}{\begin{list}{\textrm{\textup{(\roman{counter_a})}}\;}%
{\usecounter{counter_a}
\setlength{\itemsep}{0.5ex}\setlength{\topsep}{0.7ex}
\setlength{\leftmargin}{5ex}\setlength{\labelwidth}{6ex}}}{\end{list}}

\newenvironment{mybull}{\begin{list}{$\bullet$\;}%
{\setlength{\itemsep}{0.5ex}\setlength{\topsep}{0.7ex}
\setlength{\leftmargin}{4ex}\setlength{\labelwidth}{6ex}}}{\end{list}}

\usepackage{ulem}

\begin{document}

\DOI{DOI HERE}
\copyrightyear{2026}
\vol{00}
\pubyear{2026}
\access{Advance Access Publication Date: Day Month Year}
\appnotes{Paper}
\copyrightstatement{Published by Oxford University Press on behalf of the Institute of Mathematics and its Applications. All rights reserved.}
\firstpage{1}

\title[Coarse spaces based on indefinite operators]{Spectral coarse spaces based on indefinite operators: the $H_k$-GenEO method}

\author{Théophile Chaumont-Frelet \ORCID{0000-0002-6210-0774}
\address{\orgname{Inria Univ. Lille and Laboratoire Paul Painlevé}, \orgaddress{ \postcode{59655},  \state{Villeneuve-d’Ascq},\country{France}}}}

\author{Victorita Dolean* \ORCID{0000-0002-5885-1903}
\address{\orgdiv{Department Mathematics and Computer Science}, \orgname{Eindhoven University of Technology}, \orgaddress{\state{Eindhoven}, \postcode{5612 AE}, \country{The Netherlands}}}}

\author{Mark Fry \ORCID{0009-0003-0543-8714}
\address{\orgdiv{Department of Mathematics and Statistics}, \orgname{University of Strathclyde}, \orgaddress{\state{Glasgow}, \postcode{G1 1XH}, \country{UK}}}}

\author{Ivan G. Graham \ORCID{0000-0002-5730-676X}
\address{\orgdiv{Department of Mathematical Sciences}, \orgname{University of Bath}, \orgaddress{\state{Bath}, \postcode{BA2 7AY}, \country{United Kingdom}}}}

\author{Matthias Langer \ORCID{0000-0001-8813-7914}
\address{\orgdiv{Department of Mathematics and Statistics}, \orgname{University of Strathclyde}, \orgaddress{\state{Glasgow}, \postcode{G1 1XH}, \country{UK}}}}

\authormark{Chaumont-Frelet et al.}

\corresp[*]{Corresponding author: \href{v.dolean.maini@tue.nl}{v.dolean.maini@tue.nl}}

\received{Date}{0}{Year}
\revised{Date}{0}{Year}
\accepted{Date}{0}{Year}

\abstract{
GenEO (`Generalised Eigenvalue problems on the Overlap') is a method for constructing coarse spaces used in the preconditioning 
of iterative solvers for discrete PDEs. This method combines  a (small) number of modes of local PDE eigenproblems 
to obtain a global coarse space.  A coarse solve is then combined with local solves of the global PDE to obtain the preconditioner. 
A substantial theory for GenEO has been developed for the case when the local elgenproblems are positive semi-definite. 
This has been applied mostly to  positive definite global PDEs, but also recently extended to the case of 
convection--diffusion--reaction problems, which may be neither self-adjoint, nor positive definite. 
However, when the global problem is highly indefinite, coarse spaces built from positive semi-definite local eigenproblems 
fail to be robust in practice.
In this paper we consider highly indefinite global PDE problems, characterised by a large parameter $k$
(allowing also highly variable coefficients), and we develop a new spectral coarse space built 
from solving eigenvalue problems based on \textit{local copies of the global problem}. 
We put no constraint on the diameters of the local domains, thus allowing the local eigenvalue problems to be indefinite.
The new method (which we call $H_k$-GenEO) is seen to be much more robust as $k$ increases than methods based on positive semi-definite eigenproblems.
We provide  sufficient conditions for robustness of the preconditioned GMRES iterative method, in terms of the
tolerance of the local eigenproblems and the size of the subdomains for the local PDE solves. 
In practice the method is observed to be robust with respect to $k$ under even weaker conditions on the local eigenproblem tolerance. 
The experiments also suggest the method can be resilient to high variation in PDE coefficients.}
\keywords{elliptic PDE; finite element method; domain decomposition; preconditioning; GMRES;
indefinite problem; spectral coarse space; robustness; Helmholtz problem.}

\maketitle

\section{Introduction}
This paper is concerned with the theory and implementation of two-level domain decomposition
preconditioned iterative methods for solving finite element approximations of the indefinite boundary value problem: 
\begin{subequations}
\label{eq:problem}
	\begin{alignat}{2}
		-\nabla \cdot (A \nabla u) - k^2 n u &= f \qquad && \text{in } \Omega, 
		\label{pde} \\
		u &= 0 \qquad && \text{on } \partial\Omega,
		\label{bc}
	\end{alignat}
\end{subequations}
where $k>0$, $A$ is an almost everywhere positive definite matrix with $L^\infty$ entries and $n$ is 
an almost everywhere positive $L^\infty$ function. 
The  domain $\Omega \subseteq \mathbb{R}^{d}$ ($d=2,3$) is assumed to be a Lipschitz polygon ($d=2$) or polyhedron ($d=3$). 
We emphasise the highly indefinite case---when $k$ is large---since this presents particular challenges for iterative solvers. 
While \eqref{pde} is instantly recognisable as a (heterogeneous) Helmholtz equation, the appearance of the Dirichlet boundary condition \eqref{bc} 
means that  \eqref{eq:problem} does not directly model scattering problems (where \eqref{pde} would normally be posed on an infinite domain, 
accompanied by a radiation condition instead of \eqref{bc}). 
However, \eqref{eq:problem} does appear in other physical models of wave phenomena (e.g.\ a vibrating membrane constrained at its boundary) 
and its solution typically oscillates with decreasing period as $k \rightarrow \infty$. 
Thus \eqref{eq:problem} provides us with a relevant highly indefinite model problem with variable coefficients, 
and we use it here to assess the performance of a new class of iterative solvers for such problems. 
We will assume throughout that the `wavenumber' $k$ is such that the problem \eqref{pde}, \eqref{bc} is well posed,
in the sense that it has a unique weak solution for all $f \in L^2(\Omega)$. 

The problem \eqref{eq:problem} is discretised using conforming finite elements to obtain a linear system (see \eqref{eq: 2_12}). 
When this system is very large (especially in 3D), it should be solved iteratively using a good preconditioner 
(i.e.\ a cheap approximate inverse).  
Here we propose a novel  domain decomposition preconditioner and analyse its performance when used with GMRES as the iterative solver. 
In classical one-level overlapping domain decomposition methods, the global domain $\Omega$ is covered by a set of 
overlapping `local' subdomains $\Omega_j^\ell$, $j=1,\ldots,Q$,
and the one-level Schwarz preconditioner is built from partial solutions of the global problem on each subdomain.
Since this is, in general, not scalable as the number of subdomains grows, an additional global coarse solve is usually added to enhance scalability, 
as well as robustness with respect to coefficient heterogeneity or increasing $k$.

\smallskip
\noindent
\textbf{Spectral coarse spaces.}
Since classical coarse spaces (piecewise polynomials on a coarse grid), are generally not robust for highly heterogeneous problems, various  
operator-dependent coarse spaces have been  introduced to better treat the heterogeneity,  
e.g.\ `spectral' coarse spaces, built from well-chosen modes of suitable generalised PDE eigenvalue problems
defined on subdomains.  Such spectral information was used  for both robust solution approximation  (e.g.\    
\cite{Babuska:2011:OMSGFEM,Efendiev:2013:GMF}) and for  preconditioning in \cite{Galvisa:2010:DDP,Galvisb:2010:DDP}. 
However, all this work was for positive definite (albeit highly heterogeneous) PDEs.

In this paper we propose and analyse a novel variant of the GenEO spectral coarse space.
GenEO coarse spaces were first introduced in \cite{Spillane:2014:ARC}, where it was proved that, 
for self-adjoint positive definite problems, when combined with a one-level method, and using a preconditioned conjugate gradient solver,  
the resulting algorithm is not only scalable with respect to the number of subdomains but also robust with respect to coefficient variation.   
Related works for positive definite problems include, 
for example \cite{Agullo:2019:RPG, Heinlein:2019:AGC,Bastian:2021:MSD,Spillane:2021:TFA}; 
see also \cite{Spillane:2014:ARC} and \cite{Galvis:2018:ODD}. {Here we construct coarse spaces  for \eqref{eq:problem}
from eigenproblems defined on each subdomain of an  overlapping `coarse' cover:  $\Omega_i^c$, $i=1,\ldots,N$.
This \textit{need not be related} to the `local' cover  $\Omega_j^\ell$ introduced above;
indeed the $\Omega_i^c$ can be  relatively large.
To explain ideas simply, now let $\xi_i^c$,  $i=1,\ldots,N$  be a partition of unity subordinate to the coarse cover.
The original GenEO method was aimed at  problem \eqref{eq:problem} with $k = 0$, and built a coarse space  
from selected modes of the generalised eigenvalue problem (GEVP):
\begin{equation} \label{new1} 
	-\nabla \cdot (A \nabla p) =  \lambda \bigl[-\xi_i \nabla \cdot (A \nabla (\xi_i  p))\bigr] \qquad \text{on} \;\; \Omega_i^c,
\end{equation} 
(computed using standard variational finite element approximation).  The eigenmodes of \eqref{new1} provide
a powerful coarse space for \eqref{eq:problem} when $A$ is highly heterogeneous and $k = 0$ \cite{Spillane:2014:ARC}.
However, they are less suitable when $k$ is very large, as demonstrated numerically in
\cite{Bootland:2021:ACS,Bootland:2022:GCS} where the coarse space based on \eqref{new1} (and called  $\Delta$-GenEO) 
was compared to a coarse space (called $H$-GenEO), based on eigenmodes of the GEVP:
\begin{equation} \label{new2} 
	-\nabla \cdot (A \nabla p) - k^2 n p  =  \lambda \bigl[-\xi_i \nabla \cdot (A \nabla (\xi_i  p))\bigr] \qquad \text{on} \;\; \Omega_i^c. 
\end{equation}
The indefinite operator from \eqref{pde} now appears on the left-hand side of \eqref{new2}, ensuring
that some eigenmodes of this problem will be oscillatory (provided the $\Omega_i^c$ are large enough). 
In \cite{Bootland:2021:ACS,Bootland:2022:GCS} $H$-GenEO was found to have better robustness properties than $\Delta$-GenEO 
at high frequency but no theory was provided. 
In the current paper we propose and  analyse the `$H_\kappa$-GenEO' coarse space based on the following variant of \eqref{new2}:
\begin{align} \label{new3} 
	-\nabla \cdot (A \nabla p) - k^2 n p  = \lambda \bigl[-\xi_i \nabla \cdot (A \nabla (\xi_i  p)) + k^2 \xi_i n (\xi_i p) \bigr]  
	\qquad \text{on} \;\; \Omega_i^c. 
\end{align}
The same indefinite operator appears on the left-hand side, while the right-hand side has the extra positive term 
making it close to the norm in which we do the analysis of this paper. 
Our actual analysis is for the finite element approximation of this---where the multiplication $\xi_i p$ 
is replaced by an algebraically defined operator $\Xi_i(p)$; see \eqref{eq: 5_12}.       

\vspace{0.1cm}
\noindent
\textbf{Main results.}
Our main results are  given in Theorem~\ref{theorem: convergence} and Corollary~\ref{conditions_corollary} 
and we give a simplified summary of these here.  
Let $H_\ell$ (respectively $H_c$) denote the maximum diameter of the `local' subdomains $\Omega_j^\ell$ 
(respectively `coarse' subdomains $\Omega_i^c$).  
On each $\Omega_i^c$ let $\{\lambda_m^i: m = 1,2,\ldots\}$ denote the eigenvalues of the numerical approximation of
\eqref{new3} (real and ordered in non-decreasing order). 
Then, if only the first $m_i$ eigenfunctions on each $\Omega_i^c$ are used in the coarse space, and $\lambda_{m_i +1 }^i > 0$, 
we define   
\begin{equation}
	\tau \assign \min_{i=1}^N \lambda^i_{m_i+1}
	\label{def:tau}
\end{equation}
to be the \textit{eigenvalue tolerance}.
Let $\Cstab > 0$ denote the stability constant arising from the assumption that \eqref{eq:problem} has a unique weak solution 
in $H^1_0(\Omega)$ for each $f \in L^2(\Omega)$  (see  Assumption~\ref{Ass: 2_3}).   
Without loss of generality, we assume that, on $\Omega$, the eigenvalues of the coefficient matrix $A$ are all $\ge 1$, 
and all values of $n$ are $\leq 1$ (Assumption~\ref{ass: 2_1}). 
Then, when the 2-level additive Schwarz preconditioned system is solved by GMRES, our theory proves 
its robust convergence under certain conditions. 
For this illustration, suppose the finite element mesh is quasiuniform with mesh diameter satisfying $h \sim k^{-1}$. 
Then, $k$-robustness holds provided
\begin{equation}
\label{Hk_results}
    H_\ell \lesssim k^{-1} \qquad \text{and} \qquad  \tau \gtrsim  (1+\Cstab)^{2}\, k^{2} ,
\end{equation}
where the hidden constants depend on the number of times any subdomain is overlapped by others. 
The rate of GMRES convergence depends modestly on $\amax$ (the maximum eigenvalue of $A$) and $\nmin$ (the minimum value of $n$), 
although this dependence is minimal in practice. Our results extend to more general meshes, with stricter conditions for robustness.

The first condition in \eqref{Hk_results} arises because $H_\ell$ is required to be small enough to guarantee the solvability of the local problems.  
\textit{Importantly, our theory imposes no explicit upper constraint on the coarse diameter $H_c$}, 
and we see in the numerical results that it is beneficial to have the coarse subdomains large enough, 
making the GEVP \eqref{new3} indefinite, and providing oscillatory modes.  
The separation of local and coarse subdomain sizes which we employ here contrasts with some earlier work, 
which assumed a single subdomain cover for both local and coarse decomposition. 

Our analysis of \eqref{new3} given here is a substantial generalisation of the theory in \cite{Bootland:2022:OSM}, 
which studied the application of  $\Delta$-GenEO coarse spaces (using \eqref{new1}, but applied to systems arising 
from general non-self-adjoint and indefinite PDEs of convection-diffusion-reaction type). 
In that context only the much more stringent estimate $\tau \gtrsim  (1+\Cstab)^2\,k^{8}$ on the eigenvalue tolerance could be proved, 
although the $k^8$ factor was later reduced in \cite{Dolean:2024:ITE}.  
Nevertheless, the eigenvalue tolerance estimate in \eqref{Hk_results} still turns out to be rather pessimistic in practice. 
Our numerical results (in 2D)  seem to indicate that choosing a large enough positive (but constant) value for $\tau$ 
yields $k$-robust convergence of GMRES with very modest coarse space dimensions as $k$ increases. 
Remarks on this are given in \S \ref{sec:numerics}.  

To emphasise the difference between this paper and what went before, we remark that a
substantial task here concerns the study of the indefinite GEVP \eqref{new3} and the corresponding projection onto the eigenspace 
spanned by all negative (and some positive) eigenvalues.  In this respect our analysis is substantially more difficult and different 
from earlier studies using the positive definite problem \eqref{new1}. 
\textit{Since this spectral theory may be interesting in its own right, it is written in an abstract form in Section~\ref{subsec:abstract}}.  
Our proof of \eqref{Hk_results} uses the `Elman theory' for GMRES (\cite{Elman:1983:VIM} or \cite[Lemma C.11]{Toselli:2005:DDM}),
and proceeds by proving an upper bound for the norm of the preconditioned matrix and a lower bound on the 
distance of its field of values from the origin. We estimate these quantities, explicitly analysing their dependence on, 
and independence of, the key parameters of the problem.

\vspace{0.1cm}
\noindent
\textbf{Other relevant work.}
The first publications on spectral coarse spaces for Helmholtz problems with truncated far-field Sommerfeld radiation condition   
include methods based on Dirichlet-to-Neumann maps \cite{Conen:2014:ACS,Bootland:2019:ODN},
while a  numerical survey is given in  \cite{Bootland:2021:ACS} and then enriched in \cite{Dolean:2026:AWR}. 
In \cite{MaAlSc:23} approximation spaces for Helmholtz problems are constructed by solving local positive semi-definite GEVPs  
(similar to \eqref{new1}) in subspaces of local Helmholtz harmonic functions.  
The resulting multiscale approximation space is shown to be  quasi-optimal 
(i.e.\ the Galerkin error is bounded by the best approximation in this space times a constant independent of the parameters---in this case 
the coefficients of the PDE and $k$).  
Subsequently \cite{MaAlSc:24} showed that the corresponding classical additive Schwarz preconditioner yields a Richardson iteration 
which is a parameter-independent contraction.  However, the subdomains in \cite{MaAlSc:24} are constrained to be
of diameter $\mathcal{O}(k^{-1})$ and therefore the local Helmholtz problems are (close to) coercive.
A similar coarse space was proposed in \cite{Hu:2024:ANC}, where Helmholtz problems with constant coefficients are analysed. 
Most of the theory in \cite{Hu:2024:ANC} is concerned with the absorptive case where $k$ is replaced by $k + \mathrm{i} \varepsilon$
with $\varepsilon>0$. In the theory for  $\varepsilon$  small the subdomain diameter is constrained to be $o(k^{-1})$
as $k \rightarrow \infty$ meaning the local problems are again coercive. 

Related results on the extension of multiscale approximation methods (originally devised
for elliptic problems with oscillatory coefficients), to the Helmholtz case are in \cite{Pe:17}.
The corresponding LOD (localised orthogonal decomposition) is used to construct a coarse space in \cite{LuXuZhZo:24}, 
where corresponding theory (for Helmholtz problems with constant coefficients) is presented. 
Related methods are discussed in \cite{FuGoLiWa:24}. 
However, these multiscale approaches are quite different from the spectral methods studied in the present paper. 

Additional recent work \cite{graham:2025:TLH} has shown that classical additive Schwarz preconditioners using coarse spaces 
based on piecewise polynomials are robust for Helmholtz problems as $k \rightarrow \infty$, 
provided the approximations on the fine and coarse levels are  quasi-optimal (independent of $k$). 
Such a prescription can allow the coarse space dimension to be much smaller than the fine space dimension for high wavenumber problems, 
although both dimensions can be quite high. Methods based on fixed degree polynomials are considered in \cite{Galkowski:2025:CTT} 
while $hp$ versions where $p$ increases logarithmically with $k$ are included in \cite{graham:2025:TLH}. 
However, spectral information is not used to construct these coarse spaces and dependence on coefficient variation is not analysed.

\vspace{0.1cm}
\noindent 
\textbf{Plan of the paper.}
In \S \ref{sec:useful} the problem is described and a theory for finite element approximation of \eqref{eq:problem} 
with elements of any order is given. 
Then the spectral theory for indefinite problems is presented in abstract followed by the domain decomposition set-up 
and then the $H_k$-GenEO coarse space is defined. In \S \ref{sec:theory} the main result is stated and the properties 
of the $H_k$-GenEO coarse space are analysed (applying the abstract theory from the previous section, and employing a number of delicate estimates). 
The main theorem is the proved in  \S \ref{sec:proof}, and numerical results given in the final section.

\section{Background and abstract theoretical results}
\label{sec:useful} 

Here we introduce the variational formulation of problem \eqref{pde}, \eqref{bc}, and its discretisation. 
Then, we present an abstract finite-dimensional theory for variational generalised eigenvalue problems with indefinite forms.  
Subsequently, we define our additive Schwarz domain decomposition preconditioner, with the novel feature being the $H_k$-GenEO coarse space. 

\subsection{Problem formulation and finite element discretisation}

The weak formulation of \eqref{eq:problem} is to find $u \in H^1_0(\Omega)$ such that 
\begin{equation}
\label{weak_form}
	b(u,v) = (f,v) \qquad \text{for all} \ v \in H^1_0(\Omega),
\end{equation}
where $f \in L^2(\Omega)$, $(\cdot, \cdot)$ is the $L^2(\Omega)$ inner product and $b : H^1_0(\Omega) \times H^1_0(\Omega) \rightarrow \mathbb{R}$ 
is defined as 
\begin{equation*}
	b(u,v) = \int_\Omega (A\nabla u \cdot \nabla v - \kappa^2{n}  uv)\,\rd x.
\end{equation*}
The following weak regularity assumptions are used throughout this work.

\begin{assumption}\label{ass: 2_1}
	The coefficients $A$ and $n$ 
	and the domain $\Omega$ in problem \eqref{eq:problem} satisfy the following:
	\begin{myenum}
		\item
		$A: \Omega \rightarrow \mathbb{R}^{d \times d}$ and $n: \Omega \rightarrow \mathbb{R}$ are measurable functions 
		with $A$ a symmetric tensor field, such that, for some $0 < \amin \le \amax$ and $0 < \nmin \le \nmax$, 
		\begin{alignat*}{2}
			\amin|\bm{\xi}|^2 &\le A(x)\bm{\xi} \cdot {\bm \xi} \le \amax|\bm{\xi}|^2 
			\qquad && \text{for a.e.} \ x \in \Omega \text{ and all } \bm{\xi} \in \mathbb{R}^d, \\
			\text{and} \qquad \nmin &\le n(x) \le \nmax \qquad && \text{for a.e.} \ x \in \Omega.
		\end{alignat*}
                      
		\item
		Since \eqref{pde} can be equivalently rewritten as
		\[
			- \nabla \cdot \biggl(\frac{A}{\amin}\biggr) \nabla u -
			\biggl(\frac{k^2 \nmax}{\amin}\biggr)
			\biggl(\frac{n}{\nmax}\biggr) u
			= \frac{f}{\amin},
		\]              
		we shall assume, without loss of generality, that $\amin = 1 = \nmax$ and that the diameter, $D_\Omega$, 
		of the domain $\Omega$ satisfies $D_\Omega \le 1$. 
		\item
		We assume that $k\ge k_0$ for some given $k_0 > 0$.
	\end{myenum}
\end{assumption}

\begin{notation}
\label{not:forms}
For any subdomain $\Omega' \subseteq \Omega$ we use $(\cdot,\cdot)_{\Omega'}$ to denote the $L^2(\Omega')$ inner product, 
$(\cdot,\cdot)_{n,\Omega'}\assign(n\cdot,\cdot)_{\Omega'}$ to denote the weighted $L^2(\Omega')$ inner product with weight $n$,
and $\|\cdot\|_{\Omega'}$ and $\|\cdot\|_{n,\Omega'}$ to denote the corresponding norms.  We also introduce the bilinear forms:
\begin{equation}
	a_{\Omega'}(u,v) \assign \int_{\Omega'} A\nabla u \cdot \nabla v \,\rd x, \qquad
	b_{\Omega'}(u,v) \assign a_{\Omega'}(u,v) - k^2(u,v)_{{n},\Omega'}
	\qquad\text{for} \ u,v\in H^1(\Omega'),
    \label{eq:2.18}
\end{equation}
and define the semi-norm induced by $a$, $|u|_{a,\Omega'} \assign \sqrt{a_{\Omega'}(u,u)}$.
We also use the $\kappa$-weighted and $A$- and $n$-dependent inner product:
\begin{equation*}
    (u,v)_{1, \kappa,\Omega'} \assign a_{\Omega'}(u,v) + \kappa^2 (u,v)_{{n},\Omega'},
\end{equation*}
and we denote the induced $\kappa$-norm by $\|u\|_{1, \kappa, \Omega'}$.
When $\Omega'= \Omega$, we abandon the subscript $\Omega$ in all the notation.
\end{notation}

The bilinear forms $a_{\Omega'}(\cdot,\cdot )$ and $(\cdot,\cdot)_{1,\kappa, \Omega'}$ are symmetric positive definite 
on $H_0^1(\Omega')$ and $H^1(\Omega')$ respectively; $b(\cdot,\cdot)$ is symmetric, but in general indefinite.
However, using the Cauchy--Schwarz inequality twice one can easily show that: 
\begin{align}
	|b_{\Omega'}(u,v)| \le \|u\|_{1,k,\Omega'}\|v\|_{1,k,\Omega'}.  \label{lem:est_b} 
\end{align}
We make the following assumption on solvability of \eqref{weak_form}.

\begin{assumption}
\label{Ass: 2_3}
For any $f\in L^2(\Omega)$, problem \eqref{weak_form} has a unique solution $u \in H_0^1(\Omega)$ and 
there exists a constant $\Cstab>0$ such that%
\begin{equation}
\label{eq: 2_10}
	\|u\|_{1,\kappa}  \le \Cstab \|f\| \qquad \text{for all} \ f\in L^2(\Omega). 
\end{equation}
\end{assumption}

\noindent 
Note that $\Cstab$ exists if $k^2$ is not a Dirichlet eigenvalue of the operator $-\frac{1}{n}\diverg(A\nabla\cdot)$.

\begin{assumption}\label{ass:mesh} 
Throughout, $\mathcal{T}_h$ will denote a family of simplicial meshes on $\Omega$, parameterised by maximum diameter $h$. 
We assume this family is \textit{shape regular}, i.e.\ there exists a constant  $\shapereg$ independent of all parameters 
such that $\rho_T \ge \shapereg h_T$ for each $T \in \mathcal{T}_h$, where $h_T$ denotes the diameter of $T$ 
and $\rho_T$ denotes the diameter of the largest inscribed sphere in $T$.  
In the analysis (see Lemma \ref{PBP_PD}) we will require that $h \leq \sqrt{C_1} k^{-1}$ for some constant $C_1$ independent of $k$.
We also assume inverse estimates:
\begin{align} \label{inv_est}
  \|\nabla v \|_{\tau} \leq \Cinv k  \|v\|_{\tau} \qquad \text{for all}
  \;\; \tau \in \mathcal{T}_h, \;\; v \in V^h, 
\end{align}
for some constant $\Cinv\ge1$.   
\end{assumption}

\begin{remark}
A sufficient condition to ensure
\eqref{inv_est} with $\Cinv$ independent of $k$ is obtained by requiring $h_\tau k$
to be bounded below with respect to $k$ for all $\tau$, so that  $h_\tau \sim 1/k$ 
(e.g.\ a quasiuniform mesh with all elements refined of order $1/k$). 
Our strongest results are when $\Cinv$ is independent of $k$.   
However, $\Cinv$ could still be quite large, allowing local mesh refinement treating corners or edges, for example.
In our estimates below, we will treat $\Cinv$ explicitly, while $\shapereg$ will remain implicit. 
\end{remark}

Let $V^h \subseteq H^1_0(\Omega)$ be the $H^1$-conforming Lagrange finite element space of polynomials of degree $r\ge 1$ 
with respect to this mesh, with nodal basis functions denoted by $\{\phi_j: j = 1, \ldots, \dim(V^h)\}$. 
The indices $j$ are also known as \textit{degrees of freedom}.  The Galerkin approximation of \eqref{weak_form} is:
\begin{equation}
\label{eq:discrete_problem}
	\text{find } u_h \in V^h \text{ such that }\qquad b(u_h,v_h) = (f,v_h) \qquad \text{for all} \ v_h \in V^h.
\end{equation}
This is equivalent to the linear system
\begin{equation}
\label{eq: 2_12}
	\mathbf{B}\mathbf{u} = \mathbf{f},
\end{equation}
where $(\mathbf{B})_{ij} \assign b(\phi_j, \phi_i)$ and $(\mathbf{f})_i \assign (f, \phi_i)$. 
Later we also need the matrices $\mathbf{A}$ and $\mathbf{S}$ that correspond to $a$ and $({n} \cdot,\cdot)$, 
namely, $(\mathbf{A})_{ij} \assign a(\phi_j, \phi_i)$ and $(\mathbf{S})_{ij} \assign ({n} \phi_j, \phi_i)$ respectively. 

The solvability of \eqref{eq:discrete_problem} for sufficiently fine meshes {and/or {high enough} polynomial degree} 
is assured by the following lemma, proved for $r=1$ {and $h \rightarrow 0$} in \cite[Theorem~2]{Schatz:1996:SNE}.
Since we could not find  this result for general  $r\ge 1$ in the literature, we give a proof here.

\begin{lemma}[Solvability for higher polynomial degrees]
\label{schatz_wang}
Let Assumptions~\ref{ass: 2_1} and \ref{Ass: 2_3} hold, and set
\begin{equation}
\label{eq:definition_gamma}
	\beta
	\assign
	k\max_{\substack{{\psi} \in L^2(\Omega) \\[0.2ex] {\|\psi\|_n = 1}}}\; \min_{v_h \in V^h} \|{z_\psi}-v_h\|_{1,k},
\end{equation}
where $z_\psi$ is the unique element of $H^1_0(\Omega)$ such that $b(w,{z_\psi}) = (w,\psi)_{{n}}$ for all $w \in H^1_0(\Omega)$. 
Then, there exists a continuous function $\theta: (0,\infty)\to(0,\infty)$, solely depending on $\Omega$, $A$, $n$ and $k$, 
with $\lim_{\tau \to 0} \theta(\tau) = 0$ and such that
\begin{equation}
\label{eq:estimate_gamma}
	\beta \le C \, \theta\Bigl(\frac{h}{r}\Bigr),
\end{equation}
where $C>0$ is a constant which can depend on the shape-regularity constant $\shapereg$.
In addition, if $\beta < 1/\sqrt{2}$, then the linear system in~\eqref{eq:discrete_problem} uniquely determines $u_h$, 
and we have, for all $f \in L^2(\Omega)$,
\begin{equation}
\label{eq:quasi_optimality}
	\|u-u_h\|_{1,k} \le \frac{1}{{1-2\beta^2}} \min_{v_h \in V^h} \|u-v_h\|_{1,k}
	\le \biggl(\frac{\beta}{{1-2\beta^2}}\biggr)\frac{1}{k} \|f\|_{{n^{-1}}}.
\end{equation}
\end{lemma}

The uniqueness of $u_h$ and the estimates in~\eqref{eq:quasi_optimality} for $\beta < \sqrt{2}/2$
constitute a standard result typically referred to as the `Schatz argument' \cite{Schatz:1974:AOC,Schatz:1996:SNE}.  
For completeness, we review the proof of \eqref{eq:quasi_optimality} here before proving \eqref{eq:estimate_gamma}.

\begin{proof}\,
Let us first note that, since $b$ is symmetric, the definition of $\beta$ actually makes sense due to Assumption~\ref{Ass: 2_3}. 
Note also that \textit{a priori} we only have a supremum instead of a maximum in \eqref{eq:definition_gamma};
we will later justify that it is a maximum. It follows from the definition of $\beta$ that, for each $\psi\in L^2(\Omega){\backslash \{0\}}$,
\[
	k \min_{v_h \in V^h} \|{z_\psi}-v_h\|_{1,k} 
	= k \min_{v_h \in V^h} \|(z_\psi/\|\psi\|_n) - v_h\|_{1,k} \|\psi\|_n  
	= k \min_{v_h \in V^h} \|(z_{(\psi/\|\psi\|_n)} - v_h\|_{1,k} \|\psi\|_n 
	\le \beta\|\psi\|_n.
\]
Hence, there exists $v_h\in V^h$ such that
\begin{equation}\label{def_beta_reformulated}
	k\|{z_\psi} - v_h\|_{1,k} \le \beta\|\psi\|_n.
\end{equation}
Now let $f\in L^2(\Omega)$ and let $u\in H_0^1(\Omega)$ be the unique solution of \eqref{weak_form}.
Let us first assume that $u_h \in V^h$ is \textit{any} solution of~\eqref{eq:discrete_problem}.
By Assumption~\ref{Ass: 2_3}, there exists a unique $\xi  {= z_{k(u - u_h)}}\in H^1_0(\Omega)$
such that 
\begin{equation}\label{existence_xi}
	b(w,\xi) = k(w,u-u_h)_{{n}} \qquad \text{for all} \;\; w \in H^1_0(\Omega). 
\end{equation}
On one hand, (by {\eqref{def_beta_reformulated} with $\psi=k(u-u_h)$}), there exists $\xi_h \in V^h$ such that
\begin{equation*}
	\|\xi-\xi_h\|_{1,k} \le \beta \|u-u_h\|_{{n}}.
\end{equation*}
On the other hand, taking the test function $w = u-u_h$ {in \eqref{existence_xi}} 
we see that
\begin{equation*}
	k\|u-u_h\|_{{n}}^2 = b(u-u_h,\xi) = b(u-u_h,\xi-\xi_h)
	\le \|u-u_h\|_{1,k}\|\xi-\xi_h\|_{1,k},
\end{equation*}
where we used \eqref{lem:est_b} and the fact that $b(u-u_h,w_h)=0$ for $w_h\in V^h$ since $u$ solves \eqref{weak_form} 
and $u_h$ solves \eqref{eq:discrete_problem}.
Combining these two estimates we obtain
\begin{equation*}
	k\|u-u_h\|_{{n}} \le \beta\|u-u_h\|_{1,k},
\end{equation*}
and hence
\begin{align*}
	(1-2\beta^2)\|u-u_h\|_{1,k}^2
	&\le \|u-u_h\|_{1,k}^2-2k^2\|u-u_h\|_{{n}}^2 
	= b(u-u_h,u-u_h)
	\\[1ex]
	&= b(u-u_h,u-v_h)
	\le \|u-u_h\|_{1,k}\|u-v_h\|_{1,k}
\end{align*}
for all $v_h \in V^h$.  Assuming that $2\beta^2 < 1$ we obtain the first estimate in~\eqref{eq:quasi_optimality}.
The second estimate follows from the definition of $\beta$; see again \eqref{def_beta_reformulated}.

At that point, we have established \eqref{eq:quasi_optimality} for \textit{any} solution,
but we do not know yet if such a solution exists.  
However, $u_h$ is defined through a finite-dimensional linear system. 
In this setting, stability implies uniqueness, which in turns implies existence. 
This concludes the proof of~\eqref{eq:quasi_optimality}.

We now turn our attention to~\eqref{eq:estimate_gamma}.  We first observe that, thanks
to Fredholm alternative {(applied in $H^{-1}(\Omega)$)}, Assumption~\ref{Ass: 2_3} guarantees that the mapping
$S: H^{-1}(\Omega) \to H^1_0(\Omega)$ defined by
\begin{equation*}
	b(S\psi,v) = \langle \psi,v \rangle
\end{equation*}
for all $\psi \in H^{-1}(\Omega)$ and $v \in H^1_0(\Omega)$ is a well-defined continuous linear operator
where $\langle\cdot,\cdot\rangle$ denotes the $H^{-1}(\Omega)$, $H_0^1(\Omega)$-pairing.

Let  $B_{L^2} = \{\phi \in L^2(\Omega) \; | \; \|{n^{-1}}\phi\|_{n} \le 1\}$
denote the closed unit ball in $L^2(\Omega)$.
Since the injection $L^2(\Omega) \hookrightarrow H^{-1}(\Omega)$ is compact, we conclude
that $K \assign S(B_{L^2})$ is a compact set in $H^1_0(\Omega)$.  
Note that for $\phi \in L^2(\Omega)$ we have
$S\phi = z_{\{n^{-1} \phi\}}$ (with the notation introduced above).
Hence, 
\begin{equation}
\label{beta_K}
	\max_{v \in K} \min_{v_h \in V^h} \|v-v_h\|_{1,k}
	= \max_{\substack{\phi \in L^2(\Omega) \\[0.2ex] {\|n^{-1}\phi\|_n \le 1}}}\; \min_{v_h \in V^h} \| z_{(n^{-1}\phi)}-v_h\|_{1,k}
	= \max_{\substack{{\psi} \in L^2(\Omega) \\[0.2ex] \|\psi\|_n \le 1}}\; \min_{v_h \in V^h} \| z_\psi-v_h\|_{1,k}
	= \frac{\beta}{k}\,.
\end{equation}
Let us fix $\varepsilon > 0$.  Since $K$ is compact, there exist $N$ functions $\{v^j\}_{j=1}^N \subseteq K$
such that, for each $v \in K$, there exists $j \in \{1,\dots,N\}$ with
\begin{equation*}
	k\|v-v^j\|_{1,k} \le \frac{\varepsilon}{3}.
\end{equation*}
We then invoke the density of $C_{\rm c}^\infty(\Omega)$ {in}
$H^1_0(\Omega)$, which guarantees that, for each {$j\in\{1,\ldots,N\}$},
there exists ${\widetilde{v}^j} \in C_{\rm c}^\infty(\Omega)$ with
\begin{equation*}
	k\|v^j-\widetilde{v}^j\|_{1,k} \le \frac{\varepsilon}{3}.
\end{equation*}
Now we invoke {a} degree-explicit approximation result
(e.g.\ Melenk and Sauter \cite{MeSa:10}) which implies that there
exists a constant $C$ (depending on $\shapereg$) such that
\[
	\min_{v_h \in V^h} \| \widetilde{v}^j - v_h\|_{H^i(\Omega)} 
	\le C \Bigl(\frac{h}{r}\Bigr)^{2-i}|\widetilde{v}^j|_{H^2(\Omega)} 
	\qquad \text{for each} \;\;  i = 0,1, \quad j = 1,\ldots , N.
\]
This estimate follows from \cite[Theorem~B.4]{MeSa:10} (given only on a unit element), by applying a standard scaling argument on each finite element
and using the fact that the resulting global approximation operator maps continuous functions into $V^h$;
see, e.g.\ the steps in \cite[Theorem~5.5]{MeSa:10}.
Hence there exist $v^1_h,\ldots,v^N_h \in V^h$ such that
\begin{align}
	k \|\widetilde{v}^j-v^j_h\|_{1,k}
	&\le k\Bigl(\amax\|\widetilde{v}^j-v_h^j\|_{H^1(\Omega)}+k^2\|\widetilde{v}^j-v_h^j\|_{L^2(\Omega)}^2\Bigr)^{1/2}
	\nonumber\\[1ex]
	&\le C\frac{kh}{r}\biggl(\amax+\Bigl(\frac{kh}{r}\Bigr)^2\biggr)^{1/2}|\widetilde{v}^j|_{H^2(\Omega)}
	\le	C\frac{kh}{r}\biggl(\amax+\Bigl(\frac{kh}{r}\Bigr)^2\biggr)^{1/2}M(\varepsilon)
\label{est_interm_ml01}
\end{align}
for $h/r$ small enough, where
\(
	M(\varepsilon) \assign \max_{1 \le i \le N} |\widetilde{v}^i|_{H^2(\Omega)}.
\)
Putting the pieces together and using the triangle inequality we obtain
\begin{equation*}
	k \min_{v_h \in V^h} \|v-v_h\|_{1,k}
	\le \frac{2\varepsilon}{3} + C\frac{kh}{r}\biggl(\amax+\Bigl(\frac{kh}{r}\Bigr)^2\biggr)^{1/2}M(\varepsilon)
\end{equation*}
for all $v \in K$.
Let us write $\beta(h,r)\assign\beta$ to indicate the dependence on $h$ and $r$.  
We can assume, without loss of generality, that $C \, \ge1$.
It follows from \eqref{beta_K} and \eqref{est_interm_ml01} that
\[
	\frac{\beta(h,r)}{C\,} \le \frac{2\varepsilon}{3} + C\frac{kh}{r}\biggl(\amax+\Bigl(\frac{kh}{r}\Bigr)^2\biggr)^{1/2}M(\varepsilon).
\]
Clearly, there exists $\delta>0$ such that $\frac{\beta(h,r)}{C\,}\le\varepsilon$ if $h/r\in(0,\delta)$.
Now choose a decreasing sequence $\varepsilon_\ell>0$ such that $\lim_{\ell\to\infty}\varepsilon_\ell=0$.
Then there exist $\delta_\ell>0$ such that $(\delta_\ell)_{\ell=1}^\infty$ is strictly decreasing, $\lim_{\ell\to\infty}\delta_\ell=0$ and
\[
	\frac{\beta(h,r)}{C\,} \le \varepsilon_\ell \qquad\text{for} \;\; h,r>0 \quad \text{with}\;\;\; \frac{h}{r}<\delta_\ell.
\]
Defining $\theta(\delta_\ell)\assign\varepsilon_\ell$ and interpolating linearly we obtain  $\theta$
satisfying  \eqref{eq:estimate_gamma}.
\end{proof}

\noindent
Before proceeding, we  recall  Friedrichs' inequality (e.g.\ \cite[Theorem 13.19]{Leoni:2017:FCS}).

\begin{lemma}[Friedrichs' inequality and a consequence]\label{lemma: 2.5}
Let $\Omega' \subseteq \mathbb{R}^d$ be an open set that lies between two parallel hyperplanes, separated by a distance $L$.  
Then, for all $u \in H^1_0(\Omega')$, 
\begin{equation}
\label{eq:friedrichs}
	\|u \|_{\Omega'} \le \frac{L}{\sqrt{2}\,} \|\nabla u\|_{\Omega'}. 
\end{equation}
Combining \eqref{eq:friedrichs} with Assumption~\ref{ass: 2_1}\,(ii) we obtain that, for any subdomain $\Omega' \subseteq \Omega$ with diameter $H$, 
the following estimate holds:
\(
	\|u\|_{\Omega'} \le \frac{H}{\sqrt{2}\,} \|\nabla u\|_{\Omega'} \le \frac{H}{\sqrt{2}\,} \|u\|_{1,k,\Omega'} 
\)
for all $u \in H^1_0(\Omega')$.
\end{lemma}

\subsection{Abstract spectral theory for indefinite generalised  eigenvalue problems}
\label{subsec:abstract}

We shall build our coarse space by solving 
local generalised eigenvalue problems,  which may be indefinite (see \eqref{eq: 5_12}).  
So in this section we derive, in  an abstract setting, the required properties of such problems.
First we {collect} some properties of a positive semi-definite bilinear form
(such as appears on the right-hand side of \eqref{eq: 5_12}).

\begin{lemma} \label{cpd}
Suppose $\wV$ is a real vector space with $\dim \wV = \tilde n$.   
Let $\txtc$ be a symmetric positive semi-definite bilinear form on $\wV$, 
let $\ker \txtc \assign \{ v \in \wV: \txtc(v,w) = 0 \ \text{for all} \ w \in \wV\}$ 
and suppose that $\dim \ker \txtc = s < \tilde n$. 
Further, let $\whV\subseteq \wV$ be an $(\tilde n-s)$-dimensional subspace such that $\whV\cap\ker\txtc=\{0\}$.
Then $\wV=\whV\oplus\ker\txtc$, where $\oplus$ denotes a direct sum, and $\txtc$ is positive definite on $\whV$.
\end{lemma}

\begin{proof}
The relation $\wV=\whV\oplus\ker\txtc$ is clear since $\dim\whV+\dim\ker\txtc=\tilde n$ and $\whV\cap\ker\txtc=\{0\}$.
The form $\txtc|_{\whV\times\whV}$ is symmetric and satisfies $\txtc(v,v)>0$ for $v\in\whV\setminus\{0\}$ by the assumption on $\whV$.
Hence $\txtc|_{\whV\times\whV}$ is positive definite.
\end{proof}

\begin{theorem}[Abstract theorem for indefinite eigenvalue problems]
\label{lem_ind_ev}
Let $\wV$, $\txtc$ and $s$ be 
as in Lemma~\ref{cpd} and suppose also that $\txtb$ is a symmetric \textup{(}possibly indefinite\textup{)} bilinear form on $\wV$. 
Consider the generalised eigenvalue problem: find $\lambda \in \mathbb{R}$ and $p \in \wV \backslash \{0\}$ such that 
\begin{equation} \label{Ivan1}
  \txtb(p,v)  = \lambda \txtc(p,v) \qquad \text{for all} \quad v \in \wV.
\end{equation}
If $s \ge 1$, suppose also that $\txtb$ is positive definite on $\ker\txtc$. 
Then  $\wV$ has a basis $\{p_j: j = 1, \ldots,\tilde n\}$,  in which  $\{p_j: j = 1, \ldots, \tilde n-s\}$ are eigenvectors of \eqref{Ivan1}  which
can be chosen to be orthonormal with respect to $\txtc$ and the corresponding finite eigenvalues can be ordered
\begin{equation} \label{Ivan11}
	-\infty < \lambda_1 \leq  \lambda_2 \leq   \ldots \leq \lambda_{\tilde n-s} < \infty. 
\end{equation} 
The remaining  basis vectors $\{p_j: j = \tilde n-s+1, \ldots, \tilde n\}$  form a basis of $\ker\txtc$. 
In this sense we say that they correspond to infinite  eigenvalues of \eqref{Ivan1}.
Moreover, 
\begin{alignat}{2}
	\txtb(p_j,p_{j'}) &= \lambda_j \delta_{j,j'}, \quad && \ j \in \{1, \ldots, \tilde n-s\}, \ j'\in\{1,\ldots,\tilde n\};
	\label{orth_b}
	\\
	\txtc(p_j,p_{j'}) &= \delta_{j,j'}, \quad && \ j\in\{1,\ldots,\tilde n-s\},\ j'\in\{1,\ldots,\tilde n\}.
	\label{orth_c}
\end{alignat}
If $s=0$ (trivial case) the same conclusion holds but the positive definiteness condition on $\txtb$ is not required, 
and all eigenvalues are finite.
\end{theorem}

\begin{proof}
Suppose $s \ge 1$ and $\txtb$ is positive definite on $\ker \txtc$. 
Let us choose $\whV\subseteq\wV$ such that $\wV = \whV\oplus \ker \txtc$ (as in Lemma~\ref{cpd});
then every $p \in \wV$ has a unique decomposition $p = \whp + p_{\txtc} $ with $\whp\in \whV$ and  $p_{\txtc}\in \ker \txtc$. 
Thus the GEVP \eqref{Ivan1} is equivalent to seeking $\lambda \in \mathbb{R}$, $\whp \in \whV$ and  $p_{\txtc} \in \ker \txtc$ such
that $\whp + p_{\txtc} \ne 0$ and  
\begin{align} \label{Ivan1a} 
	\txtb(\whp + p_{\txtc}, v ) = \lambda \txtc (\whp, v) \qquad \text{for all} \;\; v \in \wV.
\end{align}
This, in turn, is equivalent to solving  the coupled system
\begin{alignat}{2}
	\txtb(\whp, v_\txtc) + \txtb(p_{\txtc}, v_{\txtc})    &= 0 \qquad 
	&& \text{for all} \;\; v_{\txtc} \in \ker \txtc, 
	\label{Ivan1b}\\
	\txtb(\whp,\whv) + \txtb(p_{\txtc},\whv) &= \lambda \txtc (\whp,\whv) \qquad\quad 
	&& \text{for all} \;\; \whv \in \whV,
	\label{Ivan1c}
\end{alignat}
for $\lambda \in \mathbb{R}$,   $\whp\in \whV$ and $p_{\txtc} \in \ker \txtc$ such
that $ \whp + p_{\txtc} \ne 0$. 
    
Now, by assumption on  $\txtb$, for each $\whw \in \whV$, there exists a unique $w_\txtc \in \ker \txtc$ such that
\begin{align}\label{Ivan1c1} 
	\txtb(w_\txtc, v_\txtc) = - \txtb(\whw , v_\txtc) \qquad \text{for all} \;\; v_\txtc \in \ker \txtc. 
\end{align} 
Denoting the solution to \eqref{Ivan1c1} as $w_\txtc = \scS \whw$, this defines a  linear map
$\scS: \whV\rightarrow \ker \txtc$ with the property
\begin{align}\label{Ivan1b1}
	\txtb(\scS \whw, v_\txtc) = - \txtb(\whw, v_\txtc) \qquad \text{for all} \;\;
	\whw \in \whV \;\; \text{and} \;\; v_\txtc \in \ker \txtc,
\end{align}
or, equivalently,  
\begin{align} \label{Ivan1d} 
	\txtb\bigl((I + \scS) \whw, v_\txtc\bigr) = 0 \qquad \text{for all} \;\; \whw \in \whV \;\; \text{and} \;\; v_\txtc \in \ker \txtc.
\end{align} 

We now  consider the GEVP: find $\lambda \in \mathbb{C}$ and   $\whp \in \whV\backslash \{0\}$ such that
\begin{align} \label{Ivan1e} 
	\widetilde{\txtb} (\whp, \whv) \assign \txtb\bigl((I + \scS) \whp, \whv\bigr) = \lambda \txtc(\whp, \whv) \qquad 
	\text{for all} \;\; \whv \in \whV.  
\end{align}
(This is the Schur complement of \eqref{Ivan1b}, \eqref{Ivan1c}.)
    
Note that $\scS$ is self-adjoint with respect to $\txtb$. This is because, using the symmetry of $\txtb$ (at the first and third steps) 
and the definition of $\scS$ (at the second and fourth steps),  
\[
	\txtb(\scS \whp, \whv) = \txtb(\whv, \scS \whp) = -\txtb(\scS \whv, \scS \whp) 
	= -\txtb( \scS \whp,\scS \whv) =  \txtb(\whp, \scS \whv).
\] 
Thus the bilinear form $\widetilde{\txtb}$ on the left-hand side of \eqref{Ivan1e} is symmetric on $\whV$.
Now $\txtc$ is symmetric and (by Lemma~\ref{cpd}) positive definite on $\whV$;
so the GEVP \eqref{Ivan1e} has $\tilde n-s$ eigenpairs,
which we denote  $(\lambda_j, \whp_{j})$, $j = 1, \ldots , \tilde n-s$, with real finite eigenvalues $\lambda_j$,
which can be ordered as in \eqref{Ivan11}.  The $\whp_{j} \in \whV$ can be chosen orthonormal with respect to $\txtc$ 
and form a basis of $\whV$.  
Now we define 
\begin{align} \label{eigdefj} 
	p_j = (I + \scS)\whp_{j}, \qquad j = 1, \ldots, \tilde n-s.
\end{align} 
By \eqref{eigdefj}, \eqref{Ivan1e} and since $\scS \whp_{j} \in \ker \txtc$, we have 
\[
	\txtb(p_j, \whv) = {\widetilde{\txtb}(\whp_{j},\whv)}
	= \lambda_j \txtc(\whp_{j}, \whv) = \lambda_j \txtc(p_{j}, \whv)
	\qquad \text{for all} \;\; \whv \in \whV.
\]
Also, by \eqref{eigdefj} and \eqref{Ivan1d} we have $\txtb(p_j, v_{\txtc}) = 0$ for all $v_{\txtc} \in \ker \txtc$, and so
(see \eqref{Ivan1b} and  \eqref{Ivan1c}), 
$(\lambda_j, p_j),\,  j = 1, \ldots ,\tilde n-s$
are eigenpairs of the original GEVP \eqref{Ivan1}.
Now, choosing an arbitrary basis, $\{p_{\tilde n-s+1},\ldots,p_{\tilde n}\}$, of {$\ker\txtc$} we obtain a basis of $\wV$. 
For $j \in \{1, \ldots , \tilde n-s\}$, relation \eqref{orth_c} follows for $ j' \in \{ 1, \ldots, \tilde n-s\}$ 
because  $\scS \whp_{j}, \scS \whp_{j'} \in \ker \txtc$, and so $\txtc(p_j,p_{j'}) = \txtc(\whp_{j},\whp_{j'}) = \delta_{j,j'}$ 
since the $\{\whp_{j}: j = 1, \ldots, \tilde n-s\}$ were chosen orthonormal with respect to $\txtc$.   
For $j' \in  \{\tilde n-s+1, \ldots, \tilde n\}$,  \eqref{orth_c} is trivial because then $p_{j'} \in \ker \txtc$.

Relation \eqref{orth_b} then follows directly  from \eqref{orth_c} because,  for $j \in \{1, \ldots, \tilde n-s\}$ 
and any $j' {\in \{1,\dots,\tilde n\}}$, $\txtb(p_j, p_{j'}) = \lambda_j\txtc(p_j, p_{j'})$. 
This completes the proof when $s\ge1$, while the case $s = 0$ is trivial since \eqref{Ivan1} can then be reduced to a standard eigenvalue problem.     
\end{proof}

\begin{remark}
Note that some statements in the previous proof can be made more general:
\begin{enumerate}
\item 
	The requirement that $\txtb$ is positive definite on $\ker \txtc$ in part \textup{(i)} of Theorem~\ref{lem_ind_ev} 
	could be weakened to simply requiring that problem \eqref{Ivan1c1} has a unique solution 
	for all {$\whw \in \whV$, where $\whV$ is as in Lemma~\ref{cpd}}. 
	However, since we shall need the positive definiteness in Proposition~\ref{lem_abs_proj} anyway, we avoid this extra generality.
\item
	In general, the eigenvectors $p_j,\, j = 1, \ldots \tilde n-s$ defined via \eqref{eigdefj} are not elements of $\whV$ 
	because $\scS$ maps into $\ker \txtc$.
\end{enumerate} 
\end{remark} 

\noindent
The following corollary then follows directly from Theorem \ref{lem_ind_ev}.

\begin{corollary}\label{lem_v_abstract2}
Under the conditions of Theorem~\ref{lem_ind_ev},  every $v \in V$ can be expressed as
\begin{equation}\label{Ivan5} 
	v =  v_0 +  \sum_{j=1}^{\tilde n-s} \txtc(v,p_j)p_j    \quad \text{with} \quad v_0 \in \ker \txtc,
\end{equation}
If $\ker\txtc = \{0\}$ then $s=0$ and $v_0 = 0$ .
\end{corollary}

We now define a projection operator mapping  ${\wV}$ onto the span of the first $m$
eigenvectors of \eqref{Ivan1}, where $m$ is chosen so that all eigenvectors corresponding to negative eigenvalues are included
(recall the ordering \eqref{Ivan11}).
This  operator is  crucial in the analysis of the proposed preconditioner.

\begin{proposition}[Projection onto the $m$-dimensional eigenspace]\label{lem_abs_proj}
Under the conditions of Theorem~\ref{lem_ind_ev}, 
suppose that there exists $m \in \{ 1, \ldots, \tilde n-s-1\}$
such that
$\lambda_{m+1}>0${,} 
and  define the projector
\begin{equation}\label{Ivan6}
	\Pi v = \sum_{j=1}^m \txtc(v,p_j) p_j \qquad \text{for all} \;\; v \in \wV.
\end{equation}
Then
\begin{equation}\label{Ivan7}
	\txtc(v-\Pi v,v-\Pi v) \le \frac{1}{\lambda_{m+1}} \txtb(v-\Pi v,v-\Pi v).
\end{equation}
\end{proposition}

When applying this proposition we will show  that such an $m$ exists, under a modest assumption.
\begin{proof}\, 
Let $v \in {\wV}$ be expressed as in \eqref{Ivan5} and set $\alpha_j\assign\txtc(v,p_j)$, $j\in\{1,\ldots,\tilde n-s\}$. 
We use the properties from Corollary \ref{lem_v_abstract2} and Theorem~\ref{lem_ind_ev} to obtain,
\begin{align*}
	\txtb(v-\Pi v,v-\Pi v) &= \txtb\Biggl(v_0 + \sum_{j=m+1}^{\tilde n-s} \alpha_j p_j\, ,\, v_0 +  \sum_{j=m+1}^{\tilde n-s} \alpha_j p_j \Biggr)
	\nonumber\\[1ex]
	&= \txtb(v_0,v_0)   + 2 \sum_{j = m+1}^{\tilde n-s} \alpha_j \txtb( p_j, v_0) + \sum_{j=m+1}^{\tilde n-s}  \lambda_j \alpha_j^2 \nonumber                                 \\
	&= \txtb(v_0,v_0) +  2 \sum_{j = m+1}^{\tilde n-s} \alpha_j \lambda_j \txtc(p_j, v_0) +  \sum_{j=m+1}^{\tilde n-s} \lambda_j \alpha_j^2.
\end{align*}
Then we use the the fact that $v_0 \in \ker \txtc$ and the positive definiteness of $\txtb$ on $\ker \txtc$ to obtain 
\begin{equation} 
\label{Ivan281}
	\txtb(v-\Pi v,v-\Pi v)  = \txtb(v_0,v_0) + \sum_{j=m+1}^{\tilde n-s} \lambda_j \alpha_j^2 \ge \sum_{j=m+1}^{\tilde n-s} \lambda_j \alpha_j^2  
	\ge \lambda_{m+1} \sum_{j=m+1}^{\tilde n-s}  \alpha_j^2.
\end{equation}
Using an almost identical argument and the results from Theorem~\ref{lem_ind_ev} we also obtain
\begin{equation}
	\txtc(v-\Pi v,v-\Pi v) = \sum_{j=m+1}^{\tilde n-s}  \alpha_j^2.
\label{Ivan282}
\end{equation}
{Now \eqref{Ivan7}} follows from \eqref{Ivan281} and \eqref{Ivan282}. 
\end{proof}

This lemma shows that, although the bilinear form $\txtb(\cdot,\cdot)$ may be indefinite on the full space, 
it becomes positive definite on the complement of the subspace spanned by the first $m$ eigenvectors, 
provided all discarded eigenvalues $\lambda_{j}$ (for $j > m$) are strictly positive. 
In the context of coarse space construction, this implies that if we retain all eigenfunctions associated 
with non-positive or small eigenvalues, then the remainder is well-controlled by $\txtb(\cdot,\cdot)$, 
leading to stability estimates essential for robust preconditioning.

\subsection{Domain decomposition}

\begin{notation} \label{not:generalD}
Let $\Omega'$ be any subdomain of  $\Omega$, composed of a union of elements of the mesh $\mathcal{T}_h$.
We introduce the finite element spaces:
\begin{equation*}
	\widetilde{V}_{\Omega'} \assign \bigl\{v|_{\Omega'} : v \in V^h\bigr\} \subseteq H^1(\Omega') 
	\qquad \text{and} \qquad 
	V_{\Omega'} \assign \bigl\{v \in \widetilde{V}_{\Omega'} : v|_{\partial \Omega'} = 0\bigr\} \subseteq H^1_0(\Omega'). 
\end{equation*}
For any $v_{\Omega'} \in V_{\Omega'}$, we let $E_{\Omega'} v_{\Omega'} \in V^h$ denote its zero extension to the whole domain $\Omega$, 
and we  define the operator $R_{\Omega'}:  V^h \rightarrow V_{\Omega'}$  by
\[
  (R_{\Omega'} v , w_{\Omega'})_{\Omega'} \assign (v,E_{\Omega'}w_{\Omega'}) \quad \text{for all} \quad  v \in V^h \quad \text{and} \quad  w_{\Omega'} \in V_{\Omega'}.
\]
\end{notation}

In order to construct the two-level Schwarz preconditioner, we  choose two overlapping covers of $\Omega$.
The first  consists of `local subdomains' $\{\Omega_j^\ell\}_{j=1}^Q$ 
on which  restrictions of \eqref{eq:  2_12} 
will be solved.  
The second consists of `coarse-space subdomains' $\{\Omega_i^c\}_{i=1}^N$ on which generalised eigenvalue problems will be solved
to build  the  coarse space.  All subdomains are assumed to be polygonal (2D) or polyhedral (3D) and have boundaries which are  
resolved by the fine mesh $\mathcal{T}_h$,
and satisfy  
\begin{equation} 
	\Omega = \bigcup_{i=1}^{N} \Omega_i^c, \qquad \Omega = \bigcup_{j=1}^{Q} \Omega_j^\ell.
	\label{covers}
\end{equation}
There is no need for these covers to coincide and in practical algorithms
it may be convenient for them to be different. The `coarse' subdomains should capture the indefiniteness and are not required to be small in theory. 
For each, ${\Omega_i^c}$ and ${\Omega_j^\ell}$, we denote their  diameters by $H_{c,i}$, $H_{\ell,j}$ respectively, 
and we set $H_c \assign \max_{i=1}^N H_{c,i}$ and $H_\ell \assign \max_{j=1}^Q H_{\ell,j}$.
We  define the corresponding finite element spaces 
$V_{\Omega^c_i}, \widetilde{V}_{\Omega^c_i}$ and $V_{\Omega^\ell_j}, \widetilde{V}_{\Omega^\ell_j}$ 
as in Notation \ref{not:generalD}.  We also use abbreviations $E_i^c = E_{\Omega_i^c}$ and $E_j^\ell = E_{\Omega_j^\ell}$.

\begin{assumption}\label{ass:overlap}
  In the theory we assume that the overlap $\delta_\ell$  of the local cover  $\{\Omega_j^\ell\}$ (as defined in,
  e.g.\ \cite[p.~56]{Toselli:2005:DDM}) satisfies $\delta_\ell \geq k^{-1}$.
\end{assumption} 

Standard properties for such overlapping covers are (see, e.g.\ \cite[eq.~(2.10)]{Graham:2020:DDI}):
\begin{equation} 
	\sum_{i=1}^N \Vert v \vert_{\Omega_i^c} \Vert_{1,k,\Omega_i^c}^2 \leq \Lambda_c \Vert v \Vert_{1,k}^2 , 
	\qquad
	\sum_{j=1}^Q \Vert v\vert_{\Omega_j^\ell} \Vert_{1,k,\Omega_j^\ell}^2 \leq \Lambda_\ell \Vert v \Vert_{1,k}^2 
	\qquad \text{for all} \;\; v \in V^h,
	\label{overlap}
\end{equation}
where
\begin{equation}\label{Lambda_l_c}
	\Lambda_\circ \assign \max_{T\in\mathcal{T}_h}\bigl(\#\bigl\{\Omega^\circ_i \mid 1 \le i \le N,\, T \subseteq \Omega^\circ_i\bigr\} \bigr) 
	\qquad \text{for} \;\;  \circ \in \{c,\ell\}.
\end{equation}
In addition (\cite[Lemma 3.6]{Graham:2020:DDI}), for any  $v_j \in H^1_0(\Omega)$ with
$\supp v_j \subseteq \overline{\Omega_j^\ell}$ for each $j = 1, \ldots, Q$, we have the estimate 
\begin{equation}\label{overlap1}
	\bigg\| \sum_{j=1}^Q v_j \bigg\|_{1,k}^2  \le \Lambda_\ell \sum_{j=1}^Q \|v_j|_{\Omega_j^\ell}\|_{1,k,\Omega_j^\ell}^2,
\end{equation}
and, similarly, for  $v_i \in H_0^1(\Omega)$ with $\supp v_i \subseteq \overline{\Omega_i^c}$, for each $i=1,\ldots,N$.
\begin{equation}\label{overlap2}
	\bigg\| \sum_{i=1}^N v_i \bigg\|_{1,k}^2  \le \Lambda_c \sum_{i=1}^N \|v_i|_{\Omega_i^c}\|_{1,k,\Omega_i^c}^2 .
\end{equation}
The one-level additive Schwarz preconditioner can now be given in matrix form as 
\begin{equation}
  \label{one_level}
	\mathbf{M}^{-1}_{AS,1} = \sum_{j=1}^{Q} \mathbf{E}^{\ell}_j (\mathbf{B}^{\ell}_j)^{-1} \mathbf{R}^{\ell}_j, \qquad 
	\text{where } \mathbf{B}^{\ell}_j = \mathbf{R}^{\ell}_j \mathbf{B} \mathbf{E}^{\ell}_j;
\end{equation}
here, $\mathbf{E}^{\ell}_j$ and $\mathbf{R}^{\ell}_j$ denote the matrix representations 
of \(E_j^\ell\) and \(R_j^\ell\), respectively, with respect to the 
basis functions $\lbrace \phi_i \rbrace_{i=1}^{\dim(V^h)}$ of $V^h$. 
 
In order to improve the preconditioner, a global coarse space is added.  
Let $V_0 \subseteq V^h$ be such a coarse space, let $E_0: V_0 \rightarrow V^h$ be the natural embedding, 
and let $R_0$ be the $L^2$ adjoint of $E_0$,
\begin{equation*}
	(R_0 w,v_0) = (w,E_0 v_0) \qquad \text{for all} \ w \in V^h, v_0 \in V_0.
\end{equation*}
The two-level additive Schwarz preconditioner is then 
\begin{equation} \label{eq: 2_23}
	\mathbf{M}^{-1}_{AS,2} = \mathbf{E}_0 \mathbf{B}^{-1}_0 \mathbf{R}_0 + \mathbf{M}^{-1}_{AS,1}, 
	\qquad \text{where } \mathbf{B}_0 \assign \mathbf{R}_0 \mathbf{B} \mathbf{E}_0 
\end{equation}
(with $\mathbf{E}_0$ and $\mathbf{R}_0$ denoting matrix representations of $E_0$ and $R_0$).
The preconditioned version of \eqref{eq: 2_12} is
\begin{equation} \label{eq: 2_24}
	\mathbf{M}^{-1}_{AS,2} \mathbf{B} \mathbf{u} = \mathbf{M}_{AS,2}^{-1} \mathbf{f}.
\end{equation}

For the analysis we need certain projection operators defined as follows.  
For each $j = 1, \ldots, Q$, we define  $T^\ell_j: V^h \rightarrow V_{\Omega^\ell_j}$ and  $T_0: V^h \rightarrow V_0$ by
\begin{align}
	b_{\Omega^\ell_j}(T^\ell_j u, v) = b(u, E^\ell_j v) \qquad &\text{for all} \ {u\in V^h,}\, v \in V_{\Omega^\ell_j}, 
	\label{eq: bT} 
	\\[0.5ex]
	b(T_0 u, v) = b(u, E_0 v) = b(u,v)  \qquad &\text{for all} \;\; u\in V^h \;\; \text{and} \;\;  \ v \in V_0.
	\label{eq: 2_25_zero}
\end{align}
Sufficient conditions for the existence of $T^\ell_j$ and $T_0$ (and hence invertibility of $\mathbf{B}^\ell_j$, $\mathbf{B}_0$) 
are given later (Lemmas~\ref{lemma_3_4} and \ref{lemma_3_5}).
Given the operators $T^\ell_j$ and $T_0$, we define $T: V^h \rightarrow V^h$ by  
\begin{equation}
	T = E_0T_0 + \sum_{j=1}^Q E^\ell_j T^\ell_j. \label{eq: 2_26}
\end{equation}
Then the preconditioned system \eqref{eq:  2_24} is related to $T$ via the following proposition
(cf, e.g.\ \cite[Theorem~5.5]{Graham:2017:DDP}).

\begin{proposition}\label{prop: 2_6}
For any $u,v \in V^h$, with corresponding nodal vectors $\mathbf{u},\mathbf{v} \in \mathbb{R}^{\dim(V^h)}$, 
\begin{equation} \label{eq: 2_27}
	\langle \mathbf{M}_{AS,2}^{-1} \mathbf{B} \mathbf{u}, \mathbf{v} \rangle_{\mathbf{D}_\kappa} = (T u, v)_{1, \kappa},
\end{equation}
where $\langle \cdot , \cdot \rangle_{\mathbf{D}_\kappa}$ is the inner product on $\mathbb{R}^{\dim(V^h)}$ and the matrix $\mathbf{D}_\kappa$
given by
\begin{equation}\label{def:Dk}
    \mathbf{D}_\kappa \assign \mathbf{A} + \kappa^2  \mathbf{S}.
\end{equation}
\end{proposition}

\subsection[The $H_k$-GenEO coarse space]{The {\boldmath$H_k$}-GenEO coarse space}

\begin{definition}[{\cite[Definition 3.2]{Spillane:2014:ARC}}] \label{Def: 3_2}
Given any $\Omega' \subseteq \Omega$, formed from a union of elements $T \in \mathcal{T}_h$, let 
\begin{align*}
	\ovdof(\Omega') &\assign \bigl\{j \mid 1 \le {j} \le \dim(V^h) \text{ and } \supp(\phi_{j}) \cap \Omega' \ne \emptyset\bigr\},
	\\[0.5ex]
	\dof(\Omega') &\assign \bigl\{j \mid 1 \le j \le \dim(V^h) \text{ and } \supp(\phi_{{j}}) \subseteq \overline{\Omega'}\bigr\}
\end{align*}
denote, respectively,  the  degrees of freedom in the closed domain  $\overline{\Omega'}$ and its interior.
\end{definition}

\begin{notation}\label{notation:sim}
In what follows we will write $A\lesssim B$, to mean $ A \leq C B$ with a constant $C$ independent of the 
\textbf{key parameters} $k,h,\Cinv, \delta_\ell, H_\ell, H_c, \Lambda_c, \Lambda_\ell$, $\amax$, $\nmin$ and $\Cstab$.
Equivalently, we write this as  $B \gtrsim A$. We write $A \sim B$ when $A \lesssim B$ and $B\lesssim A$. 
The hidden constants in these expressions may depend on other constants such as $k_0$ and $\shapereg$, 
but we do not keep explicit track of these.
\end{notation}

\begin{definition}[{Operator POU} on the coarse subdomains] \label{def: 2_7}
For any  {$l\in\dof \Omega$}, let $\mu_l$ denote  the number of subdomains for which $l$ is an internal degree of freedom, i.e.\
\begin{equation}
	\mu_l \assign \#\bigl\{i \mid 1 \le i \le N,\, l \in \dof({\Omega}_i^c)\bigr\}{.}
\end{equation}
Then, for each $i = 1, \ldots, N$, we define the coarse space partition of unity operator, 
$\Xi_i^c: \widetilde{V}_{\Omega^c_i} \rightarrow V_{\Omega^c_i}$  
using  a weighted combination of degrees of freedom, namely, for $v=\sum_{l\in\ovdof(\Omega^c_i)} v_l\phi_l^i\in \widetilde{V}_{\Omega^c_i}$ we define
\begin{equation*}
	\Xi_i^c(v) \assign \sum_{l\in\dof(\Omega^c_i)}\frac{1}{\mu_l }v_l \phi_l^i, \qquad \text{where} \quad
	\phi_l^i \assign \phi_l\vert_{\Omega_i^c}. 
\end{equation*}             
\end{definition}

\noindent 
It is easy to verify that this is a partition of unity, in the sense that 
\begin{equation}\label{partition_of_unity}
	\sum_{i = 1}^N E_i^{c} \Xi_i^{c}({v|_{\Omega_i^c}}) = v \qquad \text{for all}  \;\; v \in V^h. 
\end{equation}

\begin{definition}[Local generalised eigenvalue problem (GEVP)] \label{defH_GenEO}
For each $i\in\{1,\ldots,N\}$, define $b_{\Omega_i^c}$ as in \eqref{eq:2.18} and set
\begin{equation} \label{Ivan14}  
	c_{\Omega_i^c}(w,v) \assign \bigl(\Xi_i^c(w),\Xi_i^c(v)\bigr)_{1,\kappa,\Omega_i^c} 
    \qquad \text{for all} \ w,v \in \widetilde{V}_{\Omega_i^c}.
\end{equation}  
The generalised eigenvalue problem is then to  
find $p^i\in \widetilde{V}_{\Omega_i^c}  \setminus\{0\}$ and $\lambda^i\in\mathbb{R}$, such that
\begin{equation}\label{eq: 5_12} 
	b_{\Omega_i^c}(p^i,v) = \lambda^i \, c_{\Omega_i^c} (p^i, v)  \qquad 
	\text{for all} \ v\in\widetilde{V}_{\Omega_i^c}. 
\end{equation}
\end{definition}

In order to define the $H_k$-GenEO coarse space we need an assumption on the spectra of \eqref{eq: 5_12}.

\begin{assumption} \label{ass:further} 
Let $(p^i_m,\lambda^i_m)$ be eigenpairs for \eqref{eq: 5_12} with $\lambda_1^i, \lambda_2^i, \ldots $ chosen in
non-decreasing order. 
For each  $i= 1,\ldots,N$,  we assume that there is $m_i\geq 1 $ such that $\lambda_{m_i+1}^i>0$ .
(A mild sufficient condition for the existence of such  $m_i$ is given in Theorem~\ref{pos_ev}.)
\end{assumption}

\begin{definition}[$H_\kappa$-GenEO Coarse space] \label{def: 5_1}
Under Assumption \ref{ass:further}, the coarse space, $V_0$, is given by 
\begin{equation}
	V_0 \assign \linspan \bigl\{E_i^c\Xi_i^c(p^i_m): \,  m = 1, \ldots, m_i\;\; \text{and} \;\; i = 1, \ldots, N\bigr\}.
\end{equation}
\end{definition}

\begin{remark}
In contrast to previous works on the GenEO construction (e.g.\ \cite{Spillane:2014:ARC},  \cite{Bootland:2022:OSM} and \cite{Bastian:2021:MSD}) 
(where the forms are always at least semi-definite), the left-hand side of \eqref{eq: 5_12} can here  be indefinite and also  
the right-hand side in \eqref{eq:  5_12}
is based on a $k$-weighted scalar product. 
However, as $k \rightarrow 0$ the GEVP  \eqref{Ivan14} approaches the GEVP solved in the classical GenEO method, e.g.\ \cite{Spillane:2014:ARC}.
\end{remark}

In the proof of our main result we also need the additional estimates for $\Vert \Xi_i^{{c}}(v)\Vert_{\Omega_i^c}$.

\begin{lemma} \label{Xi_below}
The following estimates hold:
\begin{alignat}{2}
	\Vert \Xi_i^{{c}} (v)\Vert_{\Omega_i^c} &\lesssim  \Vert v \Vert_{\Omega_i^c} \qquad 
	&& \text{for all} \;\; v \in \wV_{{\Omega_i^c}},
	\label{Xiest1}
	\\[1ex]
	\Vert \Xi_i^{{c}} (v)\Vert_{\Omega_i^c} &\gtrsim \Lambda_c^{-1} \Vert v \Vert_{\Omega_i^c} \qquad 
	&& \text{for all} \;\;  v \in V_{{\Omega_i^c}}. 
	\label{Xiest2} 
\end{alignat}
Here the hidden constants depend on $\shapereg$ but not on any of the key parameters (see Notation \ref{notation:sim}).
\end{lemma}

\begin{proof}\,
Due to the assumption that we are working with Lagrange elements,   
for any $T \in \scT_h$, there is an affine map $F_T$ mapping the unit simplex $\widehat{T}$ to $T$.
Then (see, e.g.\ \cite[Theorem~3.1.2]{Ciarlet:2002:FEM}), for any $w \in V^h$, setting $\widehat{w} = w \circ F_T$, we have   
\begin{align} \label{eqnorm}
	\Vert w \Vert_T^2 \ \sim\  h_T^{d} \Vert \widehat{w}\Vert_{\widehat{T}}^2 \ \sim \ h_T^{d} \sum_{j \in \ovdof(T)} w_j^2,
\end{align}
where $w_j$ is the value of $w$ at $j \in \dof(T)$ and the hidden  constants are independent of the mesh and $w$.   
(The second relation in \eqref{eqnorm} uses {($r$-dependent)} equivalence of norms on the space of polynomials of degree $r$.)
Hence, by Definition \ref{def: 2_7} of $\Xi_i^{{c}}(v)$ and since freedoms of $\Xi_i^c(v)$ are interior to $\Omega_i^c$, 
we have {for $v \in V_{\Omega_i^c}$,} 
\begin{align*}
	\|\Xi_i^{{c}}(v)\|_{\Omega_i^c}^2 
	&= \sum_{T \subseteq \overline{\Omega_i^c}}\|\Xi_i^{c}(v)\|_{T}^2 
	\sim \sum_{T \subseteq \overline{\Omega_i^c}} h_T^d \sum_{j \in \ovdof(T) \cap \dof(\Omega_i^c)}
	(\Xi_i^{{c}}(v))_j^2  
	= \sum_{T \subseteq \overline{\Omega_i^c}} h_T^d \sum_{j \in \ovdof(T) \cap \dof(\Omega_i^c)}
	\mu_j^{-2} v_j^2 
	\nonumber\\[0.5ex]
	&\ge \Lambda_c^{-2} \sum_{T \subseteq \overline{\Omega_i^c}} h_T^d \sum_ {j \in \ovdof(T) \cap \dof(\Omega_i^c)}  v_j^2 
	\sim  \Lambda_c^{-2}\sum_{T \subseteq \overline{\Omega_i^c}}\|v\|_{T}^2 = \Lambda_c^{-2} \|v\|_{\Omega_i^c}^2, 
\end{align*}
proving \eqref{Xiest2}.  The proof of \eqref{Xiest1} is analogous but simpler.  
\end{proof}

\section{Statement of the main result and technical tools}
\label{sec:theory}

We introduce the notation
\begin{align} 
  \Theta \assign \frac{1}{\min_{1 \le i \le N} \lambda^i_{m_i+1}} = {\frac{1}{\tau}}, \quad \text{where} 
  \ \tau \ \text{is defined in \eqref{def:tau}}. 
\label{notation:theta}
\end{align}

\begin{theorem}[GMRES convergence of the two-level preconditioned system] \label{theorem: convergence}
Let Assumptions~\ref{ass: 2_1}, \ref{Ass: 2_3}, \ref{ass:mesh}, \ref{ass:overlap} and \ref{ass:further} be satisfied
and let $\Cinv$, $\Ceig$, $\Lambda_l$, $\Lambda_c$ be as in Assumption~\ref{ass:mesh},
Theorem~\ref{est_below} and \eqref{Lambda_l_c}, respectively, and set
\begin{align}\label{Cstar}
	\Cstar \assign \Ceig \Cinvt.
\end{align}
Further, let $C_1$, $C_2$, $C_3$ be the constants (independent of the key parameters) derived in 
Lemma~\ref{PBP_PD}, Theorem~\ref{Lemma_3_1} and Lemma~\ref{lemma_3_3} respectively.
Assume that $k\ge k_0$.
Then there exists $h_{\star}>0$ such that the following statements hold for all $h\in(0,h_{\star})$ with $(kh)^2 \le C_1$. 
Suppose that
\begin{equation}
\label{eq: 4_3}
	s \assign 2{C_3}\bigl(1+\Lambda_\ell \Lambda_c^2 {\amax \nmin^{-1}} \Cstar \Theta \bigr)\Bigl({2 C_2\Lambda_c
	(\amax \nmin^{-1}\Cstar)^{1/2}k\Theta^{1/2}}(1+\Cstab)+3k\Lambda_\ell H_\ell\Bigr) < 1.
\end{equation}
Then, when GMRES is applied with the $\langle\cdot,\cdot\rangle_{\mathbf{D}_k}$-inner product with $\mathbf{D}_\kappa$ as in \eqref{def:Dk} 
to solve the preconditioned system given by \eqref{eq: 2_24}, then after $m$ iterations, the norm of the residual, $\mathbf{r}^{(m)}$, 
is bounded as follows:
\begin{equation}\label{eq: 4_22}
	\|\mathbf{r}^{(m)}\|_{\mathbf{D}_{\kappa}}^2 \le \bigl(1-\gamma^2\bigr)^m \|\mathbf{r}^{(0)}\|_{\mathbf{D}_{\kappa}}^2, 
\end{equation} 
where $\gamma$ is given by
\begin{equation}\label{def_gamma}
	\gamma \assign \frac{1-s}{C_3\bigl(1+\Lambda_\ell\Lambda_c^2\amax\nmin^{-1}\Cstar \Theta\bigr)(18 + 8 {\Lambda_{\ell}}^2)}.
\end{equation}
\end{theorem}

\begin{corollary} \label{conditions_corollary}
Under the same conditions as Theorem~\ref{theorem: convergence}, if \eqref{eq: 4_3} holds, then there exists $C>0$, which is independent of key  parameters such that
\begin{equation}\label{eq:main}
	k H_\ell \le C \qquad\text{and}\qquad (1+\Cstab)^2 k^2  \Theta \le C.
\end{equation}
Conversely, if \eqref{eq:main} holds for $C>0$ small enough so that
\begin{equation}\label{C_small}
	2C_3\bigl(1+\Lambda_\ell \Lambda_c^2 \amax\nmin^{-1}\Cstar C {k_0^{-2}} \bigr)
	\bigl({2 C_2\Lambda_c(\amax \nmin^{-1} \Cstar C)^{1/2}} +3 \Lambda_\ell  C\bigr) < 1,
\end{equation}
then \eqref{eq: 4_3} is satisfied and $\gamma$ from \eqref{def_gamma} is bounded below by a positive
number independent of $k$.
\end{corollary}

\begin{proof}
Assume first that \eqref{eq: 4_3} is satisfied.  
Since  $\Lambda_\ell\ge1$,  $\Lambda_c\ge1$, $\amax\ge1$, $\nmin \le 1$ and $\Cstar \ge 1$, we have
\[
	2C_2 k \Theta^{\frac{1}{2}} (1+\Cstab) + 3k H_\ell < \frac{1}{2C_3},
\]
and hence
\[
	k H_\ell < \frac{1}{6C_3} \qquad\text{and}\qquad
	(1+\Cstab)^2 k^{2}\Theta < \frac{1}{16C_2^2C_3^2},
\]
which yields \eqref{eq:main}.  

Conversely, assume that \eqref{eq:main} is satisfied with $C>0$ such that \eqref{C_small} holds.
Since, by assumption, ${k\ge k_0}$, we have
\[
	\Theta \le \frac{C}{(1+\Cstab)^2 k^{2}} \le C{k_0^{-2}}.
\]
Hence, using this with  \eqref{eq:main} and \eqref{C_small}, we have
\begin{align}
	s \le 2C_3 \bigl(1 + \Lambda_\ell \Lambda_c^2 \amax \nmin^{-1} \Cstar C k_0^{-2}\bigr)
	\bigl(2 C_2\Lambda_c(\amax \nmin^{-1} \Cstar C)^{1/2}  +3\Lambda_\ell C\bigr) < 1,
\label{ests}
\end{align}
i.e.\ \eqref{eq: 4_3} is satisfied.
        
Moreover, combining \eqref{def_gamma} with \eqref{ests},  we obtain a $k$-independent lower bound for $\gamma$. 
\end{proof}

In practice, conditions \eqref{eq:main} will introduce constraints in the size of the
local subdomains $\Omega_j^\ell$ and on the number of modes to be added in the coarse space, both
depending on $k$.

For the rest of the paper we will assume that Assumptions~\ref{ass: 2_1}, \ref{Ass: 2_3}, \ref{ass:mesh}, \ref{ass:overlap} and \ref{ass:further}
are  satisfied, although some of the intermediate results below require only a subset of these assumptions.

\subsection[Properties of the $H_k$-GenEO coarse space]{Properties of the {\boldmath$H_k$}-GenEO coarse space}
\label{subsec:coarse_space}

In this subsection we  apply the abstract spectral theory from  \S\ref{subsec:abstract}
to the generalised eigenvalue problem \eqref{eq:  5_12}.  To ease the exposition, we restrict here to a  
single coarse subdomain $\Omega_i^c$ with $i$ fixed and (in this subsection only) we use the following simplified notation.

\begin{notation}[Notation used in  \S \ref{subsec:coarse_space}]
\label{not:simple}
For a given $i$ we set
\begin{align*} 
	& \widetilde{V}_i \assign \widetilde{V}_{\Omega_i^c}, \quad 
	V_i \assign V_{\Omega_i^c}, \quad 
	b_i \assign {b}_{\Omega_i^c}, \quad 
	c_i \assign c_{\Omega_i^c},
	\\ 
	& n_i \assign \dim \widetilde{V}_i \quad \text{and} \quad 
	s_i \assign \dim \ker c_i. 
\end{align*}
\end{notation}

\noindent 
The next result identifies  $\ker c_i$ and shows that $b_i$ is positive definite on it, 
as required in Theorem~\ref{lem_ind_ev}.

\begin{lemma}[Positive definiteness of $b_i$ on the kernel of $c_i$] 
\label{PBP_PD}
We have 
\begin{equation}\label{ker_scC}
	\ker c_i = \mathrm{span} \{ \phi_l^i: l \in \ovdof(\Omega_i^c) \backslash \dof(\Omega_i^c)\}, 
	\quad \text{where} \quad \phi_l^i = \phi_l\vert_{\Omega_i^c},   
\end{equation}
and $\phi_l\in \tV_i$ is the nodal basis function $(\phi_l)_{l'} = \delta_{l,l'}$. 
Moreover, there exists a constant $C_1>0$ (which may depend on $\shapereg$ from Assumption \ref{ass:mesh}) such that,  
for all $k,h$ satisfying $(hk)^2 \le C_1$, we have 
\begin{align} \label{pdb1}
	b_i(v,v) \ge {C_1} h^{-2} \|v\|_{\Omega_i^c}^2   \qquad \text{for all} \;\; v \in \ker c_i.
\end{align}
\end{lemma}

\begin{proof}\, 
The hidden constants in this proof may depend on $\shapereg$.
To obtain \eqref{ker_scC}, note that, from Definition~\ref{def: 2_7}, if $v \in \wV_i$ with $v_l = 0$ 
for all $l \in \dof \Omega_i^c$, then $\Xi_i^c(v) \equiv 0$ and so $v \in \ker c_i$. 
Conversely, if {$v \in \ker c_i$}, then, by Assumption \ref{ass: 2_1},
\[
	0 = c_i(v,v) 
	= \int_{\Omega_i^c} \left({\nabla} {\Xi_i^c}(v)\cdot A \nabla {\Xi_i^c}(v) + k^2 {n} \,{(\Xi_i^c(v))^2} \right)
	\ge {\Vert \nabla \Xi_i^c(v)\Vert_{a,\Omega_i^c}^2 + k^2 \nmin\Vert \Xi_i^c(v) \Vert_{\Omega_i^c}^2},
\]
which implies $\Xi_i^c(v)\equiv 0$, and thus  $v_l = 0$ for all $l \in \dof \Omega_i^c$, hence proving \eqref{ker_scC}.    

To obtain \eqref{pdb1}, note first that, for any $v \in \wV_i$ (since $\amin=1 = \nmax$), 
\begin{equation}     
	b_{i}(v,v)
	= \int_{\Omega_i^c} \bigl(\nabla v \cdot (A\nabla v) - k^2 {n} v^2\bigr) 
	\ge \|\nabla v\|_{\Omega_i^c}^2 - k^2 \|v\|_{\Omega_i^c}^2.  
	\label{iggeq1} 
\end{equation}
Now let $v \in \ker c_i$ and note that $v$ is supported only on elements which touch $\partial \Omega_i$.
For any degree of freedom $l \in \partial \Omega_i^c$ let $\mathcal{T}_\ell^0$ denote the set of elements containing the freedom $l$
and let $\mathcal{T}_l$ denote the union of $\mathcal{T}_\ell^0$ with all elements which touch the elements in $\mathcal{T}_\ell^0$. 
Then let $h_\ell = \mathrm{diam}(\mathcal{T}_\ell)$, and let $B(l,h_l)$ denote the open ball of radius $h_l$ centred at $l$. 
Then it is easy to see that $\supp v \subseteq \bigcup_{l\in \partial \Omega_i^c}  B(l,h_l)$.
Moreover, for each $l$, $v$ vanishes on a portion of the boundary of  $B(l,h_l)$ having $d-1$ dimensional measure $\gtrsim h_l$. 
So, by the generalised Friedrichs' inequality (see, e.g.\ \cite[Lemma~A.15]{Toselli:2005:DDM}) and since $h_l \lesssim  h$,
we have
\begin{equation} \label{Fried}
	\Vert \nabla v \Vert_{B(l,2h_l)}^2  \gtrsim  {h_{l}}^{-2} \Vert v \Vert_{B(l,2h_l)}^2 \gtrsim   {h}^{-2}
	\Vert v \Vert_{B(l,2h_l)}^2.
\end{equation}
Also, by shape regularity, each $T \in \supp v$ can only be overlapped by a bounded number of balls $B(l,h_l)$ as $h \rightarrow 0$; 
so we have
\begin{align*}
	\Vert \nabla v \Vert_{\Omega_i^c}^2  &= \sum_{T\in \scT_h(v)} \Vert \nabla v \Vert_{T}^2 
	\sim \sum_{T\in \scT_h(v)} \sum_{l \in \partial \Omega_i^c}\Vert \nabla v \Vert_{T\cap B(l,h_l)}^2 
	= \sum_{l \in \partial \Omega_i^c} \sum_{T\in \scT_h(v)} \Vert \nabla v \Vert_{T\cap B(l,h_l)}^2  
	= \sum_{l \in \partial \Omega_i^c}  \Vert \nabla v \Vert_{B(l,h_l)}^2.
\end{align*}
Combining this with \eqref{Fried} we obtain
$ \Vert \nabla v \Vert_{\Omega_i^c}^2 \gtrsim h^{-2} \Vert v \Vert_{\Omega_i^c}^2 $.

Hence, from \eqref{iggeq1}, there exists a constant $C'>0$ (which may depend on $\shapereg$) such that
\[
	b_i (v,v) \ge (C' h^{-2} - k^2) \|v\|_{\Omega_i^c}^2 
	= (C' - (hk)^2) h^{-2} \|v\|_{\Omega_i^c}^2
\]
for all $h$ and $v \in \ker c_i$. Hence, if $(hk)^2 \le C_1\assign C'/2$, we have
$b_i (v,v) \ge C_1 h^{-2} \|v\|_{\Omega_i^c}^2$.
\end{proof}

\begin{remark} \label{rem:apply}
We can now apply the theory in \S \ref{subsec:abstract} with
$\txtb \assign b_{i}$ and $\txtc \assign c_{i}$. 
The corresponding eigenvalues \eqref{Ivan11} are denoted as follows
\begin{align} \label{Ivan11a}
	-\infty < \lambda_1^i \leq \lambda_2^i \leq \cdots \leq \lambda_{n_i-s_i}^i < \infty,
\end{align}
and the corresponding eigenfunctions are $\{p_m^i: m = 1, \ldots, n_i-s_i\} \subseteq \wV_i$.
The remaining eigenfunctions
$\{p_{n_i-s_i+1}^i, \ldots , p_{n_i}^i\}$ form a basis of $\ker c_i$.
The eigenfunctions satisfy the orthogonality properties corresponding to \eqref{orth_b} and \eqref{orth_c}.
To apply Proposition~\ref{lem_abs_proj}, we need to verify Assumption~\ref{ass:further}, i.e.\ 
show that at least one eigenvalue in \eqref{Ivan11a} is positive.
A sufficient condition for this is given next.
\end{remark}
 
\begin{theorem}[Positivity of at least one eigenvalue]\label{pos_ev}
Consider the GEVP of Definition \ref{defH_GenEO}.
For each $i\in\{1,\ldots,N\}$, assume that there exists a set $\mathcal{D}_i$ of
$s_i+1$ degrees of freedom in $\dof(\Omega_i^c)$ 
such that $\supp\phi_l^i\cap\supp\phi_{l'}^i$ has zero $d$-dimensional measure
for each pair $l,l' \in \mathcal{D}_i$ with $l\ne l'$.
Then there exists a constant $\Cpos>0$ such that, if $(hk)^2 \le \Cpos$, 
then, for each $i\in\{1,\ldots,N\}$, there is at least one index $m_i\in\{1,\ldots,n_i-s_i\}$ 
such that $\lambda_{m_i}^i > 0$.
\end{theorem}

\begin{remark}\label{rem:modest} 
The assumptions of Theorem~\ref{pos_ev} impose only a mild restriction.
For example, consider a square subdomain discretised by a uniform grid
$n_{dof}\times n_{dof}$, then $s_i \sim$ the number of boundary dofs $\sim n_{dof}$,
while the number of interior dofs is $\sim n_{dof}^2$.
For $n_{dof}$ large enough, a  set $\mathcal{D}_i$ can be chosen on  a
subgrid of size $\sim (n_{dof}/2)^2$.   
\end{remark}

\begin{proof}[Proof of Theorem~\ref{pos_ev}]
Let $i\in\{1,\ldots,N\}$.  Note that if  $l \in \dof(\Omega_i^c)$, then $\phi_l^{{i}} \in V_i$. 
Denoting the dofs in the assumption by {$l_1,\ldots,l_{s_i+1}$} we use Corollary~\ref{lem_v_abstract2} to write
\begin{align}\label{each_side}
	\phi_{{l_j}}^{{i}} = v_0^{i,{j}} + \sum_{m=1}^{n_i-s_i}\alpha_m^{i,{j}} p_m^i,
	\qquad j\in\{1,\ldots,s_i+1\},
\end{align}
with $v_0^{i,j}\in\ker c_i$ and $\alpha_m^{i,j}\in\mathbb R$.
Since $s_i=\dim\ker c_i$, there exist $\beta_1^i,\ldots,\beta_{s_i+1}^i\in\mathbb R$, not all vanishing, 
such that $\sum_{j=1}^{s_i+1}\beta_j^i v_0^{i,j}=0$.
Set $\psi^i\assign\sum_{{j=1}}^{s_i+1}\beta_{{j}}^i\phi_{{l_j}}^i$.  
Then multiplying each side of \eqref{each_side} by $\beta_j^i$ and summing over $j$ we obtain  
\begin{equation}\label{ml02}
	\psi^i = \sum_{m=1}^{n_i-s_i}\widetilde{\alpha}_m^i p_m^i \qquad \text{with} \;\; 
	\widetilde{\alpha}_m^i = \sum_{j=1}^{s_i+1} \beta_j \alpha_m^{i,j}.
\end{equation}
Further, set $d_l^i \assign \diam \supp\phi_l^i$.
For $j\in\{1,\ldots,s_i+1\}$, it follows from Friedrichs' inequality \eqref{eq:friedrichs} and the assumptions $\amin = \nmax = 1$ that 
\begin{align}
	b_{i}(\phi_{l_j}^{{i}},\phi_{l_j}^{{i}})
	\ge\|\phi_{l_j}^i\|_{a,\Omega_i^c}^2 - k^2\|\phi_{l_j}^i\|_{\Omega_i^c}^2
	\ge \Bigl(\frac{2}{(d_{l_j}^i)^2}-k^2\Bigr)\|\phi_{l_j}^i\|_{\Omega_i^c}^2 
    = \frac{1}{(h_{l_j}^i)^{2}}\biggl(2\Bigl(\frac{h_{l_j}^i}{d_{l_j}^i}\Bigr)^2-(kh_{l_j}^i )^2\biggr)\|\phi_l^i\|_{\Omega_i^c}^2.
\label{ml01}
\end{align}
where $h_{l_j}^i \assign \max\{ h_T: T \subseteq \supp \phi_{l_j}^i\}$.
The assumed {shape-regularity} of the mesh implies that ${h_{l_j}^i/d_{l_j}^i}$ is bounded below by a positive constant,
and we have $k h_{l_j}^i \le kh$; so, for small enough $hk$, the right-hand side of \eqref{ml01}
is strictly positive for all $j\in\{1,\ldots,s_i+1\}$.
Since $\supp\phi_{l_j}^i\cap\supp\phi_{l_{j'}}^i$ has measure zero for $j\ne j'$, 
we also have $b_{i}(\phi_{l_j}^i,\phi_{l_{j'}}^i)=0$
for $j\ne j'$ and hence $b_{i}(\psi^i,\psi^i) = \sum_{j=1}^{s_i+1} (\beta_j^i)^2 b_i(\phi_{l_j}^i,\phi^i_{l_j}) > 0$.

Now, if $\lambda_m^i\le0$ for all $m=1,\ldots,n_i-s_i$, then, by \eqref{ml02} and \eqref{orth_b}, 
\[
	b_{i}(\psi^i,\psi^i) = \sum_{m=1}^{n_i-s_i}\lambda_m^i (\widetilde{\alpha}_m^i)^2 \le 0,
\]
which is a contradiction.
\end{proof}

Our next main result is Theorem~\ref{est_below}, which gives an estimate from below 
for the minimum eigenvalue $\lambda_1^i$ of \eqref{Ivan11a}, needed in 
the proof of Theorem~\ref{Lemma_3_1}.  
To prepare for this, we choose a complement of $\ker c_i$ as follows:
\begin{equation}\label{choice_of_complement}
	\whV_i \assign \mathrm{span} \bigl\{\phi_l^i: l \in \dof(\Omega_i^c)\bigr\} = V_{i} \subseteq H^1_0(\Omega_i^c),
\end{equation}
which is $(n_i-s_i)$-dimensional and satisfies $\whV_i\cap\ker c_i=\{0\}$ by \eqref{ker_scC};
hence the decomposition
\begin{equation} \label{decomp}
	\wV_i = \whV_i \oplus \ker c_i
\end{equation}
holds, and the assumptions of Lemma~\ref{cpd} are fulfilled.

\begin{theorem}[Estimate of $\lambda_1^i$ from below] \label{est_below}
Let $(hk)^2 \le C_1$ with $C_1$ as in Lemma~\ref{PBP_PD}.  Then
\begin{equation} \label{eigest2}
	\lambda_1^i \gtrsim -\Ceig \qquad \text{for all} \;\; i\in\{1,\ldots,N\}
\end{equation} 
with 
\[
	\Ceig = \Lambda_c^{2}\nmin^{-1}(\amax\Cinvt+2)^2.
\]
\end{theorem}

\begin{proof}\,
We recall from equation \eqref{Ivan1e} that $\lambda_1^i$ is the minimum eigenvalue of the problem:
find $\lambda^i \in \mathbb{R}$ and $\whp^i \in \whV_i\backslash \{0\}$ such that
\begin{equation} \label{Ivan31_old} 
	\widetilde{b}_i (\whp^i, \whv^i) \assign b_i\bigl((I + \scS_i ) \whp^i, \whv^i\bigr) 
	= \lambda^i c_i(\whp^i, \whv^i) \qquad \text{for all} \;\;  \whv^i\in \whV_i,  
\end{equation}
where $\scS_i: \whV_i \rightarrow \ker c_i$ is the operator defined by
\begin{equation} \label{Ivan31aa}  
	b_i(\scS_i \whv^i, {w_{c_i}}) = - b_i(\whv^i, {w_{c_i}}) \qquad \text{for all} \;\; {w_{c_i}} \in \ker c_i.
\end{equation}
Using \eqref{Ivan31aa} and the fact that $\scS_i$ maps into $\ker c_i$, we have that \eqref{Ivan31_old} is equivalent to
\begin{equation} \label{Ivan31} 
	b_i\bigl((I + \scS_i ) \whp^i, (I + \scS_i) \whv^i\bigr)
	= \lambda^i c_i(\whp^i, \whv^i) \qquad \text{for all} \;\; \whv^i\in \whV_i.
\end{equation}
Since, by Lemma~\ref{cpd}, $c_i$ is positive definite on $\whV_i$, we obtain from Rayleigh's principle that
\begin{equation} \label{Rayleigh} 
	\lambda_1^i 
	= \min_{0 \ne {\whv^i} \in \whV_i}\frac{{b_i\bigl((I + \scS_i ) \whv^i, (I + \scS_i) \whv^i\bigr)}}{c_{i} (\whv^i,\whv^i)}.
\end{equation} 
To estimate \eqref{Rayleigh} from below, we first derive an estimate from above for $\Vert \scS_i\whv^i \Vert_{\Omega_i^c}$ 
for any $ \whv^i \in \whV_i$.  Using \eqref{pdb1}, the definition \eqref{Ivan31aa} of $\scS_i$, 
the fact that $\nmax = 1$ and \eqref{inv_est}, we have, 
for any $\whv^i \in \whV_i$,     
\begin{align}
	C_1 h^{-2} \Vert \scS_i \whv^i  \Vert_{\Omega_i^c}^2 
	&\le b_i(\scS_i \whv^i ,\scS_i \whv^i ) 
	= -b_i(\whv^i ,\scS_i \whv^i ) 
	\nonumber \\[0.5ex]
	&\le \Vert \nabla(\scS_i \whv^i  )\Vert_{\Omega_i^c} \Vert \nabla \whv^i \Vert_{\Omega_i^c} 
	+ k^2 \Vert \scS_i  \whv^i \Vert_{\Omega_i^c} \Vert   \whv^i \Vert_{\Omega_i^c} 
	\nonumber \\[0.5ex]
	& \le (\amax\Cinvt + 1)k^2 \Vert \scS_i \whv^i \Vert_{\Omega_i^c} \Vert  \whv^i  \Vert_{\Omega_i^c}, 
	\label{Ivan31a}
\end{align} 
and thus
\begin{equation} \label{Ivan32} 
	\|\scS_i\whv^i\|_{\Omega_i^c} \le (\amax\Cinvt + 1) C_1^{-1} (hk)^2 \Vert \whv^i \Vert_{\Omega_i^c} 
	\le (\amax\Cinvt + 1)\Vert \whv^i \Vert_{\Omega_i^c}. 
\end{equation}
As a result, for the numerator in~\eqref{Rayleigh}, we have
\begin{equation} \label{Ivan34}
	b_i\bigl((I+\scS_i)\whv^i , (I+\scS_i)\whv^i\bigr)
	\ge -k^2 \|(I+\scS_i)\whv^i\|_{\Omega_i^c}^2 
	\ge -k^2 (\amax\Cinvt + 2)^2 \|\whv^i\|_{\Omega_i^c}^2.
\end{equation} 
For the denominator note that $\whV_i=V_i$ by \eqref{choice_of_complement}; hence Lemma~\ref{Xi_below} implies that
\begin{equation} 
	c_i(\whv^i, \whv^i) \ge k^2 \nmin\Vert \Xi_i^c( \whv^i )\Vert_{\Omega_i^c}^2  
	\gtrsim  k^2 \nmin\Lambda_c^{-2} \Vert \whv^i \Vert_{\Omega_i^c}^2.
	\label{Ivan35}
\end{equation}
For the rest of this proof (only) let $\mathbf{min}$ denote the  minimum of any quantity over all
$0 \ne {\whw^i} \in \whV_i $.  Recall also that $c_i$ is positive definite on $\whV_i$. 
If $\bfmin\widetilde{b}_i(\whw^i,\whw^i)\ge0$, then $\lambda_1^i\ge0$ and \eqref{eigest2} is trivially true.
Now, assume that $\bfmin \widetilde{b}_i(\whw^i,\whw^i) < 0$.  Then, for all $0 \ne \whv^i \in \whV_i$,   
\[
	\frac{\widetilde{b}_i(\whv^i, \whv^i)}{c_i(\whv^i, \whv^i)} 
	\ge \frac{\bfmin \widetilde{b}_i(\whw^i, \whw^i)}{c_i(\whv^i, \whv^i)} 
	\ge \frac{\bfmin \widetilde{b}_i(\whw^i, \whw^i)}{\bfmin{c_i(\whw^i, \whw^i)}}
	= \frac{\bfmin {b}_i((I + \scS_i)\whw^i, (I + \scS_i)\whw^i)}{\bfmin{c_i(\whw^i, \whw^i)}},
\]
and so, by \eqref{Ivan34}, \eqref{Ivan35}, for all $0 \ne \whv^i \in \whV_i$, 
\[
	\frac{ \widetilde{b}_i( \whv^i, \whv^i )}{c_i(\whv^i, \whv^i)} 
	\gtrsim - \Lambda_c^{2}\nmin^{-1}(\amax\Cinvt+2)^2,
\]
which proves \eqref{eigest2} also in this case.
\end{proof}

The next theorem is the fundamental result describing the approximation power of the $H_k$-GenEO space.  
It constitutes a substantial generalisation of previous results in \cite[Lemma~3.1]{Bootland:2022:OSM}, 
\cite[Lemma~2.11]{Spillane:2014:ARC} and \cite[Lemma~3.14]{Bastian:2021:MSD} since here we study the indefinite 
GEVP \eqref{eq:  5_12}.

\begin{theorem}
\label{Lemma_3_1}
Suppose that $(kh)^2 \le C_1$ with $C_1$ as in Lemma~\ref{PBP_PD}.
Recalling Assumption \ref{ass:further}, define the 'local projector' $\Pi_{m_i}$ by
\begin{equation}
\label{eq:proj}
	\Pi_{m_i} v \assign \sum_{m=1}^{m_i} c_i(v, p^i_m)  p_m^i, \qquad v \in \wV_i. 
\end{equation}
Let $\Cstar$ be as in \eqref{Cstar}.  There exists a constant $C_2>0$, which is independent of all key parameters,
such that, for any $v \in \wV_i$ and $w \assign v-\Pi_{m_i} v$, 
\begin{align}
	& 0 \le b_{i}(w,w) \le C_2^2 \amax \nmin^{-1} \Cstar \Vert v \Vert_{1,k,\Omega_i^c}^2,
	\label{new_1}
	\\[0.5ex]
	& c_{i}(w,w) \le (\lambda_{m_i+1}^i)^{-1} b_{i}(w,w).
	\label{new_2}
\end{align}
\end{theorem}

\begin{proof}\, 
The inequality in \eqref{new_2} follows directly from \eqref{Ivan7} of Proposition~\ref{lem_abs_proj}.
Set $\alpha_m^i \assign c_{i}(v,p_m^i)$.
Then $w = v_0 + \sum_{m = m_i+1}^{n_i-s_i} \alpha_m^i p_m^i$ 
with $v_0 \in \linspan\{ p_{n_i-s_1+1}\ldots , p_{n_i}\} = \ker c_i$.  Hence using property \eqref{orth_b},
then Lemma~\ref{PBP_PD} and the fact that $\lambda_m^i > 0$ for all $m \in \{m_i +1, \ldots n_i-s_i\}$, we obtain  
\begin{align*}
	b_{i} (w,w) &= b_{i} (v_0,v_0) + 2 \sum_{m= m_i+1}^{n_i-s_i} \alpha_m^i b_{i} (p_m^i,v_0) 
	+ \sum_{m,m'= m_i+1}^{n_i-s_i} \alpha_m^i \alpha_{m'}^i  b_{i} (p_m^i,p_{m'}^i)
	\\
	&= b_{i} (v_0,v_0) +  \sum_{m = m_i+1}^{n_i-s_i} (\alpha_m^i)^2 \lambda_m^i 
	\ge  0,
\end{align*}   
i.e.\ the left-hand inequality in \eqref{new_1}. 
To obtain the {right-hand}  inequality in \eqref{new_1}, note that $b_{i}(w,\Pi_{m_i}v) = 0$, and hence 
\begin{equation}\label{Lem341}
	b_{i}\!\left( w , w \right)  =  b_{i}(v,v)  - b_{i}( \Pi_{m_i}  v,\Pi_{m_i}  v) 
	= a_{\Omega_i^c}(v,v) - k^2 \|v\|_{n,\Omega_i^c}^2 - b_{i}( \Pi_{m_i}  v, \Pi_{m_i}  v). 
\end{equation}
To estimate the last term in \eqref{Lem341} we use \eqref{orth_b}, \eqref{Ivan11a},   
\eqref{eigest2} and then \eqref{orth_c} to obtain     
\begin{align}
	-b_{i}\!\left( \Pi_{m_i} v, \Pi_{m_i} v \right) 
	&= -  \sum_{m=1}^{m_i} (\alpha_m^i)^2 \, \lambda_m^i 
	\le -\lambda_1^i \sum_{m=1}^{m_i}(\alpha_m^i)^2
	\nonumber\\
	&\lesssim \Ceig \sum_{m=1}^{m_i} (\alpha_m^i)^2
	\le \Ceig\sum_{m=1}^{n_i-s_i} (\alpha_m^i)^2 
	= \Ceig c_{i}(v,v). 
	\label{Lem342}
\end{align}
Hence combining \eqref{Lem341} and \eqref{Lem342} we obtain 
\begin{equation}
\label{Lem343}
	b_i(w,w) \lesssim a_{\Omega_i^c}(v,v) + \Ceig\, c_i(v,v).
\end{equation}
To conclude, we now use \eqref{inv_est}, the fact that $\amax \ge 1$ and $\Cinv\ge 1$,
and \eqref{Xiest1} to obtain
\begin{align}
	c_i(v,v)  &\le \amax \|\nabla \Xi_i^c(v)\|_{\Omega_i^c}^2 + k^2 \|\Xi_i^c(v)\|_{\Omega_i^c}^2
	\nonumber\\
	&\le (\amax\Cinvt + 1)k^2 \Vert \Xi_i^c(v)\Vert_{\Omega_i^c}^2 
	\lesssim (\amax  \Cinvt    + 1 ) k^2\Vert v \Vert_{\Omega_i^c}^2 \lesssim \amax \Cinvt k^2  \Vert v \Vert_{\Omega_i^c}^2,  \label{hidden}
\end{align}
and the proof follows on combining \eqref{hidden} with \eqref{Lem343}, using $\Ceig\ge1$ and recalling Notation~\ref{not:forms}.
\end{proof}
 
From Theorem~\ref{Lemma_3_1} it is now possible to build a global approximation property.

\begin{lemma}[Global approximation property]
\label{Lemma_3_2}
Under the same conditions as in Theorem~\ref{Lemma_3_1}, define $\Theta$ as in \eqref{notation:theta}
and let $\Cstar$ and $C_2$ be as in \eqref{Cstar} and Theorem~\ref{Lemma_3_1} respectively.
Further, let $v \in V^h$ and {set}
\begin{equation}\label{def_z_0_z_i}
	z_0 \assign \sum_{i=1}^N E_i^c \Xi^c_i \bigl(\Pi_{m_i} v|_{\Omega_i^c}\bigr).
\end{equation}
Then, $z_0 \in  V_0$ and
\begin{equation}\label{inequ_v_z_0}
	\inf_{z\in V_0}\|v-z\|_{1,k}^2 \le \|v-z_0\|_{1,k}^2 \le C_2^2\,\Lambda_c^2 \amax \nmin^{-1} \Cstar
	\Theta\|v\|^2_{1,k}. 
\end{equation}
\end{lemma}

\begin{proof}\,
Since $\Pi_{m_i}^c v|_{\Omega_i^c}\in\linspan\{p_1^i,\ldots,p_{m_i}^i\}$, it is clear that $z_0 \in V_0$,
which also proves the first inequality in \eqref{inequ_v_z_0}.
Using \eqref{partition_of_unity}, \eqref{overlap2}, \eqref{new_2} and  \eqref{new_1} we obtain
\begin{align}
	\|v-z_0\|_{1,k}^2  
	&= \Bigg\| \sum_{i=1}^N  E_i^c (\Xi_i^c (v|_{\Omega_i^c}))-\sum_{i=1}^N  E^c_i \Xi^c_i \bigl(\Pi_{m_i} {v|_{\Omega_i^c}}\bigr) \Bigg\|_{1,k}^2
	\\ 
	&\le \Lambda_c \sum_{i=1}^N \left \|\Xi_i^c \left(v|_{\Omega_i^c}- \Pi_{m_i} {v|_{\Omega_i^c}} \right) \right\|_{1,k,{\Omega_i^c}}^2 \nonumber  = \Lambda_c \sum_{i=1}^N  c_{i}\Bigl(  \bigl({v|_{\Omega_i^c}}-\Pi_{m_i}{v|_{\Omega_i^c}}\bigr), \bigl({v|_{\Omega_i^c}}-\Pi_{m_i}{v|_{\Omega_i^c}}\bigr)\Bigr)
	\nonumber \\
	&\le \Lambda_c \sum_{i=1}^N \frac{1}{\lambda_{m_i+1}^i} b_{i}\Bigl(\bigl({v|_{\Omega_i^c}}-\Pi_{m_i}{v|_{\Omega_i^c}}\bigr),\bigl({v|_{\Omega_i^c}}-\Pi_{m_i}{v|_{\Omega_i^c}}\bigr)\Bigr) 
	\nonumber \\
	&\le C_2^2 \Lambda_c \amax \nmin^{-1}\Cstar \sum_{i=1}^N\frac{1}{\lambda_{m_i+1}^i}\big\|v|_{\Omega_i^c}\big\|_{1,k,\Omega_i^c}^2  
	\le C_2^2 \Lambda_c \amax \nmin^{-1}\Cstar \Theta\sum_{i=1}^N\big\|v|_{\Omega_i^c}\big\|_{1,k,\Omega_i^c}^2, 
	\label{proof_inequ_sum_z_i}
\end{align}
and the result follows from \eqref{overlap}.
\end{proof}

The next lemma shows that the spaces $V_0$ and $V_{\Omega_j^\ell}$ provide a stable decomposition of $V^h$.

\begin{lemma}[Stable decomposition]
\label{lemma_3_3}
Under the same conditions as in Theorem~\ref{Lemma_3_1}, let $v \in V^h$ and define $z_0 \in V_0$ as in~\eqref{def_z_0_z_i}. 
Then there exists a constant $C_3\ge1$, which is independent of all parameters, such that,
for each $j = 1,\ldots,Q$, there exist $z_j \in V_{\Omega_j^\ell}$ so that
\[
	v = z_0 + \sum_{j=1}^Q E_j^\ell z_j^\ell \quad \text{ and } \quad \|z_0\|_{1,k}^2 + \sum_{j=1}^Q \|z_j^\ell\|_{1,k,\Omega^\ell_j}^2 
	\le C_3 \bigl(1 + \Lambda_\ell \Lambda_c^2 \amax \nmin^{-1} \Cstar \Theta \bigr) \| v \|_{1,k}^2.
\]
\end{lemma}

\begin{proof}\,
We choose a partition of unity consisting of real-valued Lipschitz functions $\{\Phi_j^{\ell}\}_{j=1}^Q$ on $\Omega$  
satisfying $0 \le \Phi_j^{\ell} \le 1$, $\supp(\Phi_j^{\ell}) \subseteq \overline{\Omega_j^\ell}$ for each $j$,
and $\sum_{j=1}^Q \Phi_j^{\ell} \equiv 1$ on $\Omega$; see, e.g.\ \cite[\S 3.2]{Toselli:2005:DDM}. 
Then $\|\nabla \Phi_j^{\ell}\|_{L^\infty(\Omega)} \le \CPOU \delta_\ell^{-1}$ for all $j$,
where $\delta_\ell$ is the overlap parameter from Assumption~\ref{ass:overlap}. 
By \cite[estimate (3.2)]{Graham:2020:DDI} and then Assumption \ref{ass:overlap}, for any $v \in H^1_0(\Omega)$,
\begin{equation}\label{multPOU}
	\sum_{j=1}^Q \| \Phi_j^{\ell} v \|_{1,k,\Omega_j^\ell}^2 
	\lesssim \Lambda_\ell \left(1 + \frac{1}{(k{\delta_\ell})^2} \right) \|v\|_{1,k,\Omega}^2 \lesssim \Lambda_\ell  \|v\|_{1,k,\Omega}^2,
\end{equation}

Now define $z_j^\ell \assign \Phi_j^\ell(v - z_0) \in V_j^\ell$.  Then 
\[
	\sum_{j=1}^Q E_j^\ell z_j^\ell = \sum_{j=1}^Q E_j^\ell \Phi_j^\ell (v - z_0) = (v - z_0)\sum_{j=1}^Q {\Phi_j^\ell} = v - z_0. 
\]
Moreover, from Lemma~\ref{Lemma_3_2} we have 
\[
	\|z_0\|_{1,k}^2 \le 2\left(\|z_0 - v\|_{1,k}^2 + \|v\|_{1,k}^2\right) \lesssim
	(1 + \Lambda_c^2 \amax \nmin^{-1} \Cstar \Theta) \|v\|_{1,k}^2.
\]
Also, using \eqref{multPOU}, we obtain
\[
	\sum_{j=1}^Q \|z_j^\ell\|_{1,k,\Omega_j^\ell}^2 
	= \sum_{j=1}^Q \| \Phi_j^\ell (v - z_0)\|_{1,k,\Omega_j^\ell}^2 
	\lesssim \Lambda_\ell \|v - z_0\|_{1,k}^2.
\]
Combining this with Lemma~\ref{Lemma_3_2} and the bound for \(z_0\) we obtain the result.
\end{proof}

It is convenient to introduce here the orthogonal projections onto the local spaces ${P_j^\ell}: V^h \rightarrow V^\ell_j$, 
with respect to the inner product $(\cdot, \cdot)_{1,k}$, by requiring that
\begin{equation}
	(P_j^\ell v, w)_{1,k,\Omega^\ell_j} = (v, E_j^\ell w)_{1,k,\Omega} \qquad \text{for all}
	\;\; v \in V^h, \;\; w \in V_{\Omega_j^\ell},  \quad \text{and} \;\; j = 1, \ldots, Q 
	\label{eq: 3_10_nested}.
\end{equation}
Similarly, we  define the orthogonal projection $P_0 : V^h \rightarrow V_0$ by requiring 
\begin{equation}
	(P_0 v, w)_{1, k} = (v, w)_{1,k} =(v, E_0w)_{1,k} \qquad \text{for all} \;\; v \in V^h, \;\; w \in V_0. 
	\label{eq: 3_10_zero}
\end{equation}
Using these we then  define the operator $P: V^h \rightarrow V^h$ by
\begin{equation}
	P \assign E_0P_0 + \sum_{j=1}^{Q} {E_j^\ell} {P_j^\ell}. \label{eq: 3_11}
\end{equation} 
The operator $P$ is an (easier-to-analyse) proxy for the operator $T$ 
(which represents our preconditioned matrix --- see \eqref{eq: 2_26} and  Proposition \ref{prop: 2_6}). 
Here, we estimate the field of values of $P$, and then use this to estimate  the field of values of  $T$, 
leading to our main  result (Theorem~\ref{theorem: convergence}).

\begin{proposition}[Field of values of $P$] \label{prop:Puu}
Under the same assumptions as in Theorem~\ref{Lemma_3_1}, let $C_3$ be as in Lemma~\ref{lemma_3_3}.
Then 
\begin{equation}\label{eq: 3_12}
	\|v\|^2_{1,k} \le C_3\bigl(1 + \Lambda_\ell \Lambda_c^2 \amax \nmin^{-1} \Cstar \Theta \bigr) (Pv,v)_{1,k}
\end{equation}
for every $v\in V^h$.
\end{proposition}

\begin{proof}\,
Using Lemma \ref{lemma_3_3}, the definitions  of $P_j^\ell$ and $P_0$ and then Cauchy--Schwarz,  we get
\begin{align*}
	\|v\|_{1,k}^2 &= \Biggl(v, z_0 + \sum_{j=1}^{Q} {E_j^\ell} {z_j^\ell}\Biggr)_{1,k} 
	= (v, z_0)_{1,k} + \Biggl(v,\sum_{j=1}^{Q} {E_j^\ell} {z_j^\ell}\Biggr)_{1,k} \\ 
    &= (P_0v, z_0)_{1,k} + \sum_{j=1}^{Q}( {P_j^\ell} v, {z_j^\ell})_{1,k, {\Omega^\ell_j}} 
    \le \| P_0v\|_{1,k} \|z_0\|_{1,k} + \sum_{j=1}^{Q}\| {P_j^\ell} v \|_{1,k, {\Omega^\ell_j}} \|{z_j^\ell}\|_{1,k, {\Omega^\ell_j}}{.} 
\end{align*}
We now apply  Cauchy--Schwarz and  Lemma~\ref{lemma_3_3} again to obtain
\begin{align*}
	\|v\|_{1,k}^2
	&\le \left(\| P_0v\|_{1,k}^2 + \sum_{j=1}^{Q}\| {P_j^\ell} v \|_{1,k, {\Omega^\ell_j}}^2 \right)^{1/2}
	C_3^{1/2}\bigl(1 + \Lambda_\ell \Lambda_c^2 \amax \nmin^{-1} \Cstar \Theta \bigr)^{1/2} \|v\|_{1,k} \\[1ex]
	&= \left(( E_0 P_0v, v)_{1,k} + \sum_{j=1}^{Q} ( {E_j^\ell} {P_j^\ell} v, v)_{1,k} \right)^{1/2}  
	C_3^{1/2}\bigl(1 + \Lambda_\ell \Lambda_c^2 \amax \nmin^{-1} \Cstar \Theta \bigr)^{1/2}\|v\|_{1,k} \\[1ex]
	&= \left( E_0 P_0 v + \sum_{j=1}^{Q}{E_j^\ell} {P_j^\ell} v, v \right)_{1,k}^{1/2}
	C_3^{1/2}\bigl(1 + \Lambda_\ell \Lambda_c^2 \amax \nmin^{-1} \Cstar \Theta \bigr)^{1/2} \|v\|_{1,k} \\
	&= (P v,v)_{1,k}^{1/2} \, C_3^{1/2}\bigl(1 + \Lambda_\ell \Lambda_c^2 \amax \nmin^{-1} \Cstar \Theta \bigr)^{1/2}\|v\|_{1,k},
\end{align*}
and the inequality \eqref{eq: 3_12} follows. 
\end{proof}

We now study the fundamental properties of the operators $T_j^\ell$ and $T_0$.

\subsection[Existence and stability of ${T_j^\ell}$, $j=1,\ldots,Q$]{Existence and stability of {\boldmath${T_j^\ell}$}, {\boldmath$j=1,\ldots,Q$}}

\begin{lemma}[$T_j^\ell$ are well defined]
\label{lemma_3_4}
If ${H_\ell} k < \sqrt{2}$, then, for each $j=1, \ldots, Q$,  $b_{{\Omega^\ell_j}}(\cdot,\cdot)$ is positive definite
on $V_{\Omega_j^\ell}$ and the operators ${T_j^\ell}$ in \eqref{eq: bT} are well defined.
\end{lemma}

\begin{proof}
Using Friedrichs' inequality \eqref{eq:friedrichs} and the assumption that $\amin = 1 = \nmax$, we obtain, for $u\in V_{\Omega_j^\ell}$
\begin{equation*}
	b_{{\Omega^\ell_j}}(u,u) = a_{{\Omega^\ell_j}}(u,u) - k^2({n}u,u)_{{\Omega^\ell_j}} 
	\ge \frac{2}{{{(H_\ell)}}^2}\|u\|_{{\Omega^\ell_j}}^2 - k^2\|u\|_{{\Omega^\ell_j}}^2 
	= \frac{2-{{(H_\ell)}}^2k^2}{{{(H_\ell)}}^2}\|u\|_{{\Omega^\ell_j}}^2,
\end{equation*}
and hence $b_{{\Omega^\ell_j}}(\cdot,\cdot)$ is positive definite when $H_\ell k < \sqrt{2}$. 
The square system \eqref{eq: bT} defining $T_j^\ell v$ is then positive definite and the result follows.
\end{proof}

\begin{remark}
Whilst this is a sufficient condition for the ${T_j^\ell}$ to be well defined, it is
far from necessary since the ${T_j^\ell}$ will be well defined when the linear system corresponding to \eqref{eq: bT} is non-singular. 
However, it seems hard to guarantee good stability properties for $T_j^\ell$ without the above condition.
\end{remark}

\begin{lemma}[Stability of ${T_j^\ell}$, $j=1,\ldots,Q$]
\label{lemma_3_6}
Suppose that $\sqrt{2} {H_\ell}\kappa \le 1$.  Then
\begin{equation}\label{eq: 3_31}
	\|{T_j^\ell} v\|_{1,k,{\Omega^\ell_j}} \le 2\big\|v|_{{\Omega^\ell_j}}\big\|_{1,k,{\Omega^\ell_j}} 
	\qquad \text{for all} \;\; v \in V^h.
\end{equation}
\end{lemma}

\begin{proof}\,
Using \eqref{eq: bT}, \eqref{lem:est_b}, the fact that $\nmax = 1 = \amin$, and Friedrichs' inequality \eqref{eq:friedrichs}, we obtain
\begin{align*}
	\|{T_j^\ell} v\|^2_{1,k,{\Omega^\ell_j}} 
	&= b_{{\Omega^\ell_j}}({T_j^\ell} v,{T_j^\ell} v) + 2k^2\|T_j^\ell v\|_{n,\Omega_i^\ell}^2 
	= b_{\Omega_i^\ell}\bigl(v|_{\Omega^\ell_j},{T_j^\ell} v\bigr) + 2k^2\|T_j^\ell v\|_{n,\Omega^\ell_j}^2
	\\
	&\le \big\|v|_{\Omega^\ell_j}\big\|_{1,k,\Omega^\ell_j}\|T_j^\ell v\|_{1,k,\Omega^\ell_j} 
	+ 2k^2\|T_j^\ell v\|_{\Omega^\ell_j}^2
	\\
	&\le \big\|v|_{\Omega^\ell_j}\big\|_{1,k,\Omega^\ell_j}\|{T_j^\ell} v\|_{1,k,\Omega^\ell_j} 
	+ k^2(H_\ell)^2\|T_j^\ell v\|^2_{1,k,\Omega^\ell_j},
\end{align*}
which implies 
$\bigl(1-k^2 (H_\ell)^2\bigr)\|T_j^\ell v\|_{1,k,\Omega^\ell_j} \le \big\|v|_{\Omega^\ell_j}\big\|_{1,k,\Omega^\ell_j}$, 
and the result follows.
\end{proof}

\subsection[Existence and stability of $T_0$]{Existence and stability of {\boldmath$T_0$}}

\begin{lemma}[$T_0$ is well-defined] \label{lemma_3_5}
Let $C_2>0$ be as in Theorem~\ref{Lemma_3_1}.  If  
\begin{equation}\label{3_7_a}
	C_2 \Lambda_c (\amax \nmin^{-1} \Cstar)^{1/2} k\Theta^{1/2} (1 + \Cstab) < \frac{1}{\sqrt{2}\,}, 
\end{equation}
then there exists $h_1 > 0$ such that, for all $h\in(0,h_1)$, the operator $T_0$ is well defined.
\end{lemma}

\begin{proof}\,
Let $C$ and $\Theta$ be as in Lemma~\ref{schatz_wang}.
Then there exists $h_1>0$ such that
\[
	C\,\Theta\Bigl(\frac{h}{r}\Bigr) \le \min\biggl\{\frac{1}{2},\frac{k_0\nmin^{1/2}}{2}\biggr\}
	\qquad \text{for all} \;\; h\in(0,h_1).
\]
Let $h\in(0,h_1)$ and
assume (for a contradiction) that there exists a $w_0 \in V_0\backslash \lbrace 0 \rbrace {\subseteq V^h}$ such that
\begin{equation}\label{eq:w_0_sol}
	b(w_0,z) = 0 \qquad \text{for all} \;\; z \in V_0.
\end{equation}
Let $w \in H^1_0(\Omega)$ be the solution of $b(w,v)=(w_0,v)$ for all $v \in H^1_0(\Omega)$, 
which exists and is unique by Assumption~\ref{Ass: 2_3}.
It follows from Lemma~\ref{schatz_wang} that $\beta\le\min\bigl\{1/2,k_0\nmin^{1/2}/2\bigr\}<1/\sqrt{2}$
(where $\beta$ is defined as in Lemma~\ref{schatz_wang})
and there exists $w_h \in V^h$ such that
\begin{equation}
	b(w_h, v) = (w_0,v) \quad \text{for all} \;\; v \in V^h, \quad \text{and} \quad 
	\|w - w_h\|_{1,k} \le \frac{2 \beta}{k \nmin^{1/2}} \|w_0\| \le \|w_0\|.   
	\label{ScottWang}
\end{equation}
Let $z\in V_0$ be arbitrary.
Inserting $v=w_0$ in \eqref{ScottWang} and combining with \eqref{eq:w_0_sol} and then \eqref{lem:est_b}, we obtain
\begin{equation*}
	\|w_0\|^2 = b(w_h, w_0) = b(w_h - z, w_0)
	\le \|w_h - z\|_{1,k} \|w_0\|_{1,k} \qquad\text{for all} \;\; z \in V_0.
\end{equation*}
Since this is true for all $z \in V_0$, Lemma~\ref{Lemma_3_2} yields  
\begin{equation}\label{eq: 3_5_f_sq}
	\|w_0\|^2 \le \inf_{z \in V_0}\|w_h-z\|_{1,k} \|w_0\|_{1,k}
	\le C_2\, \Lambda_c (\amax \nmin^{-1} \Cstar)^{1/2}
	\Theta^{\frac{1}{2}}\|w_h\|_{1,k} \, \|w_0\|_{1,k}.
\end{equation}
Also, equation \eqref{eq:w_0_sol} with $z=w_0$ shows that
$0 = b(w_0,w_0) = \|w_0\|_{1,k}^2-2k^2\|w_0\|_n^2$, which, together with \eqref{eq: 3_5_f_sq}, implies
\begin{align*}
	\|w_0\|^2 &\le C_2 \Lambda_c (\amax \nmin^{-1} \Cstar)^{1/2} \Theta^{\frac{1}{2}} \|w_h\|_{1,k}\sqrt{2}\, k \|w_0\|_n
	\\
	&\le \sqrt{2}\, C_2 \Lambda_c (\amax \nmin^{-1} \Cstar)^{1/2} k \Theta^{\frac{1}{2}} \|w_h\|_{1,k} \|w_0\|,
\end{align*}
and hence
\begin{equation}\label{eq: 3_5_1} 
	\|w_0\| \le \sqrt{2}\, C_2 \Lambda_c (\amax \nmin^{-1} \Cstar)^{1/2} k \Theta^{\frac{1}{2}} \|w_h\|_{1,k}. 
\end{equation}
Then we use the triangle inequality, \eqref{ScottWang} and \eqref{eq: 2_10} to obtain
\begin{equation}
	\|w_h\|_{1,k} \le \|w-w_h\|_{1,k} + \|w\|_{1,k}  \le (1 + \Cstab )\|w_0\|.
	\label{eq: w_h_comb}
\end{equation}
With \eqref{eq: w_h_comb} used in \eqref{eq: 3_5_1}, we arrive at
\begin{equation*}
	\|w_0\| \le  \sqrt{2}\, C_2 \Lambda_c (\amax \nmin^{-1} \Cstar)^{1/2} k \Theta^{\frac{1}{2}} (1+\Cstab)\|w_0\|.
\end{equation*}
Together with \eqref{3_7_a} this leads to a contradiction.
\end{proof}

\begin{lemma}[Stability of $T_0$]
\label{lemma: 3_7}
With $C_2$ as in Theorem~\ref{Lemma_3_1}, assume that \eqref{3_7_a} is satisfied.  Then there exists $h_1>0$ such that, for $h\in(0,h_1)$, 
\begin{equation}
	\|T_0 u - u\| \le {C_2}\, \Lambda_c (\amax \nmin^{-1} \Cstar)^{1/2}
	\Theta^{1/2} (1 + \Cstab)\|T_0 u - u\|_{1,k} \qquad \text{for all} \ u \in V^h.
	\label{3_7_b}
\end{equation}
Suppose, in addition, that 
\begin{equation}\label{3_7_c}
	C_2\, \Lambda_c (\amax \nmin^{-1} \Cstar)^{1/2} k \Theta^{1/2} (1 + \Cstab)
	\le \frac{1}{4};
\end{equation} 
then 
\begin{equation}\label{3_7_d}
	\|u - T_0 u\|_{1,k} \le 2 \|u\|_{1,k} \qquad \text{for all} \ u \in V^h.
\end{equation}
\end{lemma}

\begin{proof}\,
Let $h_1>0$ be as in Lemma~\ref{lemma_3_5}.
Further, let $u \in V^h$ and consider the auxiliary problem: 
\begin{equation}\label{eq 3_7_4}
	\text{find } w_h \in V^h \;\; \text{such that} \quad b(w_h, v) = \bigl(T_0 u - u,v\bigr) \quad \text{for all} \;\; v \in V^h. 
\end{equation}
Analogously to \eqref{eq: w_h_comb}, problem \eqref{eq 3_7_4} has a unique solution $w_h$  for $h\in(0,h_1)$, and 
\begin{equation} \label{estwh} 
	\|w_h\|_{1,k} \le (1+\Cstab) \|T_0 u - u\|. 
\end{equation}
Now, by the definition \eqref{eq: 3_10_zero} of $T_0$, we have $b(T_0 u-u,z)=0$ for all $z \in V_0$.
Inserting $v = T_0 u - u \in V^h$ into \eqref{eq 3_7_4} and using \eqref{lem:est_b}, we obtain, for every $z\in V_0$,
\begin{equation*}
	\|T_0 u - u\|^2 = b(w_h, T_0 u - u)
	= b(w_h - z, T_0 u - u)
	\le \|w_h-z\|_{1,k}\,\|T_0 u - u\|_{1,k}.
\end{equation*}
Combining this with Lemma~\ref{Lemma_3_2} we obtain
\begin{equation}\label{ineq:T0u-u}
	\|T_0 u - u\|^2 \le \|T_0 u - u\|_{1,k}\inf_{z \in V_0}\|w_h-z\|_{1,k} 
	\le \|T_0 u - u\|_{1,k} {C_2}\, \Lambda_c (\amax \nmin^{-1} \Cstar)^{1/2}
	\Theta^{1/2} \|w_h\|_{1,k}.
\end{equation}
Together with \eqref{estwh}, this proves \eqref{3_7_b}. 

Let us now show \eqref{3_7_d}. Inserting $P_0 u - T_0 u \in V_0$ in the definition \eqref{eq: 2_25_zero} of $T_0$
we obtain 
\[
	b(u-T_0 u, P_0 u - T_0 u) = 0.
\]
Using this together with  the link between the bilinear forms $(\cdot,\cdot)_{1,k}$ and $b$, 
and the Cauchy--Schwarz inequality we obtain
\begin{align*}
	\|u - T_0 u\|_{1,k}^2 &= (u - T_0 u, u - T_0 u)_{1,k} 
	= b(u - T_0 u, u - T_0 u) + 2k^2(u - T_0 u, u - T_0 u)_n
	\\[0.5ex]
	&= b(u - T_0 u, u - T_0 u) - b(u - T_0 u, P_0 u - T_0 u) + 2 \kappa^2 (u - T_0 u, u - T_0 u)_n
	\\[0.5ex]
	&= b(u - T_0 u, u - P_0 u) + 2k^2(u - T_0 u, u - T_0 u)_n
	\\[0.5ex]
	&= (u - T_0 u, u - P_0 u)_{1,k} - 2k^2(u - T_0 u, u - P_0 u)_n + 2k^2(u - T_0 u, u - T_0 u)_n
	\\[0.5ex]
	&= (u - T_0 u, u - P_0 u)_{1,k} + 2k^2(u - T_0 u, P_0 u - T_0 u)_n
	\\[0.5ex]
	&\le \|u - T_0 u\|_{1,k}\|u - P_0 u\|_{1,k} + 2k^2\|u - T_0 u\|_n \|T_0 u - P_0 u\|_n
	\\[0.5ex]
	&\le \|u - T_0 u\|_{1,k}\|u - P_0 u\|_{1,k} + 2k\|u - T_0 u\|_n \|T_0 u - P_0 u\|_{1,k}
	\\[0.5ex]
	&= \|u - T_0 u\|_{1,k}\|u - P_0 u\|_{1,k} + 2k\|u - T_0 u\|_n \|P_0(T_0 u - u)\|_{1,k}
	\\[0.5ex]
	&\le \|u - T_0 u\|_{1,k}\|u\|_{1,k} + 2k\|u - T_0 u\|_n \|T_0 u - u\|_{1,k},
\end{align*}
where in the last step we used the fact that $P_0$ is the orthogonal projection onto $V_0$ with respect to $(\cdot,\cdot)_{1,k}$.
Dividing both sides by $\|u - T_0 u\|_{1,k}$ and then using \eqref{3_7_b}, $\nmax = 1$ and the assumption \eqref{3_7_c}, we arrive at 
\begin{align*}
	\|u - T_0 u\|_{1,k} &\le \|u\|_{1,k} + 2k\|u - T_0 u\|_n
	\le \|u\|_{1,k} + 2{C_2}\, \Lambda_c (\amax \nmin^{-1} \Cstar)^{1/2}
	k \Theta^{1/2} (1 + \Cstab)\|u - T_0 u\|_{1,k}
	\\[0.5ex]
	&\le \|u\|_{1,k} + \frac{1}{2}\|u - T_0 u\|_{1,k},
\end{align*}
which proves \eqref{3_7_d}.
\end{proof}

\section{Proof of the main result}\label{sec:proof}

Before Theorem~\ref{theorem: convergence} can be proved, it is necessary to state and prove the following key lemma..

\begin{lemma}[Core estimates for the main result] \label{ther: 4_1}
Under the assumptions of Theorem~\ref{theorem: convergence}, let 
\begin{equation*}
	c_1 \assign
	\frac{1-s}{C_3\bigl(1 + \Lambda_\ell \Lambda_c^2 \amax {\nmin^{-1} \Cstar } \Theta \bigr)},
	\qquad
	c_2 \assign 18 + 8 {\Lambda_{\ell}}^2, 
\end{equation*}
where $s$ is given by \eqref{eq: 4_3}.  Then, for all $u \in V^h$,
\begin{equation}
	c_1\|u\|_{1,k}^2 \le (Tu,u)_{1,k},
	\label{eq: 4_4}
\end{equation}
and
\begin{equation}
	\|Tu\|_{1,k}^2 \le c_2 \|u\|_{1,k}^2.
	\label{eq: 4_5}
\end{equation}
\end{lemma}

\begin{proof}\,
Let $u \in V^h$.  We proceed in several steps.
\\[1ex]
\textit{Step 1} (\textit{Relation between $(Tu,u)_{1,k}$ and $(Pu,u)_{1,k}$}). 
Using \eqref{eq: 3_11} and the definition of $b$, we obtain
\begin{align*}
	(Pu,u)_{1,k} & {= (u, Pu)_{1,k}}  
	= \biggl(u,E_0 P_0u + \sum^{Q}_{j=1} E^{\ell}_j P^{\ell}_j u\biggr)_{1,k}
	\\
	&=  b(u,E_0 P_0u) + 2k^2(u,E_0 P_0u)_n 
	+ \sum^{Q}_{j=1} \left[ b\bigl(u, E^{\ell}_j P^{\ell}_j u\bigr) + 2k^2\bigl(u,E^{\ell}_j P^{\ell}_j u\bigr)_n \right].
\end{align*}
Then, using \eqref{eq: bT} and \eqref{eq: 2_25_zero}, we obtain   
\begin{align*}
	(Pu,u)_{1,k} &= b(T_0 u,P_0u) + 2k^2(u,E_0 P_0u)_n 
	+ \sum^{Q}_{j=1} \left[ b_{\Omega^{\ell}_j}\bigl(T^{\ell}_j u, P^{\ell}_j u\bigr) + 2k^2\bigl(u,E^{\ell}_j P^{\ell}_j u\bigr)_n \right]
	\\[0.5ex]
	&= (T_0 u,P_0u)_{1,k} - 2k^2(T_0 u,P_0u)_n + 2k^2(u,E_0 P_0u)_n 
	\\[0.5ex] 
	&\quad + \sum^{Q}_{j=1} \left[ \bigl(T^{\ell}_j u, P^{\ell}_j u\bigr)_{1,k,\Omega^{\ell}_j} 
	- 2k^2\bigl(T^{\ell}_j u, {P^{\ell}_j} u\bigr)_{n,\Omega^{\ell}_j} + 2k^2\bigl(u,E^{\ell}_j P^{\ell}_j u\bigr)_n \right].
\end{align*}
Now using the definitions \eqref{eq: 3_10_nested} and \eqref{eq: 3_10_zero} of $P_j^\ell$ and $P_0$, 
we have 
\begin{align*}
  	(Pu,u)_{1,k} 
	&= (E_0 T_0 u,u)_{1,k} - 2k^2(E_0 T_0 u,E_0 P_0u)_n + 2k^2(u,E_0 P_0u)_n 
	\\[0.5ex]
	&\quad + \sum^{Q}_{j=1} \left[ (E^{\ell}_j T^{\ell}_j u, u)_{1,k} 
	- 2k^2\bigl(E^{\ell}_j T^{\ell}_j u, E^{\ell}_j P^{\ell}_j u\bigr)_n + 2k^2\bigl(u,E^{\ell}_j P^{\ell}_j u\bigr)_n \right]. 
\end{align*}
We now rearrange the right-hand side to obtain
\begin{align}
	(Pu,u)_{1,k} &= (E_0 T_0 u,u)_{1,k} + \sum^{Q}_{j=1} \bigl(E^{\ell}_j T^{\ell}_j u, u\bigr)_{1,k} 
	- 2k^2 \biggl(\bigl(E_0 T_0 u - u,E_0 P_0u\bigr)_n 
	+ \sum^{Q}_{j=1} \bigl(E^{\ell}_j T^{\ell}_j u - u, E^{\ell}_j P^{\ell}_j u\bigr)_n \biggr)
    \nonumber\\[0.5ex]
    &= (T u, u)_{1,k} - R, 
    \label{relation}
\end{align}
with $T$ defined in \eqref{eq: 2_26} and
\begin{align}\label{defR}
	R \assign 2k^2 \biggl( \bigl(E_0 T_0 u - u, E_0 P_0u\bigr)_n 
	+ \sum^{Q}_{j=1} \bigl(E^{\ell}_j T^{\ell}_j u - u, E^{\ell}_j P^{\ell}_j u\bigr)_n \biggr). 
\end{align}
\textit{Step 2} (\textit{Bounding the first term in $R$}).
Recalling that $E_0 w_0 = w_0$ for all $w_0\in V_0$, and noting that the assumption in \eqref{eq: 4_3} implies that \eqref{3_7_c} is satisfied,
we can use \eqref{3_7_b}, \eqref{3_7_d} and the fact that $P_0$ is the orthogonal projection onto $V_0$ 
with respect to $(\cdot, \cdot)_{1,k}$, to obtain
\begin{align}
	k^2\bigl(E_0 T_0 u - u, E_0 P_0 u\bigr)_n
	&\le k^2 \|E_0 T_0 u - u\|_n \|E_0 P_0 u\|_n
	= k^2\| T_0 u -u\|_n \| P_0 u\|_n
	\nonumber\\[0.5ex]
	&\le k\|T_0 u -u\|\, \|P_0 u\|_{1,k}
	\nonumber\\[0.5ex]
	&\le C_2\, \Lambda_c (\amax \nmin^{-1} \Cstar)^{1/2}
	k \Theta^{1/2} (1+\Cstab)\|T_0 u - u\|_{1,k}\|P_0 u\|_{1,k}
	\nonumber\\[0.5ex]
	&\le 2C_2\, \Lambda_c (\amax \nmin^{-1} \Cstar)^{1/2}
	k \Theta^{1/2} (1+\Cstab)\|u\|_{1,k}\|P_0 u\|_{1,k}
	\nonumber\\[0.5ex]
	&\le 2C_2\, \Lambda_c (\amax \nmin^{-1} \Cstar)^{1/2}
	k \Theta^{1/2} (1+\Cstab)\|u\|_{1,k}^2.
	\label{eq:10_0}
\end{align}
\textit{Step 3} (\textit{Bounding the second term in $R$}).
For each $j$, we use Friedrichs' inequality \eqref{eq:friedrichs} to arrive at
\begin{align}
	k^2\bigl(E^{\ell}_j T^{\ell}_j u-u,E^{\ell}_j P^{\ell}_j u\bigr)_n 
	&= k^2\bigl({T^{\ell}_j} u-u|_{\Omega^{\ell}_j}, P^{\ell}_j u\bigr)_{n,\Omega^{\ell}_j}
	\le k^2 \big\|T^{\ell}_j u-u|_{\Omega^{\ell}_j}\big\|_{n,\Omega^{\ell}_j}\big\|P^{\ell}_j u\big\|_{n,\Omega^{\ell}_j}
	\nonumber\\[0.5ex]
	&\le k\bigl\|T^{\ell}_j u-u|_{\Omega^{\ell}_j}\big\|_{1,k,\Omega^{\ell}_j}\big\|{P^{\ell}_j} u\big\|_{n,\Omega^{\ell}_j}
	\le k\bigl\|{T^{\ell}_j} u-u|_{\Omega^{\ell}_j}\big\|_{1,k,\Omega^{\ell}_j}
	\frac{H_{\ell}}{\sqrt{2}\,}\big\|{P^{\ell}_j} u\big\|_{1,k,\Omega^{\ell}_j}
	\nonumber\\[0.5ex]
	&\le \frac{k H_{\ell}}{\sqrt{2}\,}
	\Bigl(\|{T^{\ell}_j} u\|_{1,k,\Omega^{\ell}_j}+\big\|u|_{\Omega^{\ell}_j}\big\|_{1,k,{\Omega^{\ell}_j}}\Bigr)
	\big\|P^{\ell}_j u\big\|_{1,k,\Omega^{\ell}_j}.
   \label{lemma411}
\end{align}
Now, note that the hypothesis \eqref{eq: 4_3} implies that $\sqrt{2} k {H_{\ell}} < \sqrt{2}/(3 \Lambda_\ell) < 1$; 
so we can apply Lemma~\ref{lemma_3_6} to obtain from \eqref{lemma411} that
\begin{equation*} 	
	k^2\bigl(E^{\ell}_j T^{\ell}_j u-u, E^{\ell}_j P^{\ell}_j u\bigr)_n 
	\le \frac{3kH_{\ell}}{\sqrt{2}\,}\big\|u|_{\Omega^{\ell}_j}\big\|_{1,k,\Omega^{\ell}_j}\big\|P^{\ell}_j u\big\|_{1,k,\Omega^{\ell}_j}
	< 3kH_{\ell}\big\|u|_{\Omega^{\ell}_j}\big\|_{1,k,\Omega^{\ell}_j}\big\|P^{\ell}_j u\big\|_{1,k,\Omega^{\ell}_j}.
\end{equation*}
Summing over $j$, applying Cauchy--Schwarz and \eqref{overlap}, we obtain
\begin{align}
	& k^2 \sum_{j=1}^{Q} \bigl(E^{\ell}_j T^{\ell}_j u-u, E^{\ell}_j P^{\ell}_j u\bigr)_n 
	\le 3kH_{\ell} \sum_{j=1}^{Q} \big\|u|_{\Omega^{\ell}_j}\big\|_{1,k,\Omega^{\ell}_j}\|P^{\ell}_j u\|_{1,k,\Omega^{\ell}_j}
	\nonumber\\[0.5ex]
	&\le 3kH_{\ell} \Biggl(\sum_{j=1}^{Q} \big\|u|_{\Omega^{\ell}_j}\big\|_{1,k,\Omega^{\ell}_j}^2\Biggr)^{1/2}
	\Biggl( \sum_{j=1}^{Q} \|P^{\ell}_j u\|_{1,k,\Omega^{\ell}_j}^2\Biggr)^{1/2} 
	\le 3kH_{\ell} \Lambda_{\ell}^{1/2} \|u\|_{1,k}\Biggl( \sum_{j=1}^{Q} \|P^{\ell}_j u\|_{1,k,\Omega^{\ell}_j}^2\Biggr)^{1/2}.
	\label{eq:10_1}
\end{align}
To estimate the sum on the right-hand side of \eqref{eq:10_1}, note that
\begin{equation*}
	\|E_j^\ell P_j^\ell u \|_{1,k}^2 
	= (P_j^\ell u,P_j^\ell u)_{1,k,\Omega_j^\ell} 
	= (P_j^\ell u, u)_{1,k,\Omega_j^\ell} = (E_j^\ell P_j^\ell u, u)_{1,k}, 
\end{equation*}
and so, using \eqref{overlap1}, we have
\begin{align*}
	\Bigg\|\sum_{j=1}^{Q} E^{\ell}_j P^{\ell}_j u\Bigg\|_{1,k}^2
	&\le \Lambda_{\ell} \sum_{j=1}^{Q} \|E^{\ell}_j P^{\ell}_j u\|_{1,k}^2 
	= \Lambda_{\ell} \sum_{j=1}^{Q} (E^{\ell}_j P^{\ell}_j u,u)_{1,k}
	\\[0.5ex]
	&= \Lambda_{\ell} \Biggl(\sum_{j=1}^{Q} E^{\ell}_j P^{\ell}_j u,\,u\Biggr)_{1,k}
	\le \Lambda_{\ell} \,\Bigg\|\sum_{j=1}^{Q} E^{\ell}_j P^{\ell}_j u\Bigg\|_{1,k} \|u\|_{1,k}.
\end{align*}
Thus
\begin{align*}
    \sum_{j=1}^{Q} \|P^{\ell}_j u\|_{1,k,\Omega^{\ell}_j}^2
	&= \sum_{j=1}^{Q} (E^{\ell}_j P^{\ell}_j u,u)_{1,k} 
	= \Biggl(\sum_{j=1}^{Q} E^{\ell}_j P^{\ell}_j u,\,u\Biggr)_{1,k} 
	\le \Bigg\|\sum_{j=1}^{Q} E^{\ell}_j P^{\ell}_j u\Bigg\|_{1,k}\|u\|_{1,k}
	\le \Lambda_{\ell} \|u\|_{1,k}^2.
\end{align*}
This, together with \eqref{eq:10_1}, leads to
\begin{equation}\label{eq:10_2}
	k^2 \sum_{j=1}^{Q} \bigl(E^{\ell}_j T^{\ell}_j u-u, E^{\ell}_j P^{\ell}_j u\bigr)_n 
	\le 3kH_{\ell} \Lambda_{\ell} \|u\|_{1,k}^2.
\end{equation}
\textit{Step 4} (\textit{Obtaining \eqref{eq: 4_4}}).
We can now combine \eqref{defR}, \eqref{eq:10_0} and \eqref{eq:10_2} to obtain
\begin{align*}
	|R| &\le 2 \Bigl(2 C_2\, \Lambda_c (\amax \nmin^{-1} \Cstar)^{1/2}
	k \Theta^{1/2} (1+\Cstab) + 3 k \Lambda_\ell H_\ell\Bigr) \|u\|_{1,k}^2
	\nonumber\\[0.5ex]
	&= \frac{s}{C_3 \bigl(1+\Lambda_\ell \Lambda_c^2 \amax \nmin^{-1} \Cstar \Theta \bigr)}	\|u\|_{1,k}^2, 
\end{align*}
with $s$ and $C_3$ as defined in Theorem~\ref{theorem: convergence}. 
Hence, using this, \eqref{relation} and Proposition \ref{prop:Puu}, we obtain
\begin{align*}
	(Tu,u)_{1,k} &= (Pu,u)_{1,k} + R 
	\ge C_3^{-1} \bigl(1 + \Lambda_\ell \Lambda_c^2 \amax \nmin^{-1} \Cstar \Theta \bigr)^{-1} \|u\|_{1,k}^2 - |R| 
	\nonumber\\[0.5ex]
	&\ge C_3^{-1} \bigl(1 + \Lambda_\ell \Lambda_c^2 \amax \nmin^{-1} \Cstar \Theta \bigr)^{-1} (1-s) \|u\|_{1,k}^2, 
\end{align*}
as required. 
\\[1ex]
\textit{Step 5} (\textit{Obtaining \eqref{eq: 4_5}}).
Since $E_0T_0 u = T_0u $, we have 
\begin{equation}\label{eq: 4_18}
	\|Tu\|_{1,k}^2 =  \Bigg\|T_0 u + \sum_{j=1}^{Q} E^{\ell}_j T^{\ell}_j u\Biggr\|_{1,k}^2 
	\le 2\|T_0 u\|_{1,k}^2 + 2\Bigg\| \sum_{j=1}^{Q} E^{\ell}_j T^{\ell}_j u\Bigg\|_{1,k}^2.
\end{equation}
For the first term on the right-hand side of \eqref{eq: 4_18},   Cauchy--Schwarz  and \eqref{3_7_d} yield
\begin{align}
	\|T_0 u\|_{1,k}^2 &= (T_0 u, T_0 u)_{1,k} 
	= (T_0 u - u, T_0 u)_{1,k} + (u, T_0 u)_{1,k} 
	\nonumber\\[0.5ex]
	&\le \|T_0 u - u\|_{1,k}\|T_0 u\|_{1,k} + \|u\|_{1,k}\|T_0 u\|_{1,k}
	\nonumber\\[0.5ex]
	&\le 2\|u\|_{1,k}\|T_0 u\|_{1,k} + \|u\|_{1,k}\|T_0 u\|_{1,k} 
	= 3\|u\|_{1,k}\|T_0 u\|_{1,k}. 
\label{eq: 4_19} 
\end{align}
For the second term on the right-hand side of \eqref{eq: 4_18} we use \eqref{overlap1}, Lemma~\ref{lemma_3_6} and \eqref{overlap} to obtain 
\begin{equation}
	\label{eq:second}
	\Bigg\|\sum_{j=1}^{Q} E^{\ell}_j T^{\ell}_j u\Bigg\|_{1,k}^2 
	\le \Lambda_{\ell} \sum_{j=1}^{Q} \|T^{\ell}_j u\|_{1,k,\Omega^{\ell}_j}^2 
	\le \Lambda_{\ell} \sum_{j=1}^{Q} 4\big\|u|_{\Omega_i}\big\|_{1,k,\Omega^{\ell}_j}^2 
	\le 4 \Lambda_{\ell}^2\|u\|_{1,k}^2. 
\end{equation}
Combining \eqref{eq: 4_19} and \eqref{eq:second} with \eqref{eq: 4_18} we arrive at
\begin{equation*}
	\|Tu\|_{1,k}^2 \le 2\times9\|u\|_{1,k}^2 + 2\times4 \Lambda_{\ell}^2\|u\|_{1,k}^2 
	= (18 + 8\Lambda_{\ell}^2)\|u\|^2_{1,k},
\end{equation*}
which proves \eqref{eq: 4_5}.
\end{proof}

We now complete the proof of Theorem~\ref{theorem: convergence}.

\begin{proof}[Proof of Theorem~\ref{theorem: convergence}]
Using \eqref{eq: 4_4} with \eqref{eq: 2_27} we obtain
\begin{equation*}
	c_1 \|u\|_{1,k}^2 \le (Tu,u)_{1,k} = \bigl\langle \mathbf{M}_{AS,2}^{-1}\mathbf{B}\mathbf{u},\mathbf{u}\bigr\rangle_{\mathbf{D}_{k}},
\end{equation*}
which can be written as 
\begin{equation*}
	c_1 \le \frac{\bigl\langle\mathbf{M}_{AS,2}^{-1}\mathbf{B}\mathbf{u},\mathbf{u}\bigr\rangle_{\mathbf{D}_{k}}}{\|\mathbf{u}\|_{\mathbf{D}_{k}}^2 }.
\end{equation*}
Also, combining \eqref{eq: 4_5} with \eqref{eq: 2_27}, we get
\begin{equation*}
	\|\mathbf{M}_{AS,2}^{-1}\mathbf{B}\mathbf{u}\|_{\mathbf{D}_{k}}^2 
	= \|Tu\|_{1,k}^2 
	\le c_2\|u\|_{1,k}^2 = c_2\|\mathbf{u}\|^2_{\mathbf{D}_{k}},
	\qquad \text{i.e.} \quad 
	\|\mathbf{M}_{AS,2}^{-1}\mathbf{B}\|_{\mathbf{D}_k}^2 \le c_2.
\end{equation*}
The result on GMRES convergence now follows directly from the Elman theory \cite{Elman:1983:VIM}
with $\gamma=\frac{c_1}{c_2}$. 
\end{proof}

\section{Numerical results}
\label{sec:numerics}

This section illustrates the effectiveness of the the two-level $H_k$-GenEO preconditioner \eqref{eq: 2_23}, 
by testing its robustness to wavenumber $k$, scalability with respect to the number of coarse subdomains $N$ 
and sensitivity to the spectral threshold $\tau$, with the latter controlling  the dimension of the coarse space. 
All computations are performed in \texttt{FreeFEM}, using the \texttt{ffddm} framework to provide MPI-parallel tools for domain decomposition. 

The domain $\Omega \assign (0,1)^2$ is partitioned into local and coarse subdomains (see \eqref{covers}), 
each of which are here taken to be non-overlapping squares, extended by one finite element layer on each side of the 
interface to create minimal overlap.  Unless otherwise stated, minimal overlap is used throughout numerical experiments. 
The local matrices $\mathbf{B}^{\ell}_j$  on $\Omega_j^\ell$ in \eqref{eq: 2_23} are factorised using \texttt{MUMPS} 
with the factors reused throughout the iterative solve. 
To obtain the coarse matrix $\mathbf{B}_0$ the generalised eigenproblem \eqref{eq: 5_12} is solved 
on each coarse subdomain $\Omega_i^c$ using \texttt{SLEPc}, and eigenmodes with eigenvalues 
below a prescribed threshold $\tau$ are retained (see \eqref{def:tau}). 
The corresponding coarse basis functions are assembled into a distributed coarse matrix, 
factorised with \texttt{MUMPS} using a dedicated pool of MPI processes. 
The preconditioned system \eqref{eq: 2_24} is solved with GMRES (no restart) until the relative residual is reduced below $10^{-6}$ 
or 200 iterations are reached. We distinguish clearly between the setup phase (assembly, eigenproblems, factorisations) 
and the application phase (local and coarse solves at each iteration). 
Except for \S \ref{subsec:Decoupled}, the local and coarse  covers in \eqref{covers} are taken to be identical, 
so that $N=Q$ and $H_\ell = H_c$. Two representative test cases are considered.

\paragraph{Homogeneous problem.}\, 
We study problem \eqref{eq:problem} with $A = I$ and $n=1$ but with  $k$ increasing. 
The  function $f$ is a  point source  modelled by the  Gaussian:
\(f(x,y)=10^4 \exp(-10^3[(x-\frac{1}{2})^2+(y-\frac{1}{2})^2])\). The discretisation space $V^h$ consists of standard Lagrange elements 
of degree 1 on a uniform Cartesian grid with spacing $h \sim 1/k$, using alternating diagonals to create a conforming simplicial mesh. 

\paragraph{Heterogenous problem.}\, 
This is the same as  the homogenous case, except that the diffusion coefficient $A = I$ is here replaced by $A(\mathbf{x}) = a(\mathbf{x})I$ 
where the heterogeneous function $a$ represents a layered medium as depicted in Figure~\ref{fig:HetPlot_Hk}. 
Here $a$ takes two values, either $a \equiv \amin = 1$ (in the white strips) or $a \equiv \amax > 1$ (in the grey strips). 
The contrast is thus controlled by the single parameter $\amax$. 
\begin{figure}[h]
\label{Fig:AlternatingLay}
\centering
\begin{tikzpicture}[scale=1.5]
    \fill[gray,opacity=0.1] (-1,1) -- (1,1) -- (1,0.8) -- (-1,0.8) -- cycle;
	\fill[gray,opacity=1.0] (-1,0.6) -- (1,0.6) -- (1,0.8) -- (-1,0.8) -- cycle;
	\fill[gray,opacity=0.1] (-1,0.6) -- (1,0.6) -- (1,0.4) -- (-1,0.4) -- cycle;
	\fill[gray,opacity=1.0] (-1,0.2) -- (1,0.2) -- (1,0.4) -- (-1,0.4) -- cycle;
	\fill[gray,opacity=0.1] (-1,0.2) -- (1,0.2) -- (1,0.0) -- (-1,0.0) -- cycle;
	\fill[gray,opacity=1.0] (-1,-0.2) -- (1,-0.2) -- (1,0.0) -- (-1,0.0) -- cycle;
	\fill[gray,opacity=0.1] (-1,-0.2) -- (1,-0.2) -- (1,-0.4) -- (-1,-0.4) -- cycle;
	\fill[gray,opacity=1.0] (-1,-0.6) -- (1,-0.6) -- (1,-0.4) -- (-1,-0.4) -- cycle;
	\fill[gray,opacity=0.1] (-1,-0.6) -- (1,-0.6) -- (1,-0.8) -- (-1,-0.8) -- cycle;
	\fill[gray,opacity=1.0] (-1,-1) -- (1,-1) -- (1,-0.8) -- (-1,-0.8) -- cycle;
    
    \draw[step=0.5cm,gray!50,very thin, shift={(-1,-1)}] (0,0) grid (2,2);
    \draw[gray!50,very thin] (-1,-1) -- (1,1);
    \draw[gray!50,very thin] (-1,1) -- (1,-1);
    \draw[gray!50,very thin] (-1,0) -- (0,1) -- (1,0) -- (0,-1) -- cycle;
    
    \draw (-1,1) -- (-1,-1) -- (1,-1) -- (1,1) -- cycle;    
    \draw (-1.05,1) node[left] {$y = 1$}
          (-1.05,-1) node[left] {$y = 0$}
          (1,-1.05) node[below] {$x = 1$}
          (-1,-1.05) node[below] {$x = 0$};
    \draw[fill=black] (0,0) circle (0.025) node[below] {source $f$};
\end{tikzpicture}
\caption{The function $a$: in the dark regions $a \equiv \amax$, where $\amax > 1$ is a parameter, 
and in the white regions $a \equiv 1$.}
\label{fig:HetPlot_Hk}
\end{figure}
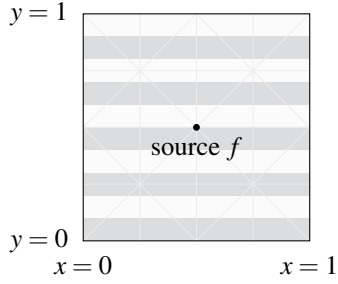

In Tables \ref{tab:merged_tau} and \ref{tab:het_merged_tau} we show (for the homogeneous and heterogeneous cases), 
the dimension of the fine mesh (dim(Fine)), the maximum number of negative eigenvalues per subdomain (Neg), 
the minimum eigenvalue ($\lambda_\text{min}$), the dimension of the coarse space (dim(CS)), 
and the GMRES iteration count (It) for varying $k$, $N$ and $\tau \in\{0,2,0,4,0.6\}$.
\begin{table}[h]
\centering
\scriptsize
\setlength{\tabcolsep}{4pt}
\begin{tabular}{cccc|cc|cc|cc|cc}
\multicolumn{4}{c|}{Problem parameters}
& \multicolumn{2}{c|}{Spectral data}
& \multicolumn{2}{c|}{$\tau=0.2$}
& \multicolumn{2}{c|}{$\tau=0.4$}
& \multicolumn{2}{c}{$\tau=0.6$} \\

$N$ & $k$ & $h^{-1}$ & Dim(Fine)
& Neg. & $\lambda_{\min}$
& Dim(CS) & It.
& Dim(CS) & It.
& Dim(CS) & It. \\
\hline

16 & 20 & 240 & 58081
& 4 & -0.324293
& 144 & 21
& 240 & 15
& 392 & 12 \\

 & 60 & 720 & 519841
& 22 & -0.849475
& 668 & 23
& 1160 & 16
& 1984 & 12 \\

 & 100 & 1200 & 1442401
& 58 & -0.941228
& 1612 & 27
& 2660 & 14
& 4572 & 11 \\

\hline

36 & 20 & 240 & 58081
& 3 & -0.178013
& 220 & 23
& 344 & 18
& 600 & 13 \\

 & 60 & 720 & 519841
& 13 & -0.703879
& 844 & 29
& 1528 & 18
& 2576 & 13 \\

 & 100 & 1200 & 1442401
& 28 & -0.875285
& 1876 & 34
& 3260 & 16
& 5584 & 12 \\

\hline

64 & 20 & 240 & 58081
& 1 & -0.125924
& 224 & 35
& 448 & 20
& 740 & 15 \\

 & 60 & 720 & 519841
& 8 & -0.553713
& 1060 & 33
& 1924 & 18
& 3232 & 13 \\

 & 100 & 1200 & 1442401
& 17 & -0.793739
& 2276 & 35
& 3840 & 21
& 6508 & 13 \\

\hline

100 & 20 & 240 & 58081
& 1 & -0.100482
& 324 & 32
& 684 & 18
& 1044 & 13 \\

 & 60 & 720 & 519841
& 4 & -0.423257
& 1144 & 42
& 2200 & 19
& 3748 & 14 \\

 & 100 & 1200 & 1442401
& 9 & -0.721832
& 1846 & 45
& 3588 & 22
& 5976 & 15 \\

\hline

144 & 20 & 240 & 58081
& 1 & -0.0852859
& 484 & 28
& 672 & 21
& 1096 & 16 \\

 & 60 & 720 & 519841
& 4 & -0.324293
& 1584 & 33
& 2448 & 23
& 4232 & 14 \\

 & 100 & 1200 & 1442401
& 7 & -0.648756
& 2594 & 50
& 4186 & 25
& 6982 & 16 \\

\hline
\end{tabular}
\caption{Homogeneous case: spectral data, and GMRES iterations for varying $k,N,\tau$}
\label{tab:merged_tau}
\end{table}

\begin{table}[h]
\centering
\scriptsize
\hspace*{-0.9cm}
\setlength{\tabcolsep}{3pt}
\begin{tabular}{ccccc|cc|cc|cc|cc}
\multicolumn{5}{c|}{Problem parameters} 
& \multicolumn{2}{c|}{Spectral data}
& \multicolumn{2}{c|}{$\tau=0.2$}
& \multicolumn{2}{c|}{$\tau=0.4$}
& \multicolumn{2}{c}{$\tau=0.6$} \\

$N$ & $\amax$ & $k$ & $h^{-1}$ & Dim(Fine)
& Neg. & $\lambda_{\min}$
& Dim(CS) & It.
& Dim(CS) & It.
& Dim(CS) & It. \\
\hline

16 & 10 & 20 & 240 & 58081
& 2 & -0.0795278
& 112 & 25
& 196 & 19
& 340 & 16 \\

 &  & 60 & 720 & 519841
& 15 & -0.620619
& 474 & 68
& 866 & 19
& 1468 & 15 \\

 &  & 100 & 1200 & 1442401
& 35 & -0.840064
& 1058 & 35
& 1860 & 20
& 3132 & 18 \\
\hline

16 & 1000 & 20 & 240 & 58081
& 2 & -0.0010342
& 112 & 25
& 190 & 19
& 332 & 16 \\

 &  & 60 & 720 & 519841
& 12 & -0.523291
& 452 & 32
& 826 & 20
& 1392 & 17 \\

 &  & 100 & 1200 & 1442401
& 31 & -0.796087
& 974 & 26
& 1758 & 18
& 2974 & 15 \\
\hline

36 & 10 & 20 & 240 & 58081
& 2 & -0.0444599
& 180 & 28
& 322 & 19
& 556 & 16 \\

 &  & 60 & 720 & 519841
& 7 & -0.512559
& 650 & 30
& 1220 & 19
& 2066 & 15 \\

 &  & 100 & 1200 & 1442401
& 17 & -0.785888
& 1350 & 35
& 2456 & 19
& 4182 & 16 \\
\hline

36 & 1000 & 20 & 240 & 58081
& 2 & -0.000860171
& 178 & 26
& 320 & 19
& 544 & 17 \\

 &  & 60 & 720 & 519841
& 7 & -0.473817
& 640 & 28
& 1180 & 20
& 2018 & 16 \\

 &  & 100 & 1200 & 1442401
& 15 & -0.767375
& 1288 & 62
& 2382 & 49
& 4026 & 22 \\
\hline

64 & 10 & 20 & 240 & 58081
& 1 & -0.0381661
& 220 & 31
& 446 & 19
& 732 & 17 \\

 &  & 60 & 720 & 519841
& 6 & -0.437747
& 854 & 56
& 1580 & 50
& 2660 & 46 \\

 &  & 100 & 1200 & 1442401
& 13 & -0.738717
& 1658 & 34
& 3100 & 18
& 5160 & 16 \\
\hline

64 & 1000 & 20 & 240 & 58081
& 1 & -0.00392729
& 218 & 27
& 458 & 20
& 736 & 17 \\

 &  & 60 & 720 & 519841
& 5 & -0.417451
& 830 & 35
& 1578 & 20
& 2638 & 17 \\

 &  & 100 & 1200 & 1442401
& 11 & -0.730374
& 1634 & 34
& 3012 & 21
& 5000 & 17 \\
\hline

100 & 10 & 20 & 240 & 58081
& 1 & -0.0364662
& 318 & 27
& 572 & 19
& 932 & 17 \\

 &  & 60 & 720 & 519841
& 4 & -0.37355
& 1130 & 32
& 2094 & 19
& 3392 & 16 \\

 &  & 100 & 1200 & 1442401
& 12 & -0.689838
& 2142 & 37
& 3882 & 19
& 6242 & 16 \\
\hline

100 & 1000 & 20 & 240 & 58081
& 1 & -0.000587616
& 360 & 23
& 548 & 18
& 932 & 16 \\

 &  & 60 & 720 & 519841
& 3 & -0.361218
& 1194 & 28
& 2122 & 18
& 3352 & 16 \\

 &  & 100 & 1200 & 1442401
& 10 & -0.687193
& 2232 & 25
& 3872 & 19
& 6200 & 16 \\
\hline

144 & 10 & 20 & 240 & 58081
& 1 & -0.0513371
& 410 & 30
& 646 & 21
& 1120 & 17 \\

 &  & 60 & 720 & 519841
& 4 & -0.295028
& 1202 & 82
& 2250 & 23
& 3894 & 18 \\

 &  & 100 & 1200 & 1442401
& 8 & -0.601849
& 2198 & 53
& 4328 & 21
& 7116 & 16 \\
\hline

144 & 1000 & 20 & 240 & 58081
& 1 & -0.016502
& 408 & 25
& 660 & 20
& 1150 & 17 \\

 &  & 60 & 720 & 519841
& 3 & -0.288366
& 1222 & 68
& 2268 & 21
& 3856 & 17 \\

 &  & 100 & 1200 & 1442401
& 8 & -0.599871
& 2162 & 114
& 4266 & 119
& 7032 & 22 \\
\hline

\end{tabular}
 \caption{Heterogeneous case ($\amax=10$ and $1000$): spectral data, and GMRES iterations for varying $k,N,\tau$}
\label{tab:het_merged_tau}
\end{table}

\subsection{Behaviour of the method with respect to different parameters}

\paragraph{Influence of the threshold \(\tau\).}\ 
As expected,  for both homogeneous and heterogeneous problems, \(\tau\) regulates the size of the coarse space and robustness of the preconditioner.
\begin{mybull}
\item 
	For \(\tau=0.2\), only a limited number of eigenfunctions are retained. 
	This leads to smaller coarse spaces but a clear deterioration in robustness as \(k\) increases. 
	As $k$ increases, the GMRES iterations  grow substantially, 
	especially for larger subdomain counts \(N\). 

\item
	For \(\tau=0.4\), a favourable balance is observed. 
	The coarse space dimension grows moderately with \(k\) and \(N\), while the iteration counts mostly remain stable 
	and essentially independent of $k$. 
	This represents an effective compromise between setup cost and iteration count.

\item
	For \(\tau=0.6\), more eigenmodes are included in the coarse space. 
	and \text{Dim(CS)} increases (typically by a factor between \(1.5\) and \(2\) compared to \(\tau=0.4\)), 
	but GMRES convergence becomes nearly independent of both \(k\) and \(N\). 
	Iteration counts flatten, demonstrating enhanced robustness.
\end{mybull}

\paragraph{Dependence on frequency \(k\) and number of subdomains \(N\).}

For all values of $\tau$ tested, the number of negative eigenvalues of \eqref{eq: 5_12} increases with \(k\),
reflecting the growing indefiniteness of the operator. Nevertheless, when \(\tau\) is chosen moderately 
or generously (\(\tau=0.4\) or \(\tau=0.6\)), the resulting coarse space appears to capture the problematic modes which, 
if left out, yield poor convergence.  We observe that the GMRES iteration counts for $\tau = 0.4, 0.6$ remain essentially bounded as \(k\) grows. 
This observed behaviour appears somewhat better than the theoretical estimate for $\tau$ given in \eqref{Hk_results}, 
which grows with $k$.  However, we should be aware that here we have only tested values of $k$ within numerical reach.

While \text{Dim(CS)} increases with \(k\), its rate of increase is moderate compared to the order $k^2$ increase rate of \text{Dim(fine)}. 
For example, simple extrapolation of the results for $\tau = 0.4$ in Table~\ref{tab:merged_tau} show a growth rate of order about $k^{1.4}$ 
for $N= 16$, $k^{1.3}$ for $N=64$ and only  $k^{1.1}$ for $N=144$. 

\paragraph{Effect of heterogeneity.}

Table~\ref{tab:het_merged_tau} exhibits broadly similar behaviour to Table~\ref{tab:merged_tau}, 
although the precise coarse space dimensions and iteration counts do depend  (but only mildly) on the contrast \(\amax\). 
This suggests that there may be scope to sharpen the theory in the heterogeneous case: 
the lower bound for the convergence rate in \eqref{def_gamma} contains an explicit dependence on \(\amax\), 
whereas this dependence is not strongly visible in the present numerical tests. 
In some moderate-frequency cases, heterogeneity even slightly reduces iteration counts, consistent with a reduced magnitude 
of the most negative local eigenvalues. In general, the method demonstrates robustness not only with respect to frequency \(k\) 
and subdomain count \(N\), but also with respect to coefficient variation.

\subsection{Spectral Complexity of Local Eigenproblems}

Table~\ref{tab:4} reports, for $\tau=0.4$, the number of local eigenvalues per subdomain lying below $\tau$, 
separated into negative and positive contributions and maximised (with $\amax=1000$ in the heterogeneous case).  We observe:
\begin{mybull}
\item
	\textbf{Spectral growth with $k$:} For fixed $N$, both Max.Neg and Max.Pos increase as $k$ grows, reflecting 
	the increasing indefiniteness of the GEVP \eqref{eq: 5_12}.
\item
	\textbf{Effect of subdomain size ($N$):} For fixed $k$, increasing $N$ (i.e.\ reducing subdomain size) 
	significantly decreases both negative and small positive eigenvalues.  For instance, at $k=60$, Max.Neg decreases 
	from $22$ (for $N=16$) to $4$ (for $N=144$).  This agrees with theory: when $H_c $ is small enough, the GEVPs \eqref{eq: 5_12} 
	will become positive definite or only weakly indefinite (cf.\ Lemma~\ref{lemma_3_4}).
\end{mybull}
\vspace{-0.5cm}
\begin{table}[h!]
\centering
\scriptsize
\begin{tabular}{cccc|cc|cc}
\multicolumn{4}{c|}{} & \multicolumn{2}{c}{Homogeneous} & \multicolumn{2}{c}{Heterogeneous} \\
$N$ & $k$ & $h^{-1}$  & Dim(Fine) & Max.Neg. & Max.Pos & Max.Neg. & Max.Pos   \\
\hline
16 &  20 & 240 & 58081 & 4 & 14 & 2 & 16 \\
 &  40 & 480 & 231361 & 13 & 36 & 7 & 35 \\
 &  60 & 720 & 519841 & 22 & 64 & 12 & 61 \\
 &  80 & 960 & 923521 & 41 & 91 & 18 & 90 \\
 &  100 & 1200 & 1442401 & 58 & 132 & 31 & 120 \\
\hline
36 &  20 & 240 & 58081 & 3 & 9 & 2 & 12 \\
 &  40 & 480 & 231361 & 6 & 22 & 3 & 24 \\
 &  60 & 720 & 519841 & 13 & 36 & 7 & 40 \\
 &  80 & 960 & 923521 & 20 & 52 & 9 & 55 \\
 &  100 & 1200 & 1442401 & 28 & 71 & 15 & 75 \\
\hline
64 &  20 & 240 & 58081 & 1 & 7 & 1 & 8 \\
 &  40 & 480 & 231361 & 4 & 14 & 3 & 17 \\
 &  60 & 720 & 519841 & 8 & 26 & 5 & 29 \\
 &  80 & 960 & 923521 & 13 & 36 & 8 & 40 \\
 &  100 & 1200 & 1442401 & 17 & 48 & 11 & 55 \\
\hline
100 &  20 & 240 & 58081 & 1 & 7 & 1 & 7 \\
 &  40 & 480 & 231361 & 3 & 13 & 2 & 17 \\
 &  60 & 720 & 519841 & 4 & 20 & 3 & 26 \\
 &  80 & 960 & 923521 & 8 & 26 & 6 & 38 \\
 &  100 & 1200 & 1442401 & 13 & 36 & 10 & 47 \\
\hline
144 &  20 & 240 & 58081 & 1 & 4 & 1 & 6 \\
 &  40 & 480 & 231361 & 3 & 9 & 2 & 11 \\
 &  60 & 720 & 519841 & 4 & 14 & 3 & 19 \\
 &  80 & 960 & 923521 & 6 & 22 & 6 & 24 \\
 &  100 & 1200 & 1442401 & 8 & 31 & 8 & 34 \\
\hline
\end{tabular}
\caption{For \(\tau=0.4\), maximum number of negative eigenvalues and maximum number of positive eigenvalues 
below the threshold per subdomain, for the homogeneous and heterogeneous cases.}
\label{tab:4}
\end{table}

\subsection{Eigenfunction behaviour and threshold Selection}
A striking feature of Tables \ref{tab:merged_tau} and \ref{tab:het_merged_tau} is that $\lambda_{\min}$ 
always appears bounded below by $-1$, whereas in Theorem~\ref{est_below} we could only prove the lower bound $-\Ceig$ 
with a fairly complicated $\Ceig$.  Some support for the conjecture $\Ceig = 1$ could be found by looking at \eqref{new3}, 
simplified by ignoring the POU:
\begin{equation} \label{cheat} 
	- \nabla\cdot(A \nabla u) - k^2 u = \lambda\bigl(-\nabla\cdot(A \nabla u) + k^2 u\bigr), \quad \text{with Neumann boundary condition}.
\end{equation}
An easy manipulation shows that $-\nabla\cdot( A \nabla u) = ((1+ \lambda)/(1-\lambda))k^2 u$, 
with $\lambda = 1$ not possible since it would imply $u = 0$.  From this it follows that
\begin{equation*} 
	\lambda = \frac{\mu_j - k^2}{\mu_j + k^2} =  -1 + \frac{2 \mu_j}{\mu_j + k^2 } \ge  -1, 
\end{equation*}
where $\mu_j \ge 0$ is the $j$th Neumann eigenvalue of the  operator $- \nabla\cdot(A \nabla u)$.
Interestingly, this also suggests that, for $\lambda$ near $-1$, the eigenfunctions of \eqref{cheat} 
are (non-oscillatory) Neumann eigenfunctions corresponding to small eigenvalues of $ - \nabla\cdot(A \nabla u)$, while
$\lambda \approx 0$ corresponds to $\mu_j \approx k^2$ with corresponding eigenfunctions oscillating
(roughly) with period $\sim k^{-1}$.
Confirmation of this for the homogeneous case is given in Figure~\ref{fig:EigRangeHomog},
which displays eigenfunctions of problem \eqref{eq: 5_12} (with the POU)  on a floating subdomain
(i.e.\ one that does not intersect the global boundary). 
We can clearly see that as $\lambda$ increases from about $ -0.94$ to $6.5 \times 10^{-5}$, the eigenfunctions become increasingly oscillatory.   
        
\begin{figure}[h!]
\centering
\setlength{\tabcolsep}{2pt}
\renewcommand{\arraystretch}{1.0}

\begin{tabular}{ccc}
\includegraphics[width=0.30\textwidth]{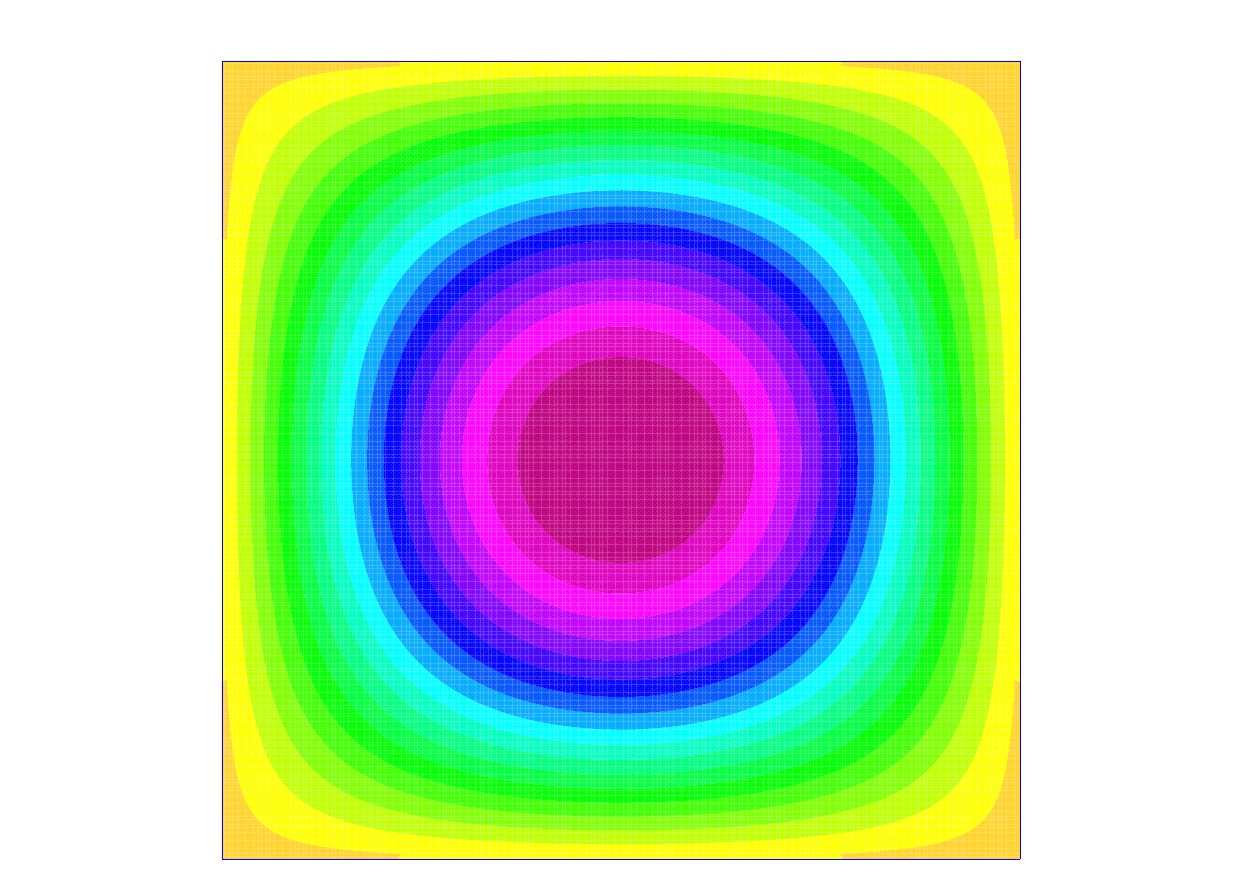} &
\includegraphics[width=0.30\textwidth]{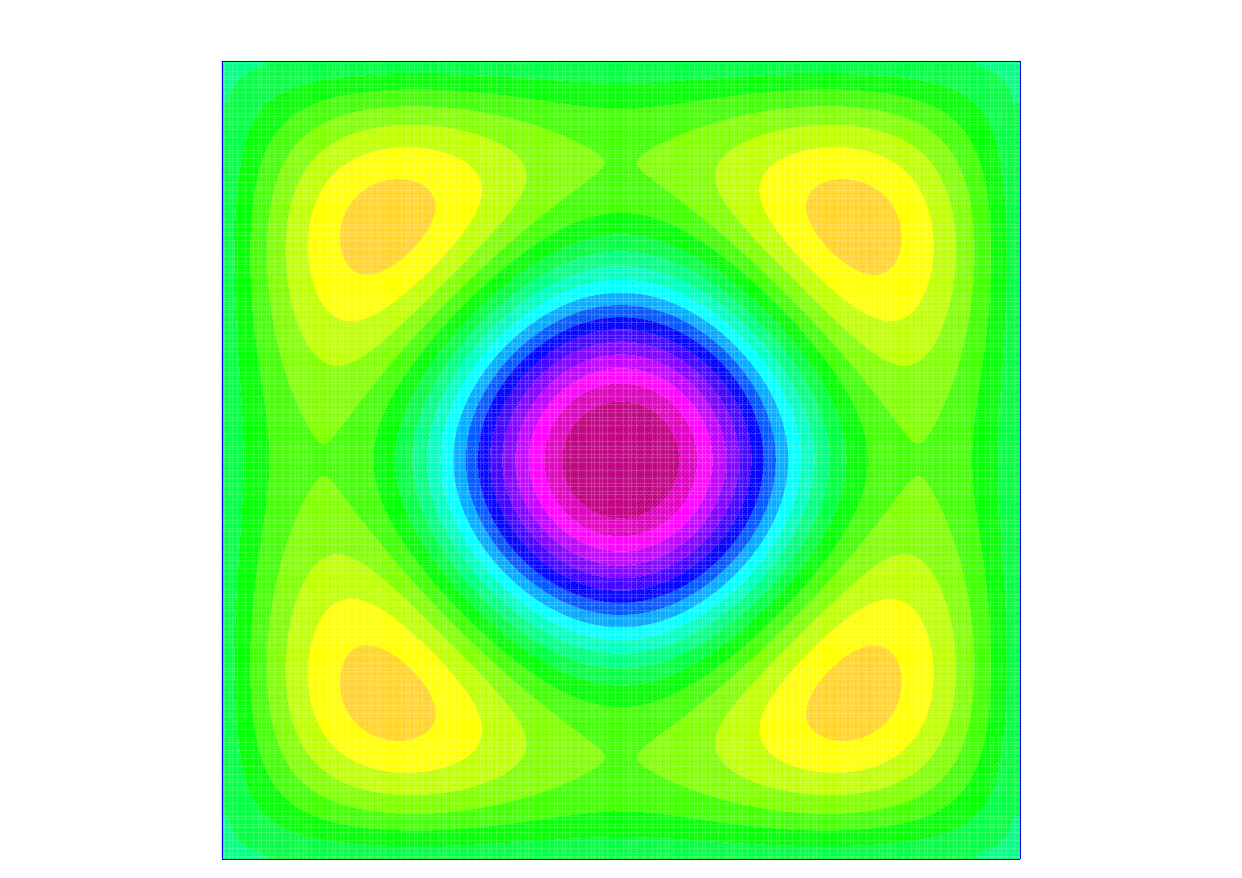} &
\includegraphics[width=0.30\textwidth]{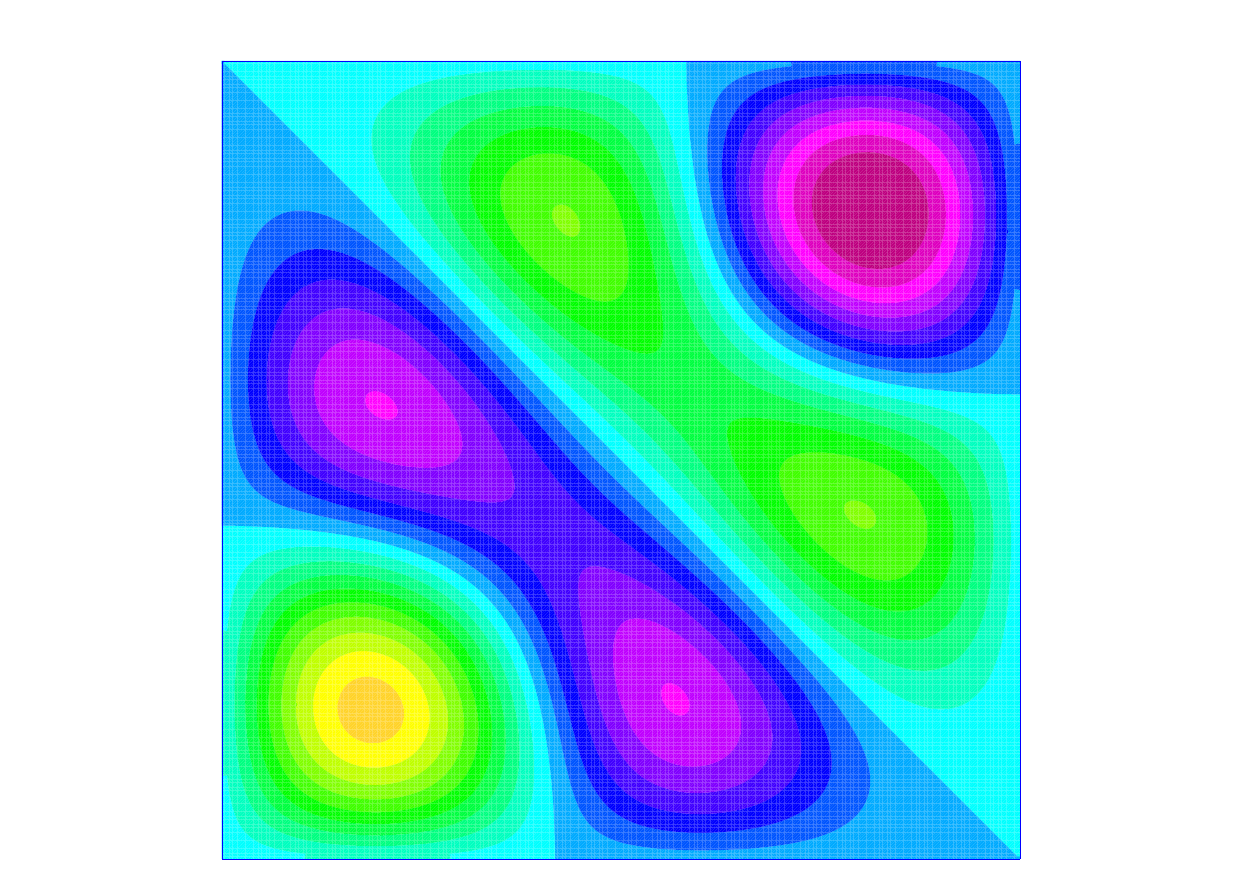} \\
{\scriptsize $\lambda=-0.941228$} &
{\scriptsize $\lambda=-0.738817$} &
{\scriptsize $\lambda=-0.674188$} \\[3pt]

\includegraphics[width=0.30\textwidth]{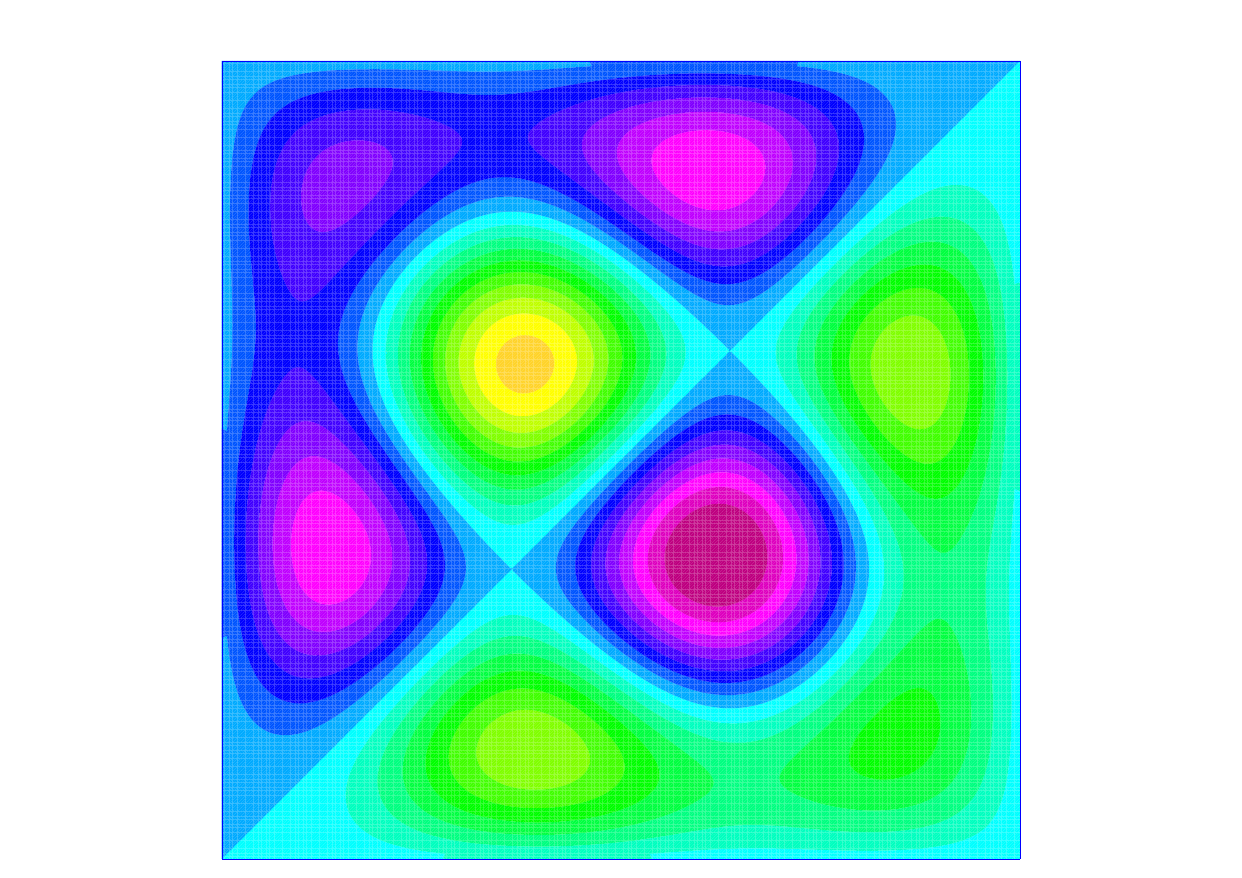} &
\includegraphics[width=0.30\textwidth]{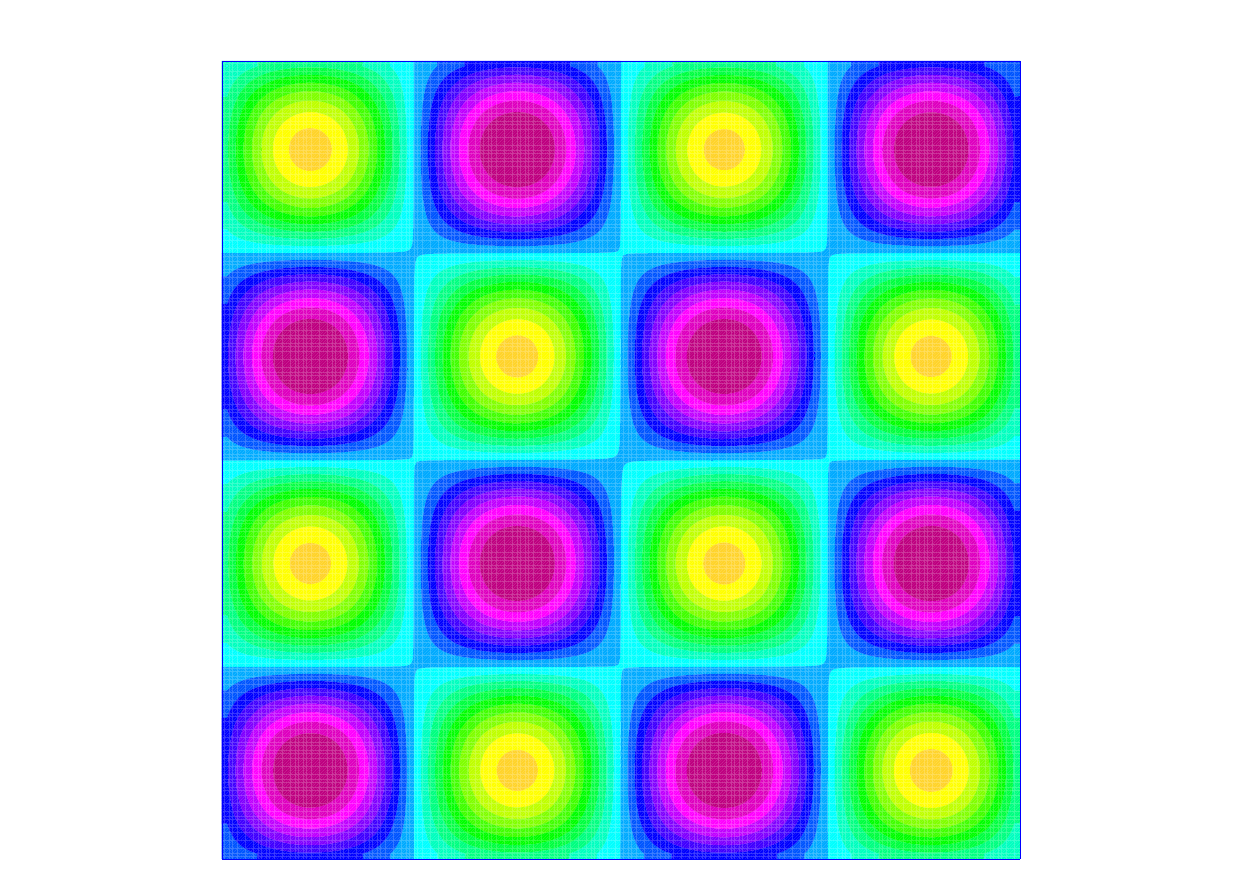} &
\includegraphics[width=0.30\textwidth]{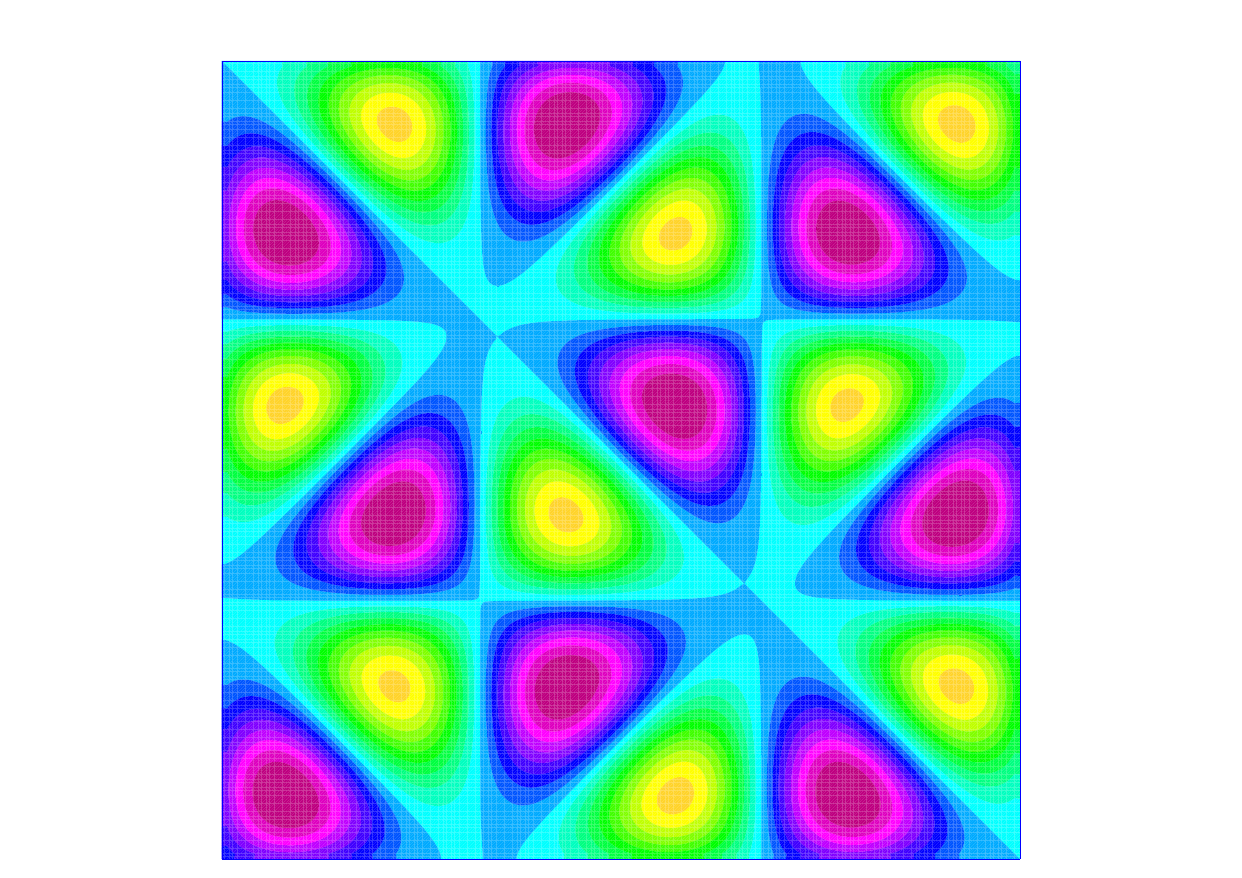} \\
{\scriptsize $\lambda=-0.595937$} &
{\scriptsize $\lambda=-0.36627$} &
{\scriptsize $\lambda=-0.227986$} \\[3pt]

  \includegraphics[width=0.30\textwidth]{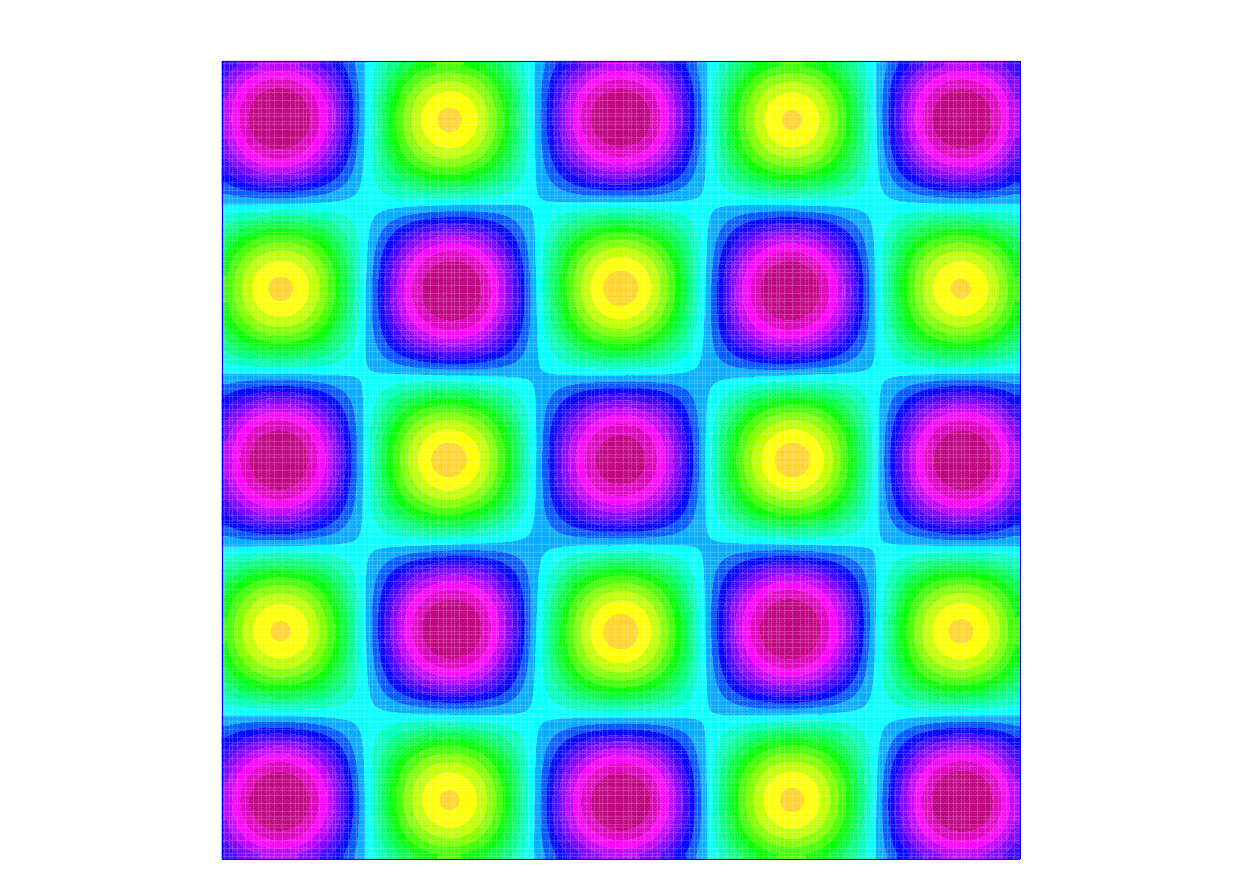} &
  \includegraphics[width=0.30\textwidth]{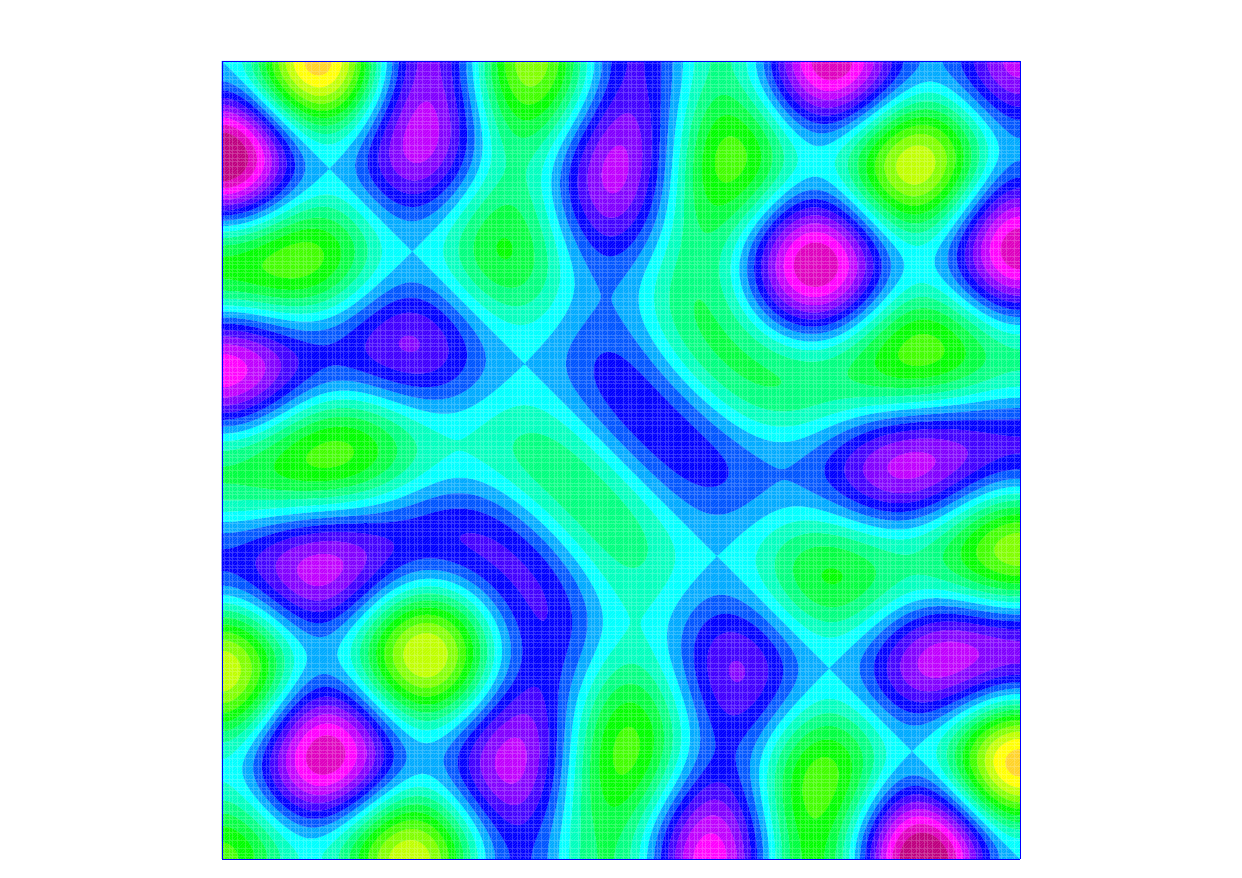} &
  \includegraphics[width=0.30\textwidth]{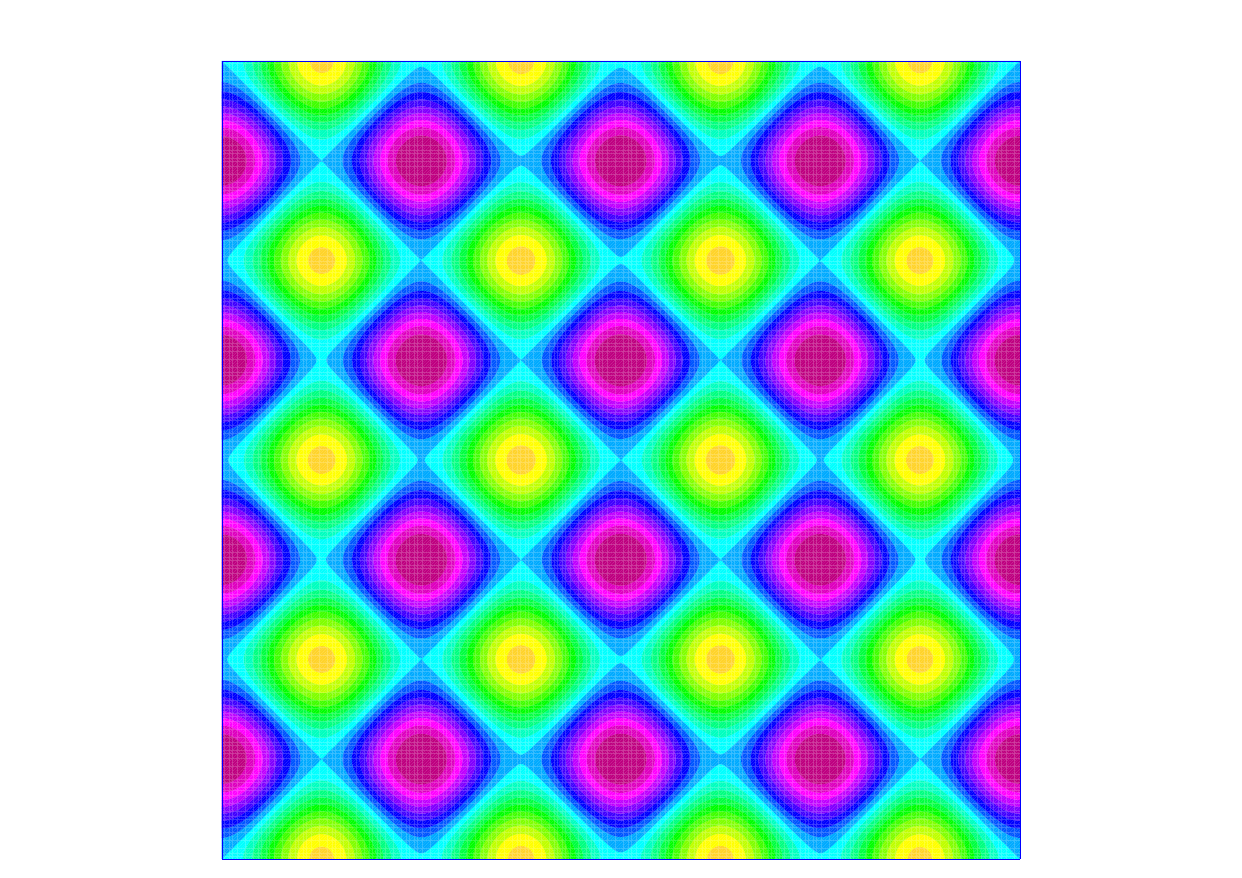} \\
{\scriptsize $\lambda=-0.187746$} &
{\scriptsize $\lambda=0.00348611$} &
{\scriptsize $\lambda=-0.000630012$}\\
\end{tabular}

\caption{Eigenvector plots for $k=100$ and $h^{-1}=1200$ (homogeneous case). Eigenfunctions are normalised in $\ell_1$ norm. Red/blue colours correspond to positive values and yellow/green to negative ones}
\label{fig:EigRangeHomog}
\end{figure}

The plots in Figure~\ref{fig:EigRangeHomog} for $\lambda$ near zero are reminiscent of previous work 
on plane wave based coarse spaces for Helmholtz problems (e.g.\ \cite{KiSa:07} and the references therein). 
Given this, it seems reasonable to expect that the `symmetric strategy' of retaining only $\lambda \in [-\tau, \tau]$ 
rather than the `upper bound' strategy  $\lambda \le \tau$ may reduce coarse space size without risking a loss of $k$-robustness.
In Table~\ref{tab:+-0.6}, we compare the symmetric and upper bound strategies when  
applied to the homogeneous problem for $\tau = 0.6$.  We see that in this case the symmetric strategy works well 
with little degradation of iteration count and a modest reduction of \text{dim(CS)}. 
However, we found the symmetric strategy worked less well when applied with $\tau = 0.2$. 
So, while the symmetric strategy can be useful in practice, the choice of $\tau$ has to be chosen 
after some  experimentation.
\begin{table}[h!]
    \centering
    \scriptsize
    \begin{tabular}{cccc|cccc|cccc}
 & & & &\multicolumn{4}{c|}{$\lambda \in [-0.6,0.6]$} & \multicolumn{4}{c}{$\lambda \leq 0.6$} \\
 &  &  & Dim & Max.Num. &  & Dim & It. & Max.Num. &  & Dim & It.  \\
$N$ & $k$ & $h^{-1}$ & (Fine) & Neg. & $\lambda_\text{min}$ & (CS) & Count & Neg. & $\lambda_\text{min}$ & (CS) & Count \\
\hline
16 & 20 & 240 & 58081 & 4 & -0.324293 & 392 & 14 & 4 & -0.324293 & 392 & 14 \\
 & 60 & 720 & 519841 & 19 & -0.52068 & 1936 & 15 & 22 & -0.849475 & 1984 & 15 \\
 & 100 & 1200 & 1442401 & 50 & -0.595937 & 4444 & 30 & 58 & -0.941228 & 4572 & 29 \\
\hline
36 & 20 & 240 & 58081 & 3 & -0.178013 & 600 & 15 & 3 & -0.178013 & 600 & 15 \\
 & 60 & 720 & 519841 & 12 & -0.419254 & 2540 & 15 & 13 & -0.703879 & 2576 & 15 \\
 & 100 & 1200 & 1442401 & 25 & -0.587053 & 5476 & 54 & 28 & -0.875285 & 5584 & 56 \\
\hline
64 & 20 & 240 & 58081 & 1 & -0.125924 & 740 & 15 & 1 & -0.125924 & 740 & 15 \\
 & 60 & 720 & 519841 & 8 & -0.553713 & 3232 & 15 & 8 & -0.553713 & 3232 & 15 \\
 & 100 & 1200 & 1442401 & 16 & -0.562646 & 6444 & 15 & 17 & -0.793739 & 6508 & 15 \\
\hline
100 & 20 & 240 & 58081 & 1 & -0.100482 & 1044 & 17 & 1 & -0.100482 & 1044 & 17 \\
 & 60 & 720 & 519841 & 4 & -0.423257 & 3748 & 16 & 4 & -0.423257 & 3748 & 16 \\
 & 100 & 1200 & 1442401 & 12 & -0.419254 & 7436 & 16 & 13 & -0.703879 & 7536 & 15 \\
\hline
144 & 20 & 240 & 58081 & 1 & -0.0852859 & 1096 & 17 & 1 & -0.0852859 & 1096 & 17 \\
 & 60 & 720 & 519841 & 4 & -0.324293 & 4232 & 16 & 4 & -0.324293 & 4232 & 16 \\
 & 100 & 1200 & 1442401 & 7 & -0.599849 & 8400 & 16 & 8 & -0.612531 & 8500 & 16 \\
\hline
\end{tabular}
\caption{Homogeneous problem: symmetric and upper bound strategies}
    \label{tab:+-0.6}
\end{table}

\subsection{Decoupled decomposition}
\label{subsec:Decoupled}

We now investigate the effect of constructing the coarse space on a decomposition that differs from that used for the one-level method. 
Specifically, the one-level preconditioner is built on $Q=144$ subdomains, while the coarse space is constructed 
using a separate decomposition with $N \in \{16,25,36,64,100,144\}$ subdomains. 
The generalised eigenproblems are solved on these coarse patches. 
In this part we will limit our investigation to the homogeneous case.
Table~\ref{tab:SepDecompTau0.6} reports the coarse space dimension (CS) and GMRES iteration counts for $\tau=0.6$. 
Several observations can be made:

\begin{table}[h]
    \centering
    \scriptsize
\begin{tabular}{ccc|cc|cc|cc|cc|cc|cc}
\multicolumn{3}{c|}{} & \multicolumn{2}{c|}{$N = 144$} & \multicolumn{2}{c|}{$N = 100$} & \multicolumn{2}{c|}{$N = 64$} & \multicolumn{2}{c|}{$N = 36$} & \multicolumn{2}{c|}{$N = 25$} & \multicolumn{2}{c}{$N_c = 16$} \\
$k$ & $h^{-1}$ & Dim(Fine) & CS & It. & CS & It. & CS & It. & CS & It. & CS & It. & CS & It. \\
\hline
20 & 240 & 58081   & 1096 & 17 & 1044 & 20 & 740  & 24 & 600  & 19 & 530  & 18 & 392  & 20 \\
40 & 480 & 231361  & 2640 & 20 & 2360 & 23 & 1800 & 24 & 1404 & 25 & 1230 & 21 & 1068 & 23 \\
60 & 720 & 519841  & 4232 & 16 & 3748 & 23 & 3232 & 20 & 2576 & 20 & 2285 & 17 & 1984 & 20 \\
80 & 960 & 923521  & 6048 & 33 & 5360 & 52 & 4732 & 45 & 3908 & 42 & 3527 & 44 & 3152 & 41 \\
100 & 1200 & 1442401 & 8500 & 16 & 7536 & 19 & 6508 & 20 & 5584 & 19 & 5065 & 18 & 4572 & 19 \\
\hline
Growth & -- & -- &
$\mathcal O(k^{1.2})$ & &
$\mathcal O(k^{1.2})$ & &
$\mathcal O(k^{1.4})$ & &
$\mathcal O(k^{1.4})$ & &
$\mathcal O(k^{1.4})$ & &
$\mathcal O(k^{1.5})$ &
\\
\end{tabular}
    \caption{Coarse space size (CS) and iteration count (It.) for a one-level decomposition
    of $Q=144$ subdomains and varying coarse-level decompositions $N$ with
    threshold parameter $\tau=0.6$ in the homogeneous domain. The final row
    reports the empirical growth of the coarse space size with respect to the
    wavenumber $k$, obtained from a least-squares fit of $\log(\mathrm{CS})$
    against $\log(k)$. In all cases, the coarse space size grows substantially
    more slowly than the fine-space dimension, which exhibits
    $\mathcal{O}(k^2)$ growth.}
    \label{tab:SepDecompTau0.6}
\end{table}

\begin{mybull}
\item
	For most values of $k$, the iteration counts remain comparable across all coarse decompositions, 
	even when the coarse space dimension is reduced by $40$--$60\%$.  For the coarsest decompositions, 
	the increase in iterations is modest (typically $1$--$4$ iterations), indicating that substantial memory savings 
	can be achieved with minimal loss of robustness.
\item
	The case $k=80$ exhibits unusually high iteration counts across all decompositions, including $N=144$. 
	This behaviour is more pronounced when smaller thresholds are used and is likely linked to the proximity of $k$ 
	to a Dirichlet eigenvalue of the Laplacian, amplifying indefiniteness effects.
\end{mybull}

Overall, in the homogeneous setting, decoupling the decompositions provides an effective trade-off: 
reducing the coarse space size leads only to a mild increase in iteration counts.
When memory constraints are relevant, coarser coarse-level decompositions are therefore attractive.

\section{Conclusion}

We have developed and analysed a two-level domain decomposition preconditioner for indefinite Helmholtz-type problems, 
in which the coarse space ($H_k$-GenEO) is constructed directly from local generalised eigenvalue problems 
involving the full indefinite operator.  This contrasts with $\Delta$-GenEO and $\Delta_k$-GenEO approaches, which rely on nearby SPD formulations.

A central contribution of this work is the theoretical framework for the underlying indefinite eigenproblem and the 
resulting $k$-explicit sufficient conditions for robustness. Practically, this requires the fine subdomain diameter 
to scale with the wavelength and the spectral threshold $\tau$ to retain all negative and sufficiently small positive local modes. 
Notably, the coarse diameter $H_c$ does not enter the robustness conditions, allowing decoupling of the one-level and coarse decompositions.

The numerical experiments on homogeneous and layered heterogeneous media confirm the theoretical findings. 
Once $\tau$ is chosen sufficiently large, iteration counts remain stable with respect to both frequency $k$ and subdomain count $N$. 
For $\tau=0.4$, GMRES typically converges in approximately $14$--$25$ iterations over broad parameter ranges, 
while $\tau=0.6$ further flattens iteration counts to roughly $11$--$16$ at the expense of a larger coarse space. 
In contrast, overly aggressive truncation (e.g.\ $\tau=0.2$) reduces coarse space size but compromises robustness at higher frequencies. 

Two-sided thresholds offer limited benefit for moderate-to-large $\tau$, as excluding mildly negative modes has 
negligible impact on convergence.  However, for small $\tau$, imposing a negative lower bound can remove oscillatory unstable modes 
and significantly degrade performance.

Finally, decoupled decompositions prove effective in homogeneous settings, enabling substantial coarse-space reduction 
with only minor iteration increases. In heterogeneous media, however, robustness depends on aligning the coarse partition 
with the coefficient structure, highlighting the importance of geometric compatibility between decomposition and heterogeneity.

Overall, by constructing the coarse space from the indefinite operator itself and establishing explicit robustness criteria, 
$H_k$-GenEO provides a practical and scalable preconditioner whose convergence is stable with respect to frequency, 
heterogeneity, and subdomain count.

\paragraph{Acknowledgement} 
IGG thanks Alastair Spence for valuable discussions on the material in \S \ref{subsec:abstract}.
MF was funded by a studentship of the Engineering and Physical Sciences Research Council.

\end{document}